\documentclass[12pt]{amsart}
\usepackage{amssymb}
\usepackage[mathcal]{eucal}
\usepackage{lmodern}
\usepackage[T1]{fontenc}
\usepackage{hyperref}
\usepackage{longtable,array}
\usepackage{enumitem}
\usepackage{tikz-cd}
\usepackage[nobysame]{amsrefs}
\usepackage[total={6.25truein,9truein},centering]{geometry}

\allowdisplaybreaks

\numberwithin{equation}{subsection}

\newtheorem{theorem}[equation]{Theorem}
\newtheorem{lemma}[equation]{Lemma}
\newtheorem{proposition}[equation]{Proposition}
\newtheorem{corollary}[equation]{Corollary}

\newtheorem{theoremintro}{Theorem}

\newtheorem{corollaryintro}[theoremintro]{Corollary}

\theoremstyle{definition}

\theoremstyle{remark}
\newtheorem{example}[equation]{Example}
\newtheorem{remark}[equation]{Remark}

\newtheorem{criteria}[equation]{Criteria}
\newtheorem*{acknowledgments}{Acknowledgments}

\newenvironment{subsubsec}[1]{\vspace{\topsep}
\noindent\refstepcounter{equation}\theequation.\ \textit{#1.}}{\vspace{\topsep}}

\newenvironment{alphenumerate}{\begin{enumerate}[label=\textup{(\alph*)}]}{\end{enumerate}}

\newcommand{\AAA}{\mathbb{A}}
\newcommand{\BB}{\mathbb{B}}
\newcommand{\FF}{\mathbb{F}}
\newcommand{\ZZ}{\mathbb{Z}}
\newcommand{\QQ}{\mathbb{Q}}

\newcommand{\CC}{\mathbb{C}}
\newcommand{\PP}{\mathbb{P}}

\newcommand{\TT}{\mathbb{T}}
\newcommand{\GG}{\mathbb{G}}

\newcommand{\KK}{\mathbb{K}}

\newcommand{\ba}{\mathbf{a}}
\newcommand{\bA}{\mathbf{A}}

\newcommand{\bB}{\mathbf{B}}

\newcommand{\be}{\mathbf{e}}

\newcommand{\bg}{\mathbf{g}}
\newcommand{\bh}{\mathbf{h}}

\newcommand{\bk}{\mathbf{k}}
\newcommand{\bK}{\mathbf{K}}

\newcommand{\bn}{\mathbf{n}}

\newcommand{\br}{\mathbf{r}}
\newcommand{\bs}{\mathbf{s}}

\newcommand{\bU}{\mathbf{U}}

\newcommand{\bx}{\mathbf{x}}
\newcommand{\bX}{\mathbf{X}}

\newcommand{\bz}{\mathbf{z}}

\newcommand{\cA}{\mathcal{A}}

\newcommand{\cC}{\mathcal{C}}
\newcommand{\cD}{\mathcal{D}}
\newcommand{\cE}{\mathcal{E}}

\newcommand{\cG}{\mathcal{G}}
\newcommand{\cI}{\mathcal{I}}
\newcommand{\cL}{\mathcal{L}}
\newcommand{\cM}{\mathcal{M}}
\newcommand{\cN}{\mathcal{N}}
\newcommand{\cO}{\mathcal{O}}
\newcommand{\cP}{\mathcal{P}}

\newcommand{\cS}{\mathcal{S}}

\newcommand{\cU}{\mathcal{U}}
\newcommand{\cV}{\mathcal{V}}
\newcommand{\cX}{\mathcal{X}}

\newcommand{\fb}{\mathfrak{b}}
\newcommand{\ff}{\mathfrak{f}}
\newcommand{\fp}{\mathfrak{p}}

\newcommand{\fq}{\mathfrak{q}}

\newcommand{\rd}{\mathrm{d}}
\newcommand{\rD}{\mathrm{D}}
\newcommand{\rE}{\mathrm{E}}

\newcommand{\rh}{\mathrm{h}}
\newcommand{\rH}{\mathrm{H}}

\newcommand{\rI}{\mathrm{I}}

\newcommand{\rR}{\mathrm{R}}
\newcommand{\rU}{\mathrm{U}}

\newcommand{\bsalpha}{\boldsymbol{\alpha}}

\newcommand{\bspi}{\boldsymbol{\pi}}

\DeclareMathOperator{\diag}{diag}
\DeclareMathOperator{\divf}{div}
\DeclareMathOperator{\Div}{Div}
\DeclareMathOperator{\Emb}{Emb}
\DeclareMathOperator{\Exp}{Exp}

\DeclareMathOperator{\GL}{GL}
\DeclareMathOperator{\Lie}{Lie}
\DeclareMathOperator{\Mat}{Mat}
\DeclareMathOperator{\mult}{mult}
\DeclareMathOperator{\Char}{Char}
\DeclareMathOperator{\End}{End}
\DeclareMathOperator{\res}{res}
\DeclareMathOperator{\Span}{Span}
\DeclareMathOperator{\Spec}{Spec}
\DeclareMathOperator{\Gal}{Gal}

\DeclareMathOperator{\rank}{rank}
\DeclareMathOperator{\Res}{Res}
\DeclareMathOperator{\sgn}{sgn}
\DeclareMathOperator{\ord}{ord}

\DeclareMathOperator{\im}{im}

\DeclareMathOperator{\wt}{wt}

\newcommand{\oD}{\mkern2.5mu\overline{\mkern-2.5mu D}}
\newcommand{\oE}{\mkern2.5mu\overline{\mkern-2.5mu E}}
\newcommand{\oF}{\mkern2.5mu\overline{\mkern-2.5mu F}}
\newcommand{\og}{\overline{g}}
\newcommand{\oh}{\overline{h}}
\newcommand{\ok}{\overline{k}}
\newcommand{\oK}{\mkern2.5mu\overline{\mkern-2.5mu K}}

\newcommand{\oM}{\mkern2.5mu\overline{\mkern-2.5mu M}}
\newcommand{\on}{\overline{n}}
\newcommand{\oalpha}{\overline{\alpha}}

\newcommand{\oGamma}{\overline{\Gamma}}

\newcommand{\otheta}{\mkern2.5mu\overline{\mkern-2.5mu \theta}}
\newcommand{\oPhi}{\overline{\Phi}}

\newcommand{\oP}{\mkern2.5mu\overline{\mkern-2.5mu P}}
\newcommand{\oQ}{\overline{Q}}
\newcommand{\oR}{\mkern2.5mu\overline{\mkern-2.5mu R}}
\newcommand{\oS}{\mkern2.5mu\overline{\mkern-2.5mu S}}
\newcommand{\oT}{\overline{T}}

\newcommand{\oW}{\overline{W}}
\newcommand{\ox}{\overline{x}}

\newcommand{\oZ}{\mkern2.5mu\overline{\mkern-2.5mu Z}}
\newcommand{\oxi}{\overline{\xi}}
\newcommand{\oXi}{\overline{\Xi}}

\newcommand{\obn}{\overline{\bn}}

\newcommand{\perf}{\mathrm{perf}}

\newcommand{\Reg}{\mathrm{Reg}}
\newcommand{\sep}{\mathrm{sep}}

\newcommand{\teps}{\widetilde{\varepsilon}}

\newcommand{\tcE}{\widetilde{\cE}}
\newcommand{\tbe}{\widetilde{\be}}
\newcommand{\tg}{\widetilde{g}}
\newcommand{\tilh}{\widetilde{h}}

\newcommand{\tiota}{\widetilde{\iota}}

\newcommand{\tpi}{\widetilde{\pi}}

\newcommand{\trho}{\widetilde{\rho}}
\newcommand{\ttheta}{\widetilde{\theta}}

\newcommand{\tx}{\widetilde{x}}

\newcommand{\C}{\CC_{\infty}}

\newcommand{\oFF}{\overline{\FF}}

\newcommand{\Qbar}{\overline{\QQ}}

\newcommand{\iso}{\stackrel{\sim}{\longrightarrow}}
\newcommand{\mayeq}{\stackrel{?}{=}}
\newcommand{\power}[2]{{#1 [[ #2 ]]}}
\newcommand{\laurent}[2]{{#1 (( #2 ))}}

\newcommand{\dnorm}[1]{\lVert #1 \rVert}

\newcommand{\inorm}[1]{{\lvert #1 \rvert}_{\infty}}

\newcommand{\diam}[1]{\langle #1 \rangle}

\newcommand{\pd}{\partial}

\newcommand{\tr}{{\mathsf{T}}}
\newcommand{\assign}{\mathrel{\vcenter{\baselineskip0.5ex \lineskiplimit0pt
                     \hbox{\scriptsize.}\hbox{\scriptsize.}}}%
                     =}
\newcommand{\rassign}{=%
                     \mathrel{\vcenter{\baselineskip0.5ex \lineskiplimit0pt
                     \hbox{\scriptsize.}\hbox{\scriptsize.}}}%
                     }

\begin{document}

\title{Hecke $L$-series for Sinha modules}

\author{Erik Davis}
\address{Department of Mathematics, Texas A{\&}M University, College Station, TX 77843, U.S.A.}
\email[E. Davis]{davis7e@tamu.edu}

\author{Matthew Papanikolas}
\email[M. Papanikolas]{papanikolas@tamu.edu}

\subjclass{Primary 11M38; Secondary 11G09, 11G15, 11R60}

\date{July 5, 2025}

\begin{abstract}
We investigate Goss $L$-functions associated to Anderson $t$-modules defined by Sinha having complex multiplication by Carlitz cyclotomic fields. We show that these $t$-modules are defined over the cyclotomic field and that their $L$-functions are products of Hecke $L$-series for Anderson's Hecke character defined via Coleman functions. Applying identities of Fang and Taelman, we prove that special values of these $L$-functions are expressible in terms of products of values of Thakur's geometric $\Gamma$-function.
\end{abstract}

\keywords{Goss $L$-series, Hecke $L$-series, Coleman functions, Anderson Hecke characters, Sinha modules, Taelman class formulas}

\maketitle

\tableofcontents

\section{Introduction} \label{S:Intro}

\subsection{Classical Hecke \texorpdfstring{$L$}{L}-series}
We investigate special values of Hecke $L$-series attached to Carlitz cyclotomic fields. The motivating problem comes from $L$-functions of Jacobi sum Hecke characters~\cites{And86Taniyama, Lichtenbaum82, Schappacher}, where one defines a Hecke character $J_{K,\ba}$ on an abelian number field $K/\QQ$ depending on a parameter~$\ba$ in the free abelian group on $\QQ/\ZZ$.

Lichtenbaum~\cite{Lichtenbaum82}*{Thm.~3.1} proved that critical values $L(J_{K,\ba},n)$ are $\Qbar$-multiples of prescribed powers of $\pi$ and powers of special $\Gamma$-values $\Gamma(a/N)$, where $a/N \in \QQ$ depend on~$\ba$. He conjectured that these multiples were in fact in $\QQ$ up to a square root of the discriminant of $K$ or its maximal real subfield. A few years later Anderson~\cite{And86Taniyama}*{Thm.~2} proved Lichtenbaum's $\Gamma$-hypothesis using his theory of ulterior motives. For detailed accounts of $L$-functions of Jacobi sum Hecke characters, the reader is directed to~\cites{And86Taniyama, Lichtenbaum82, Schappacher}.

For example, if $\ba = 2[1/3] - [2/3]$ and $K = \QQ(\sqrt{-3})$, then the character $J_{K,\ba}$ is related to the elliptic curve $y^2 +y= x^3$ with complex multiplication by $O_K$. One finds
\[
L(J_{K,\ba},1) = \frac{\Gamma(1/3)^2}{9 \Gamma(2/3)}.
\]

\subsection{Anderson Hecke characters}
Turning to the setting of function fields, we let $A = \FF_q[\theta]$ be the polynomial ring in a variable~$\theta$ over the finite field $\FF_q$, and let $A_+$ denote its monic elements. We take $k = \FF_q(\theta)$ for the fraction field of $A$, $k_{\infty} = \laurent{\FF_q}{1/\theta}$ for the completion of $k$ at its infinite place, and $\C$ for the completion of an algebraic closure of~$k_{\infty}$. We let $\ok$ denote the algebraic closure of $k$ in $\C$. For $f \in A_+$ with $\deg f \geqslant 1$, we let $K = k(\zeta)$ denote the $f$-th Carlitz cyclotomic field (see \S\ref{SS:cyclodefs}) and $B \subseteq K$ the integral closure of~$A$ in~$K$. We let $b \mapsto \rho_b$ be the isomorphism $(A/fA)^{\times} \iso \Gal(K/k)$ such that $\rho_b(\zeta) = C_b(\zeta)$, where $C$ is the Carlitz module.

\begin{subsubsec}{Coleman functions}
Inspired by work of Coleman~\cite{Coleman88}, Anderson~\cite{And92} introduced soliton functions, which are functions on the product $X \times X$, where $X/\FF_q$ is a smooth projective curve for the function field $K$. Later, researchers replaced solitons in applications with more direct specializations called Coleman functions $g_{\ba} \in \ok(\bX)$ ($\bX = \ok \times_{\FF_q} X$) in~\cites{ABP04, SinhaPhD, Sinha97a, Sinha97b}. These functions possess explicit divisors
\[
\divf(g_{\ba}) = W_{\ba}^{(1)} - W_{\ba} + \Xi_{\ba} - I_{\ba}
\]
(see \S\ref{SS:Coleman}), which we use to define both Anderson's Hecke characters and Sinha's $t$-modules. Coleman functions play a similar role to shtuka functions for Drinfeld-Hayes modules as in~\cite{Thakur93}, and they underpin many results (e.g., see \cites{BP02, CPY10, ChangTW25, Thakur99, Wei22}).
\end{subsubsec}

\begin{subsubsec}{Hecke characters}
Using Jacobi sum Hecke characters as a model, Anderson~\cite{And92} defined a class of Hecke characters, later expanded by Sinha~\cite{Sinha97a},
\begin{equation}
\chi_{\ba} : \cI_{K,\ff} \to K^{\times}
\end{equation}
where $\ba$ is an element of the free abelian group $\cA_f$ on $f^{-1}A/A$, $\ff \subseteq B$ is the radical of~$fB$, and $\cI_{K,\ff}$ is the group of fractional ideals of $B$ prime to $\ff$. For $\fp$ a maximal ideal of $B$ relatively prime to~$f$, Anderson defines $\chi_{\ba}(\fp)$ by
\[
\chi_{\ba}(\fp) = \prod_{i=0}^{\ell - 1} g_{\ba}^{(i)}(\xi) \quad  \in K^{\times},
\]
where $\ell = [\FF_{\fp}:\FF_q]$ and $\xi \in \bX(K)$ is the generic point corresponding to the inclusion $K \hookrightarrow \C$. See \S\ref{SS:AndHecke}. Anderson proved that $\chi_{\ba}$ is an algebraic Hecke character for $K$ of conductor dividing~$\ff$ and designated infinity type $\Theta_{\ba} \in \ZZ[\Gal(K/k)]$ (see Theorem~\ref{T:AndHecke}). It aligns with work of Goss~\cite{Goss92a} and Hayes~\cite{Hayes93} on Hecke characters in function fields.
\end{subsubsec}

\begin{subsubsec}{Goss $L$-series}
In Theorem~\ref{T:GossLMainintro} and Corollary~\ref{C:HeckeLMainintro} below we prove special $L$-value formulas for Goss $L$-series attached to $\chi_{\ba}$,
\[
L(\chi_{\ba},s) = \prod_{\fp \nmid \ff} \biggl( 1 - \frac{\chi_{\ba}(\fp)}{\cN(\fp)^s} \biggr)^{-1}, \quad s \in \ZZ,
\]
which is a product over primes of $B$ not dividing $\ff$ and which takes values in $k_{\infty}(\zeta) \subseteq \C$ (see \S\ref{SSS:HeckeL}, \S\ref{E:AndHeckeLdef}). Previously, Sinha~\cite{Sinha97a} had proved a function field version of Deligne's conjecture on special $\Gamma$-values that showed that, when $\Gamma(\ba) \in \oK$, the Galois action on $\Gamma(\ba)$ transforms via $\chi_{\ba}$, namely that $\rho_{\fp}(\Gamma(\ba)) = \chi_{\ba}(\fp) \Gamma(\ba)$ for a Frobenius element $\rho_{\fp} \in \Gal(K^{\sep}/K)$. To our knowledge there has been little work on special values of $L(\chi_{\ba},s)$ to date.
\end{subsubsec}

\begin{subsubsec}{$\Pi$-factorial and special $\Pi$-values} \label{SSS:GammaPi}
In \cite{Thakur91}, Thakur defined the geometric factorial function $\Pi : \C \setminus(-A_+) \rightarrow \C$ via
\begin{equation} \label{E:Pi(x)def}
\Pi(x) \assign \prod_{a \in A_+} \biggl( 1 + \frac{x}{a} \biggr)^{-1}.
\end{equation}
It is related to the geometric $\Gamma$-function by $\Pi(x) = x\Gamma(x)$, and as such it is an analogue of the classical factorial. Thakur~\cite{Thakur91} proved several classes of functional equations for $\Gamma(x)$, which established it as an analogue of the classical $\Gamma$-function.

If $x \in k \setminus (-A_+)$, then we refer to $\Pi(x)$ as a \emph{special $\Pi$-value}. We note that special $\Pi$-values fall in~$k_{\infty}$, which will play an important role in calculating Taelman regulators (see Proposition~\ref{P:RegCalc}). Sinha~\cite{Sinha97b} demonstrated that certain special $\Pi$-values occur as periods of $t$-modules $E_{\ba}$ (see below, and  see Theorem~\ref{T:periodsintro}). Algebraic relations over $\ok$ on special $\Pi$-values have been studied extensively \cites{ABP04,BP02,CPY10,Sinha97b,Wei22}, by applying transcendence techniques in~\cites{ABP04, Yu97}, and in the present paper we show they arise naturally in values of $L$-functions (see Theorem~\ref{T:GossLMainintro} and Corollary~\ref{C:HeckeLMainintro}).
\end{subsubsec}

\subsection{Sinha modules}
Just as $L$-functions of Jacobi sum Hecke characters can be interpreted as factors in $L$-functions of abelian varieties, or more generally of motives (see \cites{And86Taniyama, Lichtenbaum82, Schappacher}), we begin our study of $L(\chi_{\ba},s)$ by analyzing more closely a family of Anderson $t$-modules with complex multiplication that were originally defined by Sinha~\cites{SinhaPhD, Sinha97b}. For an effective parameter $\ba \in \cA_f$ (all coefficients are non-negative) and $\deg \ba > 0$, Sinha used Anderson's theory of $t$-motives~\cite{And86} to define an Anderson $t$-module
\begin{equation}
E_{\ba} : \bA \to \Mat_d(\ok[\tau]),
\end{equation}
where $\bA = \FF_q[t]$ and $\tau$ is the $q$-th power Frobenius endomorphism. See \S\ref{SS:tmoddualtmot} for more information on $t$-modules.

The $t$-module $E_{\ba}$ possesses complex multiplication by $\bB \cong B$ with CM-type $\Xi_{\ba}$, as defined by Brownawell, Chang, Wei, and the second author in~\cites{BP02,BCPW22} (see Remark~\ref{R:EaCM}). It is coabelian ($\bA$-finite), has rank $r = [K:k] = \#(A/fA)^{\times}$, and has dimension $d = \deg(\ba)\cdot r/(q-1)$ (see Lemma~\ref{L:Haprops}). Sinha's constructions yield a $t$-module defined over~$\ok$, but he did not produce a field of definition. In order to analyze Goss $L$-functions for~$E_{\ba}$ our first main result is to clarify this situation (stated later as Theorem~\ref{T:EaKdef}).

\begin{theoremintro} \label{T:EaKdefintro}
Let $\ba \in \cA_f$ be effective with $\deg \ba > 0$. Then $E_{\ba}$ has a model defined over~$K$.
\end{theoremintro}

\begin{subsubsec}{The $t$-comotive for $E_{\ba}$}
The arguments to prove this theorem focus first on the construction of the $t$-comotive (``dual $t$-motive,'' see Remark~\ref{R:comotive})
\[
H_{\ba}(\ok) = H^0 \bigl(\bU,\cO_{\bX}(-W_{\ba}^{(1)}) \bigr), \quad
\bU = \Spec \ok \otimes_{\FF_q} \bB,
\]
from which we define $E_{\ba}$. See \S\ref{SS:Sinhadefs} for precise definitions. We recall that the operation of $\sigma = \tau^{-1}$ on $H_{\ba}(\ok)$ is provided through multiplication by the Coleman function $g_{\ba}$, which then plays the role of shtuka functions for Drinfeld-Hayes modules as in Thakur~\cite{Thakur93}.  These $t$-comotives were originally defined by Anderson, Brownawell, and the second author~\cite{ABP04}, and they are parallel to the $t$-motives defined by Sinha~\cite{Sinha97b}.

As constructed, $H_{\ba}(\ok)$ is a left $\ok[t,\sigma]$-module (see~\S\ref{SS:tmoddualtmot}), and it is an ideal of the coordinate ring $\ok[\bU]$ of the affine part $\bU$ of $\bX$. A first step toward Theorem~\ref{T:EaKdefintro} is showing~$g_{\ba}$ is defined over $B[f^{-1}]$ in Proposition~\ref{P:coldef}. But then a key difficulty in proving Theorem~\ref{T:EaKdefintro} is determining that $H_{\ba}(\ok)$ has a $\ok[\sigma]$-basis consisting of functions defined over $K$, which we prove in Proposition~\ref{P:HaKdef}. The methods revolve around dissecting the arguments in~\citelist{\cite{ABP04}*{\S 6.4} \cite{BCPW22}*{Thm.~4.2.2} \cite{Sinha97b}*{Thm.~3.2.5}} to devise an algorithm for producing $\ok[\sigma]$-bases of $H_{\ba}(\ok)$ explicitly in \S\ref{SS:Fodef}, leading to explicit computation of $E_{\ba}$ (see~\S\ref{S:examples}).
\end{subsubsec}

\begin{subsubsec}{Good reduction}
We address reduction modulo primes of $B$ and primes of good reduction for $E_{\ba}$ in~\S\ref{SS:reduction}. To define good reduction we adapt the point of view of Gardeyn~\cite{Gardeyn02}, requiring that the rank of a $t$-comotive and its reduction be the same. In Corollary~\ref{C:EaSigma} we confirm that~$E_{\ba}$ has a model with good reduction outside of a finite set of primes $\Sigma_{\ba}$ of $B$ containing the primes dividing~$f$. These conditions are sufficient to prove Theorem~\ref{T:GossLMainintro} and Corollary~\ref{C:HeckeLMainintro} below, but we prove in the appendix that for every prime $\fp$ of $B$ not dividing~$f$, there is a model for $E_{\ba}$ for which it has good reduction at $\fp$ (see Corollary~\ref{C:Eagoodprimes}).
\end{subsubsec}

\begin{corollaryintro} \label{C:Eagoodprimesintro}
Let $\ba \in \cA_f$ be effective with $\deg \ba > 0$, and let $\fp$ be a maximal ideal of $B$ relatively prime to $f$. Then $E_{\ba}$ has good reduction at~$\fp$.
\end{corollaryintro}

Another component in our special $L$-value formulas is the determination of the period lattice of $E_{\ba}$. These periods had already been computed by Sinha~\cite{Sinha97b}*{\S 5}, but as his $t$-modules were constructed using $t$-motives instead of $t$-comotives, they are not technically the same. However, in the setting of $t$-comotives, one can apply Anderson's exponentiation theorem (Theorem~\ref{T:Anderson}, see also \citelist{\cite{HartlJuschka20}*{\S 2.5} \cite{NamoijamP24}*{\S 3.4}}), which permits us to bypass some of the complications of~\cite{Sinha97b} stemming from rigid analysis and homological algebra.

\begin{subsubsec}{Rigid analytic trivializations of $H_{\ba}$}
A major input into Anderson's exponentiation theorem is a rigid analytic trivialization of the $t$-comotive (see \S\ref{SSS:AndExp}). Adapting arguments from~\cite{ABP04}*{\S 6.4}, we show in Proposition~\ref{P:PhiaPsiadiag} that a rigid analytic trivialization for~$H_{\ba}$ can be expressed in terms of the infinite product
\[
\cG_{\ba} = \prod_{N=1}^{\infty} g_{\ba}^{(N)},
\]
which is a product of Frobenius twists of $g_{\ba}$ taken in an extension of the Tate algebra $\TT_{\theta}$ (see Proposition~\ref{P:cGaf}, and also~\cite{Sinha97b}*{\S 5.1}, where similar products were investigated).
\end{subsubsec}

\begin{subsubsec}{Interpolation formulas}
For $\ba = \sum_a m_a [a/f] \in \cA_f$ we take $\Pi(\ba) = \prod_a \Pi(a/f)^{m_a}$. Anderson~\cite{And92}*{\S 5.3} (see also~\citelist{\cite{ABP04}*{\S 6.3} \cite{Sinha97b}*{\S 5.3}}) proved interpolation formulas for Coleman functions yielding for $b \in (A/fA)^{\times}$,
\[
\Pi(b \star\ba) = \cG_{\ba}(\xi_b)^{-1},
\]
where $\xi_b \in \bX(K)$ are generic points arising from the embeddings $K \hookrightarrow \C$ and $b \star \ba$ represents the star action. See \S\ref{SS:Coleman} for more information. Sinha~\cite{Sinha97b}*{\S 5.2.6, \S 5.3.9} showed that certain coordinates of periods of $E_{\ba}$ are expressible in terms of special $\Pi$-values. We extend Sinha's results to arbitrary coordinates using hyperderivatives on $\ok(\bX)$ (stated as Theorem~\ref{T:periods}). We say $\ba \in \cA_f$ is \emph{basic} if it is of the form $\ba = [a/f]$, and note that Sinha's evaluations in~\cite{Sinha97b}*{\S 5.3} revolve around these basic cases where $(a,f)=1$.
\end{subsubsec}

\begin{theoremintro} \label{T:periodsintro}
Let $\ba \in \cA_f$ be effective with $\deg \ba > 0$. Let $E_{\ba}$ be constructed with respect to a normalized $\ok[\sigma]$-basis on $H_{\ba}(\ok)$. The period lattice $\Lambda_{\ba}$ of $E_{\ba}$ has an $\bA$-basis $\pi_1, \dots, \pi_r$ so that for index sets $\cP_{\ba}$, $\cS_{\ba}$ defined in \S\ref{SSS:normalized}, the following hold.
\begin{alphenumerate}
\item For $1 \leqslant i \leqslant r$, with $\ell_a = \mult_{\xi_b}(\Xi_{\ba}) - 1$,
\[
\pi_i = \begin{pmatrix}
\vdots \\ \pd_t^{\ell_b - j} \bigl( z^{i-1} \cG_{\ba}^{-1} \bigr)\big|_{\xi_b} \\ \vdots
\end{pmatrix}_{(b,j) \in \cS_{\ba}}.
\]
The coordinate corresponding to $(b,\ell_b) \in \cS_{\ba}$ is $\rho_b(\zeta)^{i-1} \Pi (b \star \ba)$ \textup{(cf.~\cite{Sinha97b}*{\S 5.3})}.
\item When $\ba = [a/f]$ is basic, for $1 \leqslant i \leqslant r$,
\[
\pi_i = \begin{pmatrix}
\vdots \\ \rho_b(\zeta)^{i-1} \Pi\bigl( [ba/f] \bigr) \\ \vdots
\end{pmatrix}_{b \in \cP_{\ba}}.
\]
\end{alphenumerate}
\end{theoremintro}

\subsection{Special \texorpdfstring{$L$}{L}-values}
Our main results relate Goss $L$-functions for $E_{\ba}$ with $L$-functions for $\chi_{\ba}$, and then we prove special value formulas for them in terms of special $\Pi$-values.

\begin{subsubsec}{Goss $L$-functions for $E_{\ba}$}
For $\ba \in \cA_f$ effective with $\deg \ba > 0$, we pick a model $E_{\ba} : \bA \to \Mat_d(B[\tau])$ together with a finite set of primes $\Sigma_{\ba}$ of $B$, including all primes dividing $f$, outside of which $E_{\ba}$ has good reduction (see \S\ref{SSS:EaBmodel}). The Goss $L$-function is
\[
L( E_{\ba}^{\vee}/K,\Sigma_{\ba},s) = \prod_{\fp \notin \Sigma_{\ba}} Q_{\ba,\fp}^{\vee} \bigl( \cN(\fp)^{s} \bigr)^{-1}, \quad s \in \ZZ,
\]
where $Q_{\ba,\fp}^{\vee}(X)$ is the reciprocal polynomial of the characteristic polynomial of the Frobenius at $\fp$ acting on the dual of the Tate module $T_{\lambda}(E_{\ba})$ (see \S\ref{SSS:GossLtmod}).
\end{subsubsec}

\begin{subsubsec}{Identities with Hecke $L$-series}
Letting $\psi_{\ba} = \chi_{-\ba}$, we prove in Proposition~\ref{P:Lidentities},
\[
L(E_{\ba}^{\vee}/K,\Sigma_{\ba},s) = \prod_{b \in (A/fA)^{\times}} L(\psi_{\ba}^{\rho_b}, \Sigma_{\ba},s) = L(\psi_{\ba},\Sigma_{\ba},s)^r.
\]
We obtain the first side of this identity by computing the characteristic polynomials of Frobenius of $E_{\ba}$ using techniques of~\cite{HuangP22} and relating them to $\psi_{\ba}$ (cf.~\cite{DebryPhD}*{\S 3.3.5}). The second part utilizes Galois symmetry (see Lemma~\ref{L:chiaGalsymm}).
\end{subsubsec}

\begin{subsubsec}{$L$-value identities}
The main tools for evaluating Goss $L$-functions are rooted in the work of Taelman~\cites{Taelman09, Taelman10, Taelman12}. For the $t$-modules $E_{\ba}$ we apply $L$-value formulas of Fang~\cite{Fang15}, which as in Taelman's identities equate $L(E_{\ba}^{\vee}/K,\Sigma_{\ba},0)$ with the product of a regulator term $\Reg(E_{\ba}/B) \in k_{\infty}^{\times}$ and a class module term $\rh(E_{\ba}/B) \in A_+$ (see \S\ref{SS:Taelman}). We appeal also to constructions of Angl\`es, Ngo Dac, Pellarin, and Tavares Ribeiro~\cites{ANT17b, ANT20, ANT22, APT18, AT17}. Although the model for $E_{\ba}$ used to construct $L(E_{\ba}^{\vee}/K,\Sigma_{\ba},s)$ may have bad reduction at some primes not dividing~$f$, Corollary~\ref{C:Eagoodprimesintro} permits defining $L(E^{\vee}/K,s)$ that avoids only Euler factors for $\fp \mid f$ (see Remark~\ref{R:LnoSigmas}). The following theorem and corollary (stated as Theorem~\ref{T:GossLMain} and Corollary~\ref{C:HeckeLMain}) assume that $(a,f)=1$, but the full results for $(a,f)\neq 1$ are in \S\ref{SS:specialLvalues}.
\end{subsubsec}

\begin{theoremintro} \label{T:GossLMainintro}
Let $\ba = [a/f] \in \cA_f$ be basic with $(a,f)=1$, and fix $E_{\ba}: \bA \to \Mat_d(B[\tau])$ as in \S\ref{SSS:EaBmodel}. There is $C_{\rR} \in \ok^{\times}$ with $\sgn(C_{\rR})=1$ and
\[
C_{\rR}^{r} \in \begin{cases}
k^{\times} &\text{if $r$ is even,} \\ 
K^{\times}\ \text{and}\ C_{\rR}^{2r} \in k^{\times} &\text{if $r$ is odd;}
\end{cases}
\]
$C_{\rH} \in \ok^{\times}$ with $\sgn(C_{\rH}) = 1$ so that $\rh(E_{\ba}/B) = C_{\rH}^r$; and $C_{\rE} \in K^{\times}$ with $C_{\rE}^{r} \in k^{\times}$ and $\sgn(C_{\rE}) = 1$, so that
\[
L(E_{\ba}^{\vee}/K,0) = C_{\rR}^r \cdot C_{\rH}^r \cdot C_{\rE}^r  \cdot \prod_{\substack{b \in A_+ \\ \deg b < \deg f \\ (b,f)=1}} \Pi(b/f)^r.
\]
\end{theoremintro}

The constants $C_{\rR}$, $C_{\rH}$, and $C_{\rE}$, arise from calculations for the regulator, class module, and extra Euler factors respectively, and typically depend on the model for $E_{\ba}$ (see \S\ref{SS:specialLvalues}).

\begin{corollaryintro} \label{C:HeckeLMainintro}
Under the same conditions as Theorem~\ref{T:GossLMainintro},
\[
L(\psi_{\ba},0) = C_{\rR} \cdot C_{\rH} \cdot C_{\rE} \cdot \prod_{\substack{b \in A_+ \\ \deg b < \deg f \\ (b,f)=1}} \Pi(b/f).
\]
\end{corollaryintro}

\begin{remark}
These special value formulas rely on $\ba$ being basic in a fundamental way. By Theorem~\ref{T:periodsintro}, the coordinates of periods of $E_{\ba}$ are $B$-multiples of elements of $k_{\infty}$, and as such the period lattice comprises a portion of Taelman's regulator lattice when computing $\Reg(E_{\ba}/B)$. When $\ba$ is basic, the rank of $\Lambda_{\ba}$ is the same as the rank of the regulator lattice, and thus up to a multiple from $A_+$ the regulator can be determined from the period lattice (see Remark~\ref{R:basic} and Proposition~\ref{P:RegCalc}).
When $\ba$ is not basic, $\Lambda_{\ba}$ forms a submodule of strictly smaller rank than the regulator lattice, and there are additional contributions to $\Reg(E_{\ba}/B)$ coming from as yet to be determined logarithms.
\end{remark}

\subsection{Outline}
After preliminary discussions in \S\ref{S:Prelim}, we review information on Carlitz cyclotomic fields, Coleman functions, and Anderson Hecke characters in~\S\ref{S:cyclo}. We investigate refined properties of Sinha modules $E_{\ba}$ in~\S\ref{S:Sinha}, especially rigid analytic properties of their $t$-comotives in~\S\ref{SS:analyticSinha} and fields of definition in~\S\ref{SS:Fodef}. In~\S\ref{S:GossHeckeL} we prove our main results on Goss and Hecke $L$-series for $E_{\ba}$ and $\psi_{\ba}$, and we provide examples in \S\ref{S:examples}. Finally in \S\ref{App:GoodRed} we characterize primes of good reduction for $E_{\ba}$.

\begin{acknowledgments}
The authors thank C.-Y.\ Chang, N.\ Ramachandran, and F.-T.~Wei for useful discussions on the contents of this paper.
\end{acknowledgments}

\section{Preliminaries} \label{S:Prelim}

\subsection{Notation} \label{SS:Notation}
We will use the following notation throughout.
\begin{longtable}{p{1truein}@{\hspace{5pt}$=$\hspace{5pt}}p{4.75truein}}
$\FF_q$ & a finite field with $q = p^m$ elements.\\
$A$ & $\FF_q[\theta]$, a polynomial ring in a variable $\theta$.\\
$A_+$ & the monic elements of $A$. \\
$k$ & $\FF_q(\theta)$, the field of fractions of $A$. \\
$k_{\infty}$ & $\laurent{\FF_q}{1/\theta}$, the completion of $k$ at its infinite place. \\
$\C$ & completion of an algebraic closure of $k_{\infty}$. \\
$\inorm{\,\cdot\,}$ & the $\infty$-adic norm on $\C$, normalized so that $\inorm{\theta} = q$. \\
$\deg$ & $-\ord_{\infty} = \log_q \inorm{\,\cdot\,}$. \\
$\ok$ & the algebraic closure of $k$ in $\C$. \\
$R^{\perf}$ & for a ring $R\subseteq \C$, the perfection of $R$ in $\C$. \\
$\ttheta$ & a fixed choice of $(-\theta)^{1/(q-1)}$ in $\ok$. \\
$f$ & a fixed polynomial in $A_+$, $\deg f \geqslant 1$. \\
$a \bmod f$ & for $a \in A$, the remainder of $a$ upon division by~$f$. \\
$K$, $K_f$ & the $f$-th Carlitz cyclotomic extension of $k$ (see \S\ref{SS:cyclodefs}), contained in~$\ok$. \\
$B$, $B_f$ & the integral closure of $A$ in $K$. \\
$\bA$ & $\FF_q[t]$, for a variable $t$ independent from $\theta$. \\
$\bk$, $\bK$, $\bB$ & fixed isomorphic copies of $k$, $K$, $B$ in $\overline{\FF_q(t)}$.
\end{longtable}

\begin{subsubsec}{Matrices}
For a ring $R$, we let $\Mat_{\ell \times m}(R)$ denote the set of $\ell \times m$ matrices with entries in $R$. We set $\Mat_m(R) = \Mat_{m\times m}(R)$ and $R^m = \Mat_{m\times 1}(R)$. For a matrix $M$, $M^{\tr}$ denotes its transpose. When $R$ is commutative, we let $\Char(N,X) \in R[X]$ to be the characteristic polynomial of a square matrix $N$ and $\Char(\alpha,V,X)$ to be the characteristic polynomial of an $R$-linear map $\alpha : V \to V$ on a free $R$-module $V$ of finite rank.
\end{subsubsec}

\begin{subsubsec}{Rings of functions and Frobenius twisting}
For $c \in \C^{\times}$, we define the \emph{Tate algebra}
\[
\TT_c \assign \biggl\{ \sum_{i=0}^\infty a_i t^i \biggm| \inorm{c}^i \cdot \inorm{a_i} \to 0 \biggr\},
\]
which consists of power series that converge on the closed disk of $\C$ of radius $\inorm{c}$. The canonical Tate algebra for the closed unit disk is denoted $\TT = \TT_1$. Primarily we will use $\TT$ and $\TT_\theta$ in this paper. The \emph{Gauss norm} on $\TT_c$ is defined for $g = \sum a_i t^i \in \TT_c$ as
\[
\dnorm{g}_c \assign \max_{0 \leqslant i < \infty} \bigl( \inorm{c}^i \cdot \inorm{a_i} \bigr),
\]
and $\TT_c$ is complete with respect to $\dnorm{\,\cdot\,}_c$. Moreover, $\TT_c$ is the completion of $\C[t]$ with respect to the Gauss norm. If $F$ is complete, then $\TT_c(F)$ is also complete. We will write $\dnorm{\,\cdot\,} = \dnorm{\,\cdot\,}_1$. For more information on Tate algebras see~\cite{FresnelvdPut}.

For $n \in \ZZ$ and $g = \sum a_i t^i \in \laurent{\C}{t}$, we define as usual the \emph{$n$-th Frobenius twist} to be
$g^{(n)} \assign \sum a_i^{q^n} t^i$. Restricted to subrings such as $\C[t]$ or $\TT$, Frobenius twisting is an automorphism. Moreover, for $g \in \TT$, we have $\dnorm{g^{(n)}} = \dnorm{g}^{q^n}$. More generally, for $c \in \C^{\times}$, we have $g \mapsto g^{(n)} : \TT_c \iso \TT_{c^{q^n}}$. For a matrix $M$ with entries in $\laurent{\C}{t}$, we let $M^{(n)}$ be the corresponding matrix where entrywise $(M^{(n)})_{ij} = M_{ij}^{(n)}$.
\end{subsubsec}

\begin{subsubsec}{Twisted polynomial rings}
Let $F$ be a field extension of $\FF_q$. Let $\tau: F \to F$ denote the $q$-th power Frobenius map $x \mapsto x^q$. We set $F[\tau]$ to be the ring of twisted polynomials, subject to the relation $\tau c = c^q\tau$, for $c \in F$. Then we consider $F[\tau]$ to be the $\FF_q$-algebra $\End_{\FF_q}(\GG_a/F)$ (see \cite{Goss}*{Ch.~1}). We further consider the ring $F[t,\tau]$ of polynomials in $t$ and $\tau$, which contains $F[\tau]$ as a subring, and satisfies
\[
\tau h = h^{(1)} \tau, \quad h \in F[t].
\]
If $F$ is perfect, we let $\sigma = \tau^{-1}$, and define $F[\sigma]$ and $F[t,\sigma]$ similarly with $\sigma h = h^{(-1)} \sigma$.

For $n \geqslant 1$ we define $\Mat_n(F)[\tau] = \Mat_n(F[\tau])$ and $\Mat_n(F)[\sigma] = \Mat_n(F[\sigma])$ (when $F$ is perfect) similarly so that
\[
\tau M = M^{(1)} \tau, \quad \sigma M = M^{(-1)} \sigma \quad M \in \Mat_n(F).
\]
We consider $\Mat_n(F[\tau])$ to be the $\FF_q$-algebra $\End_{\FF_q}(\GG_a^n/F)$, and if $\alpha = M_0 + M_1 \tau + \dots + M_s \tau^s \in \Mat_n(F)[\tau]$ and $\bx \in F^n$, then
\begin{equation} \label{E:twistedpolyop}
\alpha(\bx) = M_0 \bx + M_1 \bx^{(1)} + \dots + M_s \bx^{(s)}.
\end{equation}
For such $\alpha$, we set $\rd \alpha \assign M_0$.
\end{subsubsec}

\begin{subsubsec}{The adjoint $*$-operator}
For $F$ perfect, $* : F[\tau] \to F[\sigma]$ is defined by
\[
\alpha = \sum a_i \tau^i \mapsto \alpha^* = \sum a_i^{1/q^i} \sigma^i.
\]
Then $*$ is an anti-isomorphism of $\FF_q$-algebras, and notably $(\alpha\beta)^* = \beta^*\alpha^*$ for $\alpha$, $\beta \in F[\tau]$. By abuse of notation the inverse of $*$ is also denoted by $* : F[\sigma] \to F[\tau]$, and it is thought of as an anti-involution. See \cite{Goss}*{\S 1.7} for more details.

We extend the $*$-operator to matrices, $* : \Mat_{\ell \times m}(F[\tau]) \to \Mat_{m \times \ell}(F[\sigma])$, by setting
\[
(\alpha_{ij})^* \assign (\alpha_{ij}^*)^{\tr}.
\]
In this way, if $\alpha$, $\beta$ are matrices over $F[\tau]$ with dimensions so that $\alpha\beta$ is defined, then $(\alpha\beta)^* = \beta^*\alpha^*$. Again by abuse of notation the inverse of $*$ is also denoted by~$*$. See \citelist{\cite{BP20}*{\S 1.5.4} \cite{NamoijamP24}*{\S 2.2}} for more details.
\end{subsubsec}

\subsection{Anderson \texorpdfstring{$t$}{t}-modules and \texorpdfstring{$t$}{t}-comotives} \label{SS:tmoddualtmot}
Anderson defined $t$-modules in~\cite{And86} as higher dimensional versions of Drinfeld modules. Together with this definition Anderson introduced the theory of its corresponding $t$-motive. Later $t$-comotives (n\'ee dual $t$-motives) were defined in~\cite{ABP04}, and although one can attach both a $t$-motive and $t$-comotive to a given $t$-module, they do not a priori carry the same information. From the standpoint of abelian $t$-modules, an explicit equivalence was settled by Maurischat~\cite{Maurischat21}. For more on the general connections among these objects see \cites{BP20, HartlJuschka20, NamoijamP24}.

\begin{remark} \label{R:comotive}
As suggested by Gazda and Maurischat~\cite{GazdaMaurischat25}*{\S 3}, we have elected to use the term ``$t$-comotive'' instead of ``dual $t$-motive.'' The reason is that the latter term, though historical, can lead to confusion when one refers also to the ``dual of a $t$-motive.'' For similar reasons and  additional consistency we later use ``coabelian'' instead of ``$\bA$-finite.''
\end{remark}

\begin{subsubsec}{$t$-modules}
Let $F$ be a field extension of $\FF_q$. We fix an $\FF_q$-algebra homomorphism $i : \bA \to F$, and let $\otheta = i(t)$. We call the kernel of $i$ the characteristic of $F$, and if $\ker i = \{0\}$, the characteristic is generic. When $F \supseteq A$, we will always choose $i(t) = \theta \in A$ as defined in \S\ref{SS:Notation}. Otherwise, in particular when $F$ is a finite field, we will need to specify $\otheta$. An \emph{Anderson $t$-module} defined over $F$ is an $\FF_q$-algebra homomorphism $E : \bA \to \Mat_d(F)[\tau]$ such that $\rd E_t - \otheta \rI_d$ is nilpotent. We refer to $d$ as the dimension of $E$, and if $d=1$, then~$E$ is a Drinfeld module.
Using~\eqref{E:twistedpolyop}, $E$ induces an $\bA$-module structure on $L^d$, where~$L$ is an $F$-algebra. That is, for $a \in \bA$ and $\bx \in L^d$, we have $a \cdot \bx = E_a(\bx)$. We will write~$E(L)$ for $L^d$ with this $\bA$-module structure. For $\Lie(E)$, we can identify $\Lie(E)(L)$ with~$L^d$, and then the induced $\bA$-module structure is obtained through $a \mapsto \rd E_a$.

For a $t$-module $E$ over $F$ and an element $b \in \bA$, the $b$-torsion submodule of $E$ is $E[b] = \{ \bx \in \oF{}^{d} \mid E_b(\bx) = 0 \}$.
We say that $E$ is regular if there is an integer $r$ so that for all $b \in \bA$ relatively prime to the characteristic $\ker i \subseteq \bA$, we have $E[b] \cong (\bA/b \bA)^r$. In this case, if $\lambda \in \bA_+$ is a prime element with $\lambda \notin \ker i$, we have the Tate module
\[
T_{\lambda}(E) = \varprojlim E[\lambda^m] \cong \bA_{\lambda}^r,
\]
where $\bA_{\lambda}$ is the $\lambda$-adic completion of $\bA$. As usual, $T_{\lambda}(E)$ is a Galois module over $\bA_{\lambda}$.
\end{subsubsec}

\begin{subsubsec}{$t$-comotives} \label{SSS:dualtmotdef}
We follow the initial definitions from~\cite{ABP04}*{\S 4}. Let $F$ be a perfect $\bA$-field, and let $H$ be a left $F[t,\sigma]$-module. Consider the three conditions:
\begin{enumerate}
\item[(i)] $H$ is free of finite rank as a left $F[t]$-module;
\item[(ii)] $H$ is free of finite rank as a left $F[\sigma]$-module;
\item[(iii)] For $N \gg 0$, we have $(t-\otheta)^N H \subseteq \sigma H$.
\end{enumerate}
We say that $H$ is a \emph{$t$-comotive} if (ii)--(iii) holds, and additionally $H$ is \emph{coabelian} if (i) is satisfied. The \emph{dimension} of $H$ is $\rank_{F[\sigma]}(H)$, and if $H$ is coabelian, the \emph{rank} of $H$ is $\rank_{F[t]}(H)$. Morphisms of $t$-comotives are morphisms of left $F[t,\sigma]$-modules.
\end{subsubsec}

\begin{subsubsec}{From $t$-modules to $t$-comotives} \label{SSS:tmod-dualtmot}
Let $E$ be a $d$-dimensional $t$-module over a perfect $\bA$-field $F$. Then we define the $t$-comotive associated to $E$ to be
\[
H(E) = \Mat_{1 \times d}(F[\sigma]).
\]
Then $H(E)$ has the natural structure of a left $F[\sigma]$-module, and for $b \in \bA$ and $h \in H(E)$ we define $b \cdot h \assign h E_b^*$. In this way $H(E)$ has the structure of a left $F[t,\sigma]$-module, with condition (ii) of \ref{SSS:dualtmotdef} trivially satisfied. Furthermore, the condition that $\rd E_t - \otheta I_d$ be nilpotent implies that condition (iii) of \ref{SSS:dualtmotdef} holds, and so $H(E)$ is a $t$-comotive. If $H(E)$ is coabelian, then we also say that $E$ is coabelian, and then $E$ is regular with $r = \rank_{F[t]} H(E)$. For more details, see \citelist{\cite{BP20}*{\S 1.5.4} \cite{HartlJuschka20}*{\S 2.5.2} \cite{NamoijamP24}*{\S 2.3}}.
\end{subsubsec}

\begin{subsubsec}{From $t$-comotives to $t$-modules}
After some preliminary definitions the transition from $t$-comotives to $t$-modules, originally due to Anderson, can be made explicit. Define two maps $\varepsilon_0$, $\varepsilon_1 : \Mat_{1\times d}(F[\sigma]) \to F^d$ by
\begin{equation} \label{E:eps01}
\varepsilon_0 \biggl( \sum \ba_i \sigma^i \biggr) \assign \ba_0^{\tr}, \quad
\varepsilon_1 \biggl( \sum \ba_i \sigma^i \biggr) \assign \biggl( \sum \ba_i^{(i)} \biggr)^{\tr},
\end{equation}
where $\varepsilon_0$ is $F$-linear and $\varepsilon_1$ is $\FF_q$-linear.
Now let $H$ be a $t$-comotive. One checks that after selecting an $F[\sigma]$-basis for $H$ the induced maps,
\begin{equation} \label{E:eps01isoms}
\varepsilon_0 : \frac{H}{\sigma H} \iso F^d, \quad
\varepsilon_1 : \frac{H}{(\sigma -1 )H} \iso F^d,
\end{equation}
are isomorphisms of $F$-vector spaces and $\FF_q$-vector spaces respectively. They induce $\bA$-module structures on $F^d$, for which there is a $t$-module $E: \bA \to \Mat_d(F[\tau])$ with
\begin{equation} \label{E:LieEFandEFisoms}
\frac{H}{\sigma H} \cong \Lie(E)(F)\ \ \textup{($F[t]$-linear),} \quad
\frac{H}{(\sigma -1 )H} \cong E(F)\ \ \textup{($\bA$-linear)}.
\end{equation}
The defining matrices for $E$ depend on the choice of $F[\sigma]$-basis for $H$. Specifically, if we fix an $F[\sigma]$-basis $h_1, \dots, h_d$ of $H$, then there is a matrix $Q \in \Mat_d(F[\sigma])$ such that $t(h_1, \dots, h_d)^{\tr} = Q(h_1,\dots, h_d)^{\tr}$. Then $E$ is defined by
\begin{equation} \label{E:Etdef}
E_t = Q^* \in \Mat_d(F[\tau]).
\end{equation}
Moreover, if $h_1', \dots, h_d'$ is another $F[\sigma]$-basis of $H$, say $(h_1', \dots, h_d')^{\tr} = P(h_1, \dots, h_d)^{\tr}$ for $P \in \GL_d(F[\sigma])$, then the resulting $t$-module $E' : \bA \to \Mat_d(F[\tau])$ satisfies
\begin{equation} \label{E:Etprime}
E'_t = (P^*)^{-1} \cdot E_t \cdot P^*
\end{equation}
(e.g., see \cite{NamoijamP24}*{Eq.~(3.22)}). Thus any two such $t$-modules are isomorphic over~$F$. Notable also is that the processes $H \mapsto E$ and $E \mapsto H$ in this paragraph and \S\ref{SSS:tmod-dualtmot} are mutual inverses up to isomorphism. See \citelist{\cite{BP20}*{\S 1.5.6} \cite{HartlJuschka20}*{Prop.~2.5.8} \cite{NamoijamP24}*{\S 3.1}} for more details.
\end{subsubsec}

\begin{subsubsec}{Exponentials}
Let $F$ be an intermediate field $k \subseteq F \subseteq \C$, and let $E : \bA \to \Mat_d(F[\tau])$ be a $t$-module. Anderson~\cite{And86} defined a series $\exp_E = \sum_{i\geqslant 0} C_i \tau^i \in \power{\Mat_d(F)}{\tau}$ such that $C_0=\rI_d$ and as power series in $\tau$,
\[
\exp_E \cdot\, \rd E_b = E_b \cdot \exp_E, \quad b \in \bA.
\]
Moreover, $\exp_E$ is unique and is determined by the single functional equation for~$t$ itself. Letting $\bz = (z_1, \dots, z_d)^{\tr} \in \C[z_1, \dots, z_d]^d$ denote a system of coordinates for $E(\C)$, we can express $\exp_E$ as $\exp_E(\bz) = \sum_{i\geqslant 0} C_i \bz^{(i)} \in \power{F}{z_1, \dots, z_d}^d$, and the induced function $\exp_E : \C^d \to \C^d$ is entire. The functional equation for $\exp_E$ does not immediately imply that each $C_i \in \Mat_d(F)$ (but rather only in $\Mat_d(\oF)$), but this can be verified from Anderson's original arguments in~\cite{And86} (see \cite{NamoijamP24}*{Rem.~2.6}). The map
\[
\exp_E : \Lie(E)(\C) \to E(\C),
\]
is an $\bA$-module homomorphism, and if it is surjective, then $E$ is said to be \emph{uniformizable}. The kernel $\Lambda_E = \ker \exp_E$ is a discrete $\bA$-submodule of $\Lie(E)(\C)$.
\end{subsubsec}

\begin{subsubsec}{$t$-frames of coabelian $t$-modules} \label{SSS:tframe}
Continuing with the constructions of the previous paragraph, assume now that $E$ is coabelian of rank~$r$, and let $\bh_1, \dots, \bh_r \in H(E)$ be an $F[t]$-basis. For $\bsalpha = (\alpha_1, \dots, \alpha_r) \in \Mat_{1 \times r}(F[t])$, the map
\[
\iota : \Mat_{1 \times r}(F[t]) \to H(E), \quad
(\bsalpha \mapsto \alpha_1 \bh_1 + \dots + \alpha_r \bh_r),
\]
is an isomorphism of $F[t]$-modules. Notably, for $\bsalpha \in \Mat_{1\times r}(F[t])$ and $b \in \bA$ we have
\[
\rd E_b \bigl( (\varepsilon_0 \circ \iota)(\bsalpha) \bigr) = (\varepsilon_0 \circ \iota)(b \bsalpha), \quad
E_b \bigl( (\varepsilon_1 \circ \iota)(\bsalpha) \bigr) = (\varepsilon_1 \circ \iota)(b \bsalpha),
\]
(e.g., see \cite{NamoijamP24}*{Lem.~3.7}). If we let $\bh = (\bh_1, \dots, \bh_r)^{\tr}$, then there is $\Phi \in \Mat_r(F[t])$ with
\[
\sigma \bh = \Phi \bh,
\]
and $(\iota,\Phi)$ is a \emph{$t$-frame} of $E$. Also, $\det \Phi = c(t-\theta)^d$ for $c \in F^{\times}$ (see~\cite{NamoijamP24}*{Prop.~3.5}).
\end{subsubsec}

\begin{subsubsec}{Anderson's exponentiation theorem} \label{SSS:AndExp}
Continuing with the constructions of \S\ref{SSS:tframe}, a matrix $\Psi \in \GL_r(\TT)$ that satisfies $\Psi^{(-1)} = \Phi \Psi$ is a \emph{rigid analytic trivialization} for $E$, or for $\Phi$. The entries of $\Psi$ converge everywhere on $\C$ by~\cite{ABP04}*{Prop.~3.1.3}.

Anderson proved that the map $\varepsilon_0 \circ \iota : \Mat_{1\times r}(\C[t]) \to \C^d$ extends to a unique bounded $\C$-linear map
\[
\cE_0 : \bigl(\Mat_{1 \times r}(\TT_{\theta}), \dnorm{\,\cdot\,}_{\theta} \bigr) \to \bigl( \C^d, \inorm{\,\cdot\,} \bigr)
\]
of normed vector spaces (see \citelist{\cite{HartlJuschka20}*{Prop.~2.5.8} \cite{NamoijamP24}*{Lem.~3.12}}). Determining $\cE_0$ can be worked out for many common $t$-modules (see \cite{NamoijamP24}*{\S 3.5} for various examples and \S\ref{SSS:E0} below). We also let $\cE_1 \assign \varepsilon_1 \circ \iota : \Mat_{1\times r}(\C[t]) \to \C^n$.
A key result connecting exponentiation on $E$ to its $t$-comotive is due to Anderson.
\end{subsubsec}

\begin{theorem}[{Anderson; see \citelist{\cite{HartlJuschka20}*{Thm.~2.5.21, Cor.~2.5.23, Thm.~2.5.32} \cite{Maurischat22}*{\S 1} \cite{NamoijamP24}*{Thm. 3.13, Thm.~3.18}}}] \label{T:Anderson}
Let $E : \bA \to \Mat_d(\C[\tau])$ be a coabelian $t$-module. Let $(\iota,\Phi)$ be a $t$-frame of $E$.
\begin{alphenumerate}
\item Suppose $\bg \in \Mat_{1\times r}(\TT_{\theta})$ and $\bsalpha \in \Mat_{1\times r}(\C[t])$ satisfy $\bg^{(-1)} \Phi - \bg = \bsalpha$. Then
\[
\exp_E \bigl(\cE_0(\bg+\bsalpha) \bigr) = \cE_1(\bsalpha).
\]
\item $E$ is uniformizable if and only if it has a rigid analytic trivialization.
\item For a rigid analytic trivialization $\Psi$ of $E$,
\[
\Lambda_E = \cE_0 \bigl( \Mat_{1\times r}(\bA) \cdot \Psi^{-1} \bigr) \subseteq \C^n.
\]
\end{alphenumerate}
\end{theorem}

Thus values of $\exp_E$ can be captured in terms of values of rigid analytic functions satisfying certain difference equations. Moreover, the period lattice is generated by the rows of $\Psi^{-1}$ under evaluation by $\cE_0$.

\subsection{Hecke characters}
We adapt definitions from Schappacher~\cite{Schappacher} for algebraic Hecke characters, but see also work of Goss~\cite{Goss92a} and Hayes~\cite{Hayes93}.

\begin{subsubsec}{Signs and positivity} \label{SSS:signs}
For precision we adopt notions from~\cite{ANT17b}*{\S 1} for defining sign functions on $\ok_{\infty}^{\times}$. See also \cite{Goss}*{\S 7.2}. For $r \in \QQ$, we define $\theta^{r}$ unambiguously in the following way. We let $\theta_1 =\theta$, and for $n \geqslant 2$, take $\theta_n \in \ok_{\infty}^{\times} \subseteq \C^{\times}$ so that $\theta_n^n=\theta_{n-1}$. Then for $r = m/n!$, we set $\theta^{r} \assign \theta_{n}^{m}$, which is well-defined. Then $\theta^r \theta^s = \theta^{r+s}$ for all $r$, $s \in \QQ$, and moreover, for $r \in \QQ$, we have $\deg(\theta^r) = r$. Letting $U_{\infty} \assign \{ x \in \ok_{\infty} : \ord_{\infty}(x-1) > 0 \}$, we have
$\ok_{\infty}^{\times} = \oFF_q^{\times} \times ( \theta^{-1} )^{\QQ} \times U_{\infty}$,
and so each $x \in \ok_{\infty}^{\times}$ can be expressed uniquely as
\[
x = \sgn(x)\cdot \theta^{\deg(x)} \cdot u(x), \quad \sgn(x) \in \oFF_q^{\times},\ u(x) \in U_{\infty}.
\]
In this way we define the sign function $\sgn : \ok_{\infty}^{\times} \to \oFF_q^{\times}$. Elements of $\ok_{\infty}^{\times}$ of sign~$1$ are called \emph{positive}, e.g., the monic polynomials~$A_+$ are precisely the positive elements of~$A$.
\end{subsubsec}

\begin{subsubsec}{Algebraic Hecke characters} \label{SSS:Hecke}
Let $K$ and $F$ be finite separable extensions of $k$ contained in $\C$, and assume $K/k$ is abelian. Let $G = \Gal(K/k)$, and fix $\Theta = \sum m_\rho \rho \in \ZZ[G]$. Let $B$ be the integral closure of $A$ in $K$. For a non-zero integral ideal $\ff$ of $B$, let~$\cI_{K,\ff}$ denote the group of fractional ideals of $K$ with support disjoint from~$\ff$.
An \emph{algebraic Hecke character} $\chi$ of conductor dividing $\ff$ and infinity type $\Theta$ is a group homomorphism
\[
\chi : \cI_{K,\ff} \to F^{\times},
\]
such that for any principal ideal $(\alpha) \in \cI_{K,\ff}$ with $\alpha \in K^{\times}$ and $\alpha \equiv 1 \pmod{\ff}$, we have
\begin{equation} \label{E:chiprinc}
\chi((\alpha)) = \prod_{\rho \in G} \rho(\alpha)^{m_{\rho}} \rassign \alpha^{\Theta}.
\end{equation}
\end{subsubsec}

\subsection{Goss \texorpdfstring{$L$}{L}-series} \label{SS:GossL}
We consider $L$-series initially defined by Goss~\cite{Goss}*{Ch.~8}, which take values in $\C$. Although these $L$-series can be extended to Goss's non-archimedean analytic space $S_{\infty}$ (see \cite{Goss}*{\S 8.1}), we will need only consider their values at integers.

\begin{subsubsec}{$L$-series attached to Anderson $t$-modules} \label{SSS:GossLtmod}
We appeal to the definitions in Goss~\cite{Goss}*{\S 8.6}. Let $F/k$ be a finite separable extension with ring of integers $B$, and let $E : \bA \to \Mat_d(B[\tau])$ be a coabelian $t$-module of rank~$r$. For $\lambda \in \bA_+$ irreducible, let $\beta_\lambda : \Gal(F^{\sep}/F) \to \GL_r(\bA_{\lambda})$ be the representation induced by the Galois action on the Tate module $T_{\lambda}(E)$, and let $\beta_{\lambda}^{\vee}$ be its dual. For an integral prime $\fp$ of $B$ such that $\fp \nmid \lambda(\theta)$, let $\alpha_{\fp} \in \Gal(F^{\sep}/F)$ be a Frobenius element and set
\[
P_{\beta_{\lambda},\fp}(X) = \Char(\alpha_{\fp},T_{\lambda}(E),X), \quad
P_{\beta_{\lambda}^{\vee},\fp}(X) = \Char(\alpha_{\fp},T_{\lambda}(E)^{\vee},X),
\]
to be the associated characteristic polynomials in $\bA_{\lambda}[X]$ and $\bk_{\lambda}[X]$.

We now assume that $(\beta_{\lambda})$ and $(\beta_{\lambda}^{\vee})$ form \emph{strictly compatible} families of representations (see Goss~\cite{Goss}*{Def.~8.6.5}). Therefore, there is a finite set of primes $\Sigma_E$ of $B$ such that for primes $\fp$ outside~$\Sigma_E$, $P_{\beta_{\lambda},\fp}(X)$ and $P_{\beta_{\lambda}^{\vee},\fp}(X)$ have coefficients in $\bA$ and $\bk$ respectively and are independent of the choice of $\lambda$. Thus we can write
\[
P_{\fp}(X) = P_{\beta_{\lambda},\fp}(X)|_{t=\theta} \in A[X], \quad
P_{\fp}^{\vee}(X) = P_{\beta_{\lambda},\fp}^{\vee}(X)|_{t=\theta} \in k[X],
\]
without ambiguity. We further take $Q_{\fp}(X) = X^{r} P_{\fp}(1/X)$ and $Q_{\fp}^{\vee}(X) = X^{r} P_{\fp}^{\vee}(1/X)$ to be their reciprocal polynomials. For $\fp$ a finite prime of $F$ lying above $\wp \in A_+$, i.e., a maximal ideal of $B$ above $\wp A$, we let $\cN(\fp) = \wp^{[\FF_{\fp}:\FF_{\wp}]}$, where $\FF_{\wp} = A/\wp A$ and $\FF_{\fp} = B/\fp$.  Moreover, we extend $\cN(\fb)$ to all integral ideals of~$B$, and then
\[
\cN(\fb) A = \biggl( \prod_{\rho \in \Emb(F/k)} \rho(\fb) \biggr) \cap A
\]
gives the compatibility with the ideal norm. The \emph{Goss $L$-series} for $E/F$ are defined as
\begin{equation}
L(E/F,s) \assign \prod_{\fp \notin \Sigma_E} Q_{\fp}\bigl( \cN(\fp)^{-s} \bigr)^{-1}, \quad
L(E^{\vee}/F,s) \assign \prod_{\fp \notin \Sigma_E} Q_{\fp}^{\vee} \bigl( \cN(\fp)^{-s} \bigr)^{-1}.
\end{equation}
Here $s \in \ZZ$, and the products take values (when they converge) in $k_{\infty}$.
\end{subsubsec}

\begin{remark}
Goss~\cite{Goss}*{Ex.~8.6.6.2} demonstrated that, when $E$ is a Drinfeld module of positive rank, its associated Galois representations are strictly compatible. In this case the primes~$\Sigma_E$ are the primes of bad reduction for~$E$. For higher dimensional $t$-modules the situation is more subtle, as discovered by Gardeyn~\cite{Gardeyn02}*{\S 9}. However, as we will see in \S\ref{SS:reduction} and \S\ref{App:GoodRed} the Sinha modules $E_{\ba}$ we investigate have well-defined primes of good reduction and we avoid the technicalities in~\cite{Gardeyn02}.
\end{remark}

\begin{subsubsec}{Hecke $L$-series} \label{SSS:HeckeL}
Hecke $L$-series over function fields were initially developed by Goss~\cite{Goss92a}. In our context let $\chi : \cI_{K,\ff} \to F^{\times}$ be a Hecke character as in \S\ref{SSS:Hecke}, and let $B$ be the integral closure of $A$ in $K$. We define the Goss $L$-series attached to $\chi$ to be
\[
L(\chi,s) \assign \sum_{\substack{\fb \subseteq B \\ (\fb, \ff)=1}} \frac{\chi(\fb)}{\cN(\fb)^s}
= \prod_{\fp\, \nmid\, \ff} \biggl( 1 - \frac{\chi(\fp)}{\cN(\fp)^{s}} \biggr)^{-1}, \quad s \in \ZZ,
\]
which, when it converges, takes values in the completion of $F$ in $\C$. For Anderson's Hecke characters (see Proposition~\ref{P:Lidentities}), these $L$-functions converge for all $s \gg 0$.
\end{subsubsec}

\subsection{Special functions}
Here we recall various functions associated to Sinha modules that we will need. Where possible we utilize the conventions from~\cite{ABP04}.

\begin{subsubsec}{Carlitz exponential and $\be$} \label{SSS:expC}
The following constructions go back to Carlitz~\cite{Carlitz35}, but also appear in many sources such as~\citelist{\cite{Goss}*{Ch.~3} \cite{Papikian}*{Ch.~5} \cite{Rosen}*{Ch.~12} \cite{Thakur}*{Ch.~2}}. The Carlitz module $C : \bA \to A[\tau]$ is the rank~$1$ Drinfeld module defined by $C_t = \theta + \tau$. For $a \in A$, we will abuse notation and write $C_a(z)$ for the more cumbersome $C_{a(t)}(z)$.

The exponential function for $C$ and its logarithm are given by the infinite series,
\[
\exp_C(z) = \sum_{i=0}^{\infty} \frac{z^{q^i}}{D_i}, \quad
\log_C(z) = \sum_{i=0}^{\infty} \frac{z^{q^i}}{L_i},
\]
where $D_0=1$, $D_i = (\theta^{q^i} - \theta)(\theta^{q^i}-\theta^q) \cdots (\theta^{q^i}-\theta^{q^{i-1}})$, $i \geqslant 1$, $L_0=1$, and $L_i = (\theta-\theta^q) \cdots (\theta - \theta^{q^i})$, $i \geqslant 1$.
Then $\ker \exp_C = \Lambda_C = \tpi A$ is generated by the Carlitz period
\begin{equation}
\tpi = -\ttheta^q\, \prod_{j=1}^{\infty} \Bigl( 1 - \theta^{1-q^j} \Bigr)^{-1} \in k_{\infty}( \ttheta),
\end{equation}
where $\ttheta = (-\theta)^{1/(q-1)}$ is a fixed choice of a $(q-1)$-st root of~$-\theta$. To be made consistent with the choice of $\theta^{1/(q-1)}$ from {\S\ref{SSS:signs}}, this amounts to fixing a root $(-1)^{1/(q-1)}$. Then following \cite{ABP04}*{\S 5.2} we define $\be : k_{\infty} \to k_{\infty}(\ttheta)$ by
\begin{equation} \label{E:bedef}
\be(x) \assign \exp_C(\tpi x), \quad x \in k_{\infty}.
\end{equation}
If we vary $x \in k^{\times}$, then $\be(x)$ runs through torsion points on $C$. That is, for $f \in A_+$,
\begin{equation} \label{E:ftors}
C[f(t)] = \bigl\{ \be\bigl( a/f \bigr) \mid a \in A \bigr\}.
\end{equation}
We note also that $\ker \be = A$ and that 
\[
\inorm{x} < 1 \quad \Rightarrow \quad \inorm{\be(x)} = \inorm{\theta}^{q/(q-1)} \cdot \inorm{x}.
\]
Furthermore, as $\exp_C$ restricted to the open disk in $\C$ of radius~$\inorm{\theta}^{q/(q-1)}$ is an isometry by~\cite{Goss}*{Prop.~4.14.2}, we find that $\be(k_{\infty}) = \tpi\cdot \power{\FF_q}{1/\theta}$ (see \cite{ABP04}*{\S 5.2.10}).
\end{subsubsec}

\begin{subsubsec}{The function $\Omega$}
As in \cite{ABP04}*{\S 3.1.2} we define $\Omega(t)$ as
\begin{equation} \label{E:Omegaprod}
\Omega(t) = \ttheta^{-q}\,\prod_{j=1}^\infty \biggl(1 - \frac{t}{\theta^{q^j}} \biggr) \in \TT,
\end{equation}
which satisfies the functional equation $\Omega^{(-1)} = (t-\theta) \Omega$. By \cite{ABP04}*{Prop.~3.1.3} it has an infinite radius of convergence on~$\C$. As such it is a rigid analytic trivialization for the $t$-comotive of the Carlitz module (e.g., see \cite{BP20}*{\S 1.5}), and $\Omega(\theta) = -1/\tpi$.

In analogy with the Anderson-Thakur function $\omega = \sum_{i\geqslant 0} \be(1/\theta^{i+1}) t^i$ (see~\cite{AndThak90}*{\S 2.5}), which is a generating series for Carlitz $t$-power torsion, we expand $\Omega(t)$,
\begin{equation} \label{E:Omegaseries}
\Omega(t) \rassign \sum_{i=0}^{\infty} \tbe \bigl( 1 / \theta^{i+1} \bigr) t^i.
\end{equation}
This defines the sequence $\{ \tbe(1/\theta^i) \mid i \geqslant 0 \} \subseteq k_{\infty}(\ttheta)$.
Moreover, we define $\tbe : k_{\infty} \to k_{\infty}(\ttheta)$ to be the unique continuous $\FF_q$-linear function that interpolates this sequence and satisfies $\tbe|_A = 0$. The product in~\eqref{E:Omegaprod} implies that $\tbe(k_{\infty}) \subseteq (k_{\infty}(\ttheta))^q$, and thus the function
\begin{equation} \label{E:bestardef}
\be^* \assign \bigl( x \mapsto \tbe(x)^{1/q} \bigr) : k_{\infty} \to k_{\infty}(\ttheta),
\end{equation}
is well-defined, continuous, and $\FF_q$-linear. One sees readily that $\be^*$ is the same function~$\be^*$ defined in~\cite{ABP04}*{\S 5.2.5}.
Moreover, \eqref{E:bestardef} shows that $\Omega^{(-1)}(t) \in \TT(k_{\infty}(\ttheta))$.

By comparing \eqref{E:Omegaprod} and \eqref{E:Omegaseries}, for $i \geqslant 0$,
$\inorm{\tbe(1/\theta^{i+1})} = \inorm{\ttheta}^{-q} \cdot \bigl| \theta^{-(q+q^2+\cdots+q^i)} \bigr|_{\infty} = \inorm{\ttheta}^{-q^{i+1}}$,
and thus for $x \in k_{\infty}^{\times}$ with $\deg x < 0$,
\[
\inorm{\tbe(x)} = \inorm{\ttheta}^{-q^{-\deg x}}.
\]
In particular,
\begin{equation} \label{E:estarnorm}
\inorm{\be^*(x)} = \inorm{\ttheta}^{-q^{-1-\deg x}}, \quad x \in k_{\infty}^{\times},\ \deg x < 0.
\end{equation}
See also \cite{ABP04}*{Pf.~of Lem.~5.2.6}.
\end{subsubsec}

\begin{subsubsec}{The Carlitz adjoint and $\tbe$}
As defined in \citelist{\cite{Goss95} \cite{Goss}*{\S 3.7}}, the adjoint of the Carlitz module is defined by the $\FF_q$-algebra homomorphism $C^* : \bA \to k^{\perf}[\sigma]$ such that $C^*_t = \theta + \sigma$, and more generally for $a \in \bA$, $C_a^* = (C_a)^{*}$. In this way $C^*$ defines an $\bA$-module structure on perfect intermediate fields $k^{\perf} \subseteq F \subseteq \C$. Since $t\Omega(t) = \theta \Omega(t) + \Omega^{(-1)}(t)$,
\[
C_t^* \bigl( \tbe( 1/\theta^{i+1}) \bigr) = \tbe( 1/\theta^i), \quad i \geqslant 0,
\]
and so $\{ \tbe(1/\theta^i) \mid i \geqslant 1 \}$ is a division sequence of $t$-power torsion on $C^*$. By~\eqref{E:Omegaseries} we see that $\Omega(t)$ is an Anderson generating function for torsion on the adjoint Carlitz module.
Moreover, it follows that for any $x \in k_{\infty}$, we have $C_t^*(\tbe(x)) = \tbe(\theta x)$. Now as in the case of the Carlitz module, for $a \in A$ we will write $C_a^*$ for $C_{a(t)}^*$ to simplify notation. It follows from the $\FF_q$-linearity of $\tbe$ that for $a \in A$ and $x \in k_{\infty}$,
\begin{equation} \label{E:Castartbe}
C_a^*\bigl( \tbe(x) \bigr) = \tbe(ax).
\end{equation}
See \cite{ABP04}*{Rem.~5.2.9}. In particular, $C^*[f(t)] = \{ \tbe ( a/f) \mid a \in A \}$ for $f \in A$, $f\neq 0$.
\end{subsubsec}

\begin{remark}
Given the similarities between $C \leftrightarrow \be$ and $C^* \leftrightarrow \tbe$, we might have chosen to use `$\be^*$' for notation here instead of `$\tbe$.' However, $\be^*$ was already defined in \cite{ABP04} as what turns out to be the $q$-th root of $\tbe$, and we chose to be consistent with this previous work to avoid confusion. Furthermore, $\be^*$ is the more natural function from the standpoint of Coleman functions and Sinha modules (see \S\ref{SS:Coleman} and \S\ref{SS:Sinhadefs}).
\end{remark}

\begin{subsubsec}{The $\Psi_N$ polynomials}
We recall $\Psi_N(x) \in k[x]$, $N \geqslant 0$, studied initially by Carlitz~\cite{Carlitz35} and then by Anderson and Thakur~\cite{AndThak90}*{\S 3.4}. One defines $\Psi_N(x)$ through
\[
\exp_C \bigl( x \log_C(z) \bigr) = \sum_{N=0}^{\infty} \Psi_N(x) z^{q^N} = \sum_{N=0}^{\infty} \biggl( \sum_{j=0}^N \frac{x^{q^j}}{D_j L_{N-j}^{q^j}} \biggr) z^{q^N}.
\]
We then have the factorizations,
\begin{equation} \label{E:PsiNfact}
\Psi_N(x) = \frac{1}{D_N} \prod_{\substack{a \in A \\ \deg a < N}} (x-a), \quad
\Psi_N(x) - 1 = \frac{1}{D_N} \prod_{\substack{a \in A_+ \\ \deg a = N}} (x-a).
\end{equation}
Of particular importance to us is how $\Psi_N(x)$ governs a kind of duality between $\be$ and $\be^*$. We let $\Res : k_{\infty} \to \FF_q$ denote the projection to $\FF_q$ given by $\Res(\sum c_i \theta^{-i}) = c_{1}$. That is, $\Res(x) = -\res_{\infty}(x\,d\theta)$ on~$\PP^1$. We then have the following result.
\end{subsubsec}

\begin{lemma}[{\cite{ABP04}*{Lem.~5.4.2}}] \label{L:estarepair}
Let $N \in \ZZ$, $x\in k_{\infty}$, satisfy $\deg x \leqslant \min(-1,N)$. Then
\[
\sum_{i \geqslant 0} \be^*\bigl( \theta^{-i-1} \bigr)^{q^{N+1}} \be(\theta^i x) =
\begin{cases}
-\Psi_N(x) & \textup{if $N \geqslant 0$,} \\
\Res(\theta^{-N-1}x) & \textup{if $N < 0$.}
\end{cases}
\]
\end{lemma}

\begin{subsubsec}{Geometric factorial $\Pi(x)$ and $\Psi_N(x)$} \label{SSS:PiPsiN}
Recalling Thakur's factorial $\Pi(x)$ from \S\ref{SSS:GammaPi}, we also have a factorization in terms of the $\Psi_N(x)$ polynomials~\cite{Thakur91}*{Rem.~5.8},
\begin{equation} \label{E:PiPsiN}
\Pi(x) = \prod_{N=0}^\infty \bigl( 1 + \Psi_N(x)\bigr)^{-1}.
\end{equation}
Indeed this follows from~\eqref{E:PsiNfact}, using that $\Psi_N(x)$ is odd and the fact that $D_N$ is the product of all monic polynomials in~$A$ of degree~$N$ (e.g., see~\cite{Goss}*{Prop.~3.1.6}).
\end{subsubsec}

\section{Coleman functions and Anderson Hecke characters} \label{S:cyclo}

\subsection{Cyclotomic fields and Carlitz torsion} \label{SS:cyclodefs}
The construction of Carlitz cyclotomic fields mirrors the development of cyclotomic extensions of~$\QQ$. It has its foundations in Carlitz~\cite{Carlitz38}, who investigated explicit class field theory for~$k$, which Hayes~\cite{Hayes74} completed. See also \citelist{\cite{Goss}*{Ch.~7} \cite{Papikian}*{Ch.~7} \cite{Thakur}*{Ch.~3}}.
We fix throughout a polynomial $f \in A_+$ with $\deg f \geqslant 1$. As noted in~\eqref{E:ftors}, $C[f(t)] = \{ \be(a/f) \mid a \in A,\ \deg a < \deg f \}$, and we set
\[
\zeta \assign \zeta_f \assign \be(1/f),
\]
which generates $C[f(t)]$ as an $\bA$-module. The \emph{$f$-th Carlitz cyclotomic field} is
\[
K \assign K_f \assign k(\zeta_f),
\]
which is a finite Galois extension. As $\be(a/f) = C_{a}(\be(1/f))$, we have $K_f = k(C[f(t)])$.

\begin{subsubsec}{Galois action}
The irreducible polynomial of $\zeta_f$ over $k$ is the irreducible factor $\delta_f(z)$ of $C_{f(t)}(z) \in A[z]$ such that
\[
\delta_f(z) = \prod_{\substack{a \in A \\ \deg a  < \deg f \\ (a,f)=1}} \Bigl( z - \be(a/f) \Bigr) \in A[z].
\]
It is the analogue for $k$ of the classical cyclotomic polynomial. We have an isomorphism
\[
(b \mapsto \rho_b) : (A/fA)^{\times} \iso \Gal(K_f/k),
\]
which is compatible with Artin automorphisms. Moreover, for all $a \in A$, $b \in (A/fA)^{\times}$,
\begin{equation}
\rho_b \bigl( \be(a/f) \bigr) = C_{b}\bigl( \be(a/f) \bigr) = \be(ba/f).
\end{equation}
See \cite{Papikian}*{\S 7.1} for details.
\end{subsubsec}

\begin{subsubsec}{Adjoint Carlitz torsion and $\be^*$}
Goss~\citelist{\cite{Goss95}*{Thm.~3.2} \cite{Goss}*{Thm.~1.7.11}} showed that
$k( C[f(t)]) = k ( C^*[f(t)])$,
and so
\begin{equation} \label{E:Kftbe}
K_f = k\bigl( \tbe(1/f) \bigr) = k \bigl( \tbe(a/f) \mid a \in A \bigr).
\end{equation}
Furthermore, \citelist{\cite{Goss}*{Prop.~4.14.13} \cite{Poonen96}*{Cor.~11}} show that the Galois action on $C^*[f(t)]$ is dual to that on $C[f(t)]$. It follows that for $b \in (A/fA)^{\times}$,
\begin{equation} \label{E:Galoistbe}
\rho_b^{-1}\bigl( \tbe(a/f) \bigr) = C_b^*\bigl( \tbe(a/f) \bigr) = \tbe(ba/f).
\end{equation}
\end{subsubsec}

\begin{proposition} \label{P:estargal}
Let $f\in A_+$, $\deg f \geqslant 1$. Then
\[
K_f = k\bigl( \be^*(1/f) \bigr) = k\bigl( \be^*(a/f) \mid a \in A \bigr).
\]
Moreover, for $b \in (A/fA)^{\times}$ and $a \in A$,
\[
\rho_b^{-1} \bigl( \be^*(a/f) \bigr)  = \be^*(ba/f).
\]
\end{proposition}

\begin{proof}
Since $k_{\infty}(\ttheta) \cap \ok \subseteq k^{\sep}$, it follows from~\eqref{E:bestardef} that $\be^*(a/f) \in k^{\sep}$ for all $a \in A$. Moreover, since $\be^*(x) = \tbe(x)^{1/q}$ for $x \in k_{\infty}$, it must then be that $\be^*(a/f) \in K_f = k(\tbe(1/f))$ for all $a \in A$. The proposition then follows from~\eqref{E:Kftbe}--\eqref{E:Galoistbe}.
\end{proof}

The Carlitz torsion points $\be(a/f)$ are automatically integral over~$A$ since $C_f(x) \in A[x]$ has leading coefficient $\sgn(f)=1$. However, $\tbe(a/f)$ and $\be^*(a/f)$ need not be elements of $B_f$, though we can bound their denominators.

\begin{proposition} \label{P:estarval}
Let $f \in A_+$, $\deg f \geqslant 1$, and let $\fp$ be a finite prime of $B_f$. For $a \in A$,
\[
\ord_{\fp} \bigl( \be^*(a/f) \bigr) \geqslant -\frac{1}{q-1}\cdot \ord_{\fp}(f).
\]
\end{proposition}

\begin{proof}
(cf.\ Papikian~\cite{Papikian}*{Prop.~6.4.1}) Letting $m = \deg f$, we note that $C_f^*(x)^{q^m}$ is in $A[x]$ and has $\tbe(a/f) = \be^*(a/f)^q$ as a root. Also, $C_f^*(x)^{q^m}$ has degree $q^m$ in $x$, constant term~$1$, and leading coefficient $f^m$. The next highest degree term of $C_f^*(x)^{q^m}$ has degree at most $q^{m-1}$, so the steepest possible slope in the $\fp$-adic Newton polygon for $C_f^*(x)^{q^m}$ would be
\[
\frac{q^m \ord_{\fp}(f)}{q^m-q^{m-1}} = \frac{q}{q-1} \ord_{\fp}(f).
\]
This yields the desired inequality.
\end{proof}

\subsection{Cyclotomic curves over \texorpdfstring{$\FF_q$}{Fq}} \label{SS:cyclocurves}
The main reference is~\cite{ABP04}*{\S 6.3}. See also~\citelist{\cite{And92} \cite{Sinha97b}*{\S 2.2}}.

\begin{subsubsec}{Definitions}
We define $\cC_f(t,z) \in \FF_q[t,z]$ by $\cC_f(t,z) \assign C_{f}(z)|_{\theta=t}$, and likewise $\cD_f(t,z) \assign \delta_f(z)|_{\theta = t} \in \FF_q[t,z]$. We let
\[
U \assign U_f/\FF_q \assign \Spec \FF_q[t,z]/(\cD_f(t,z)) = \Spec \bB
\]
be the plane curve in $\AAA^2/\FF_q$, and we let $X/\FF_q \assign X_f/\FF_q$ be its nonsingular projective model. For an algebraically closed field $k \subseteq \KK \subseteq \C$, we further let
\[
\bU \assign \bU_f \assign \KK \times_{\FF_q} U_f, \quad
\bX \assign \bX_f \assign \KK \times_{\FF_q} X_f,
\]
be their extensions to $\KK$. It was shown in~\cite{ABP04}*{\S 6.3.2--4} that $U$ and $X$ satisfy several useful properties. The curves $U$, $\bU$, $X$, and $\bX$ are smooth and absolutely irreducible. The function field $\bK \assign \bK_f \assign \FF_q(X_f)$ is naturally isomorphic to $K$ by sending $t \leftrightarrow \theta$, $z \leftrightarrow \zeta_f$. As such $(A/fA)^{\times} \cong \Gal(\bK/\bk)$ by $b \mapsto \trho_{b}$, and
\begin{equation} \label{E:trhodef}
\trho_{b}(z) = \cC_b(t,z) \quad (\in \bB).
\end{equation}
Projecting onto the $t$-line makes $X$ into a $\Gal(\bK/\bk)$-cover of $\PP^1/\FF_q$ of degree $r \assign [K:k] = (A/fA)^{\times}$.
We have $\KK[\bU] = \KK[t,z]/(\cD_f(t,z))$ and $\KK(\bX) = \KK(t,z)$, and for $n \in \ZZ$, $n$-fold twisting extends to both $\KK[\bU]$ and $\KK(\bX)$. For $g \in \KK(\bX)$, $\divf(g^{(n)}) = \divf(g)^{(n)}$.
\end{subsubsec}

\begin{subsubsec}{Generic points on $\bX_f$}
For each $a \in (A/fA)^{\times}$ we set
\[
\xi_a \assign \bigl( \theta, \be(a/f) \bigr) = \bigl( \theta, C_a(\zeta) \bigr) \in \bU(\KK) \subseteq \bX(\KK).
\]
The points $\xi_a$ are distinct and rational over $K$, possessing a compatible $\Gal(K/k)$-action,
\[
\rho_b(\xi_a) = \xi_{ba}, \quad b \in (A/fA)^{\times}.
\]
We further let $\xi \assign \xi_1 = (\theta,\zeta)$.
\end{subsubsec}

\begin{subsubsec}{Points at $\infty$} \label{SSS:infinity}
Under the cover $t:X \to \PP^1$ projecting on the $t$-line, the decomposition subgroup of $\infty \in \PP^1(\FF_q)$ in $\Gal(K/k)$ corresponds to $\FF_q^{\times}$. Taking $n = |(A/fA)^{\times}/\FF_q^{\times}| = r/(q-1)$, we find that $X$ has $n$ points above $\infty$,
\[
\infty_1, \dots, \infty_n \in X(\FF_q),
\]
which are $\FF_q$-rational. Letting $K^+ \subseteq K$ denote the fixed field of $\FF_q^{\times} \hookrightarrow \Gal(K/k)$, we obtain the \emph{maximal real subfield} of $K$, and $K^+ = k(\zeta_f^{q-1})$. Then $\infty$ is totally split in~$K^+$, and places above $\infty$ in $K^+$ are totally ramified in $K/K^+$. See \cite{Rosen}*{Thm.~12.14}.

Letting $I = \infty_1 + \dots + \infty_n$ be the sum of the points above infinity in $\Div(\bX)$, we note that as a function on~$\bX$,
\begin{equation} \label{E:divt-theta}
\divf(t-\theta) = -(q-1)I + \sum_{\substack{a \in A \\ \deg a < \deg f \\ (a,f) = 1}} \xi_a.
\end{equation}
\end{subsubsec}

\begin{subsubsec}{Rigid analytic functions}
Let $\BB = \TT[z]$, where as above we carry the relation $\cD_f(t,z)=0$. In this way, $\BB$ is the affinoid algebra of rigid analytic functions on the inverse image of the closed unit disk in $\C$ under the projection $t:\bX \to \PP^1$. We define similarly $\BB_{\theta} = \TT_{\theta}[z]$, and $\BB_{\theta} \subseteq \BB$. The Gauss norms on $\TT$ and~$\TT_{\theta}$ extend to $\BB$ and $\BB_{\theta}$: noting that $1$, $z, \dots, z^{r-1}$ is a $\TT$-basis for $\BB$ (respectively, a $\TT_{\theta}$-basis for $\BB_{\theta}$),
\[
g = \sum_{i=0}^{r-1} g_i z^i \quad \Rightarrow \quad \dnorm{g} = \max_{0 \leqslant i \leqslant r-1} \bigl\{ \dnorm{g_i} \bigr\},
\]
(similarly for $\dnorm{\,\cdot\,}_{\theta}$). The rings $\BB$ and $\BB_{\theta}$ are complete with respect to $\dnorm{\,\cdot\,}$ and $\dnorm{\,\cdot\,}_{\theta}$. See \cite{Sinha97a}*{\S 3.2} for more information in this particular situation.
\end{subsubsec}

\subsection{Coleman functions} \label{SS:Coleman}
We review the construction of Coleman functions, which were initially studied by Anderson~\cite{And92} and Sinha~\cites{Sinha97a, Sinha97b}, who used Anderson's soliton functions to generalize previous examples of Coleman~\cite{Coleman88}. The definition of Coleman functions was simplified in~\cite{ABP04}*{\S 6.3}, and we rely on the definition there.

\begin{subsubsec}{Parameter space~$\cA_f$} Again we fix $f \in A_+$ with $\deg f \geqslant 1$, and we define symbols $[x]$ for $x \in f^{-1}A$, such that $[x]=[x']$ if $x\equiv x' \pmod{A}$. We then define $\cA_f$ to be the free abelian group on $[x]$ for $x \in f^{-1}A$, and so $\cA_f$ has rank~$q^{\deg f}$ with
\[
\cA_f = \bigoplus_{\substack{a \in A \\ \deg a < \deg f}} \ZZ \biggl[ \frac{a}{f} \biggr].
\]
A parameter $\ba \in \cA_f$ of the form $\ba = [a/f]$ with $a \neq 0$ is called \emph{basic}. If $\ba = \sum m_a [a/f]$ has all non-negative coefficients, it is \emph{effective}, and the \emph{degree} of $\ba$ is
\[
\deg \ba \assign \sum_{\substack{a \in A \setminus \{0\} \\ \deg a < \deg f}} m_a.
\]
Then $\deg$ induces a homomorphism on $\cA_f/(\ZZ[0/f])$ but not on $\cA_f$ itself.
\end{subsubsec}

\begin{remark}
As defined here, $\deg \ba = (q-1)\wt \ba$ from \cite{ABP04}*{\S 6.1.1}.
We do not need to include the case of $[0/f]$, as we will only be concerned with $\ba \in \cA_f$ that are effective with $\deg \ba > 0$; however, we maintain terminology consistent with \cites{ABP04, SinhaPhD, Sinha97a, Sinha97b}.
\end{remark}

\begin{subsubsec}{The $\star$-operator on $\cA_f$} \label{SSS:starop}
For $b \in A$ with $(b,f) = 1$, there is a unique automorphism
\[
(\ba \mapsto b \star \ba) : \cA_f \to \cA_f,
\]
with $b \star [a/f] = [ba/f]$. For $b$, $c \in A$ relatively prime to~$f$, we have $b \star (c \star \ba) = (bc)\star \ba$.
\end{subsubsec}

\begin{subsubsec}{Diamond brackets}
For $x \in k_{\infty}$ and $N \geqslant 0$ we define $\diam{x}$, $\diam{x}_N \in \{ 0,1\}$ as follows (cf.~\cite{ABP04}*{\S 5.5.1}). Write $x = a + y$ uniquely with $a \in A$ and $y \in (1/\theta)\cdot \power{\FF_q}{1/\theta}$. Then
\[
\diam{x} \assign \begin{cases}
1 & \textup{if $\sgn(y) = 1$,} \\
0 & \textup{otherwise,}
\end{cases} \quad
\diam{x}_N \assign \begin{cases}
1 & \textup{if $y = 1/\theta^{N+1} + O(1/\theta^{N+2})$,} \\
0 & \textup{otherwise.}
\end{cases}
\]
It follows that $\diam{x} = \sum_{N \geqslant 0} \diam{x}_N$. We note that for $a \in A$ and $f \in A_+$ (cf.~\cite{Sinha97b}*{\S 3.3.7}),
\begin{equation} \label{E:diamaf}
\diam{a/f} = \diam{(a\bmod f)/f} = \begin{cases}
1 & \textup{if $a \bmod f \in A_+$,} \\
0 & \textup{otherwise.}
\end{cases}
\end{equation}
We set homomorphisms $\diam{\,\cdot\,}_N$, $\diam{\,\cdot\,} : \cA_f \to \ZZ$, with $\diam{[a/f]}_N = \diam{a/f}_N$, $\diam{[a/f]} = \diam{a/f}$.
\end{subsubsec}

\begin{subsubsec}{Sinha's divisors}
Let $\ba \in \cA_f$. Define divisors on $\bX$ by
\begin{align} \label{E:XiW}
\Xi_{\ba} &= \sum_{\substack{b \in A \\ \deg b  < \deg f  \\ (b,f) = 1}} \diam{b \star \ba} \xi_b, \\
W_\ba &= \sum_{\substack{b \in A \\ \deg b  < \deg f  \\ (b,f) = 1}} \sum_{N = 1}^\infty \diam{b \star \ba}_N \sum_{j = 0}^{N-1} \xi_b^{(j)}. \notag
\end{align}
The sum defining $W_{\ba}$ has only finitely many non-zero terms, and so $W_{\ba}$ is indeed a divisor.
Using~\eqref{E:diamaf}, we note that for $[a/f] \in \cA_f$,
\begin{align} \label{E:XiWaf} 
\Xi_{[a/f]} &= \sum_{\substack{b \in A \\ \deg b < \deg f \\ (b,f)=1}} \biggl\langle \frac{ba}{f} \biggr\rangle \xi_b, \\
W_{[a/f]} &= \sum_{j=0}^{\deg f -2} \sum_{\substack{b \in A \\ \deg b < \deg f \\ (b,f)=1 \\ \deg(ba \bmod f) \leqslant j}} \biggl\langle \frac{ba}{f} \biggr\rangle \xi_b^{(\deg f - j - 2)}, \notag
\end{align}
which coincides with~\cite{Sinha97b}*{\S 3.3.7}. In particular,
\begin{equation} \label{E:XiW1f}
\Xi_{[1/f]} = \sum_{\substack{b \in A_+ \\ \deg b < \deg f \\ (b,f)=1}} \xi_b, \quad
W_{[1/f]} = \sum_{j=0}^{\deg f -2} \sum_{\substack{b \in A_+ \\ \deg b \leqslant j \\ (b,f)=1}} \xi_b^{(\deg f - j - 2)}.
\end{equation}
Necessarily, $\Xi_{[0/f]} = W_{[0/f]} = 0$. When $[a/f] \neq [0/f]$ but possibly $(a,f) \neq 1$, we can write $a/f = a'/f'$ for $f' \in A_+$ and $(a',f') = 1$. Then for the natural projection $\pi : \bX_f \to \bX_{f'}$,
\begin{equation} \label{E:XiWafpullback}
\Xi_{[a/f]} = \pi^* \Xi_{[a'/f']}, \quad W_{[a/f]} = \pi^* W_{[a'/f']},
\end{equation}
which matches with~\cite{Sinha97b}*{\S 3.3.7}.
\end{subsubsec}

\begin{lemma} \label{L:XiWprops}
Let $\ba = \sum m_a [a/f] \in \cA_f$. The following hold.
\begin{alphenumerate}
\item We have $\Xi_\ba = \sum m_a \Xi_{[a/f]}$ and $W_{\ba} = \sum m_a W_{[a/f]}$.
\item For $c \in (A/fA)^{\times}$, we have $\Xi_{c\star \ba} = \rho_c^{-1}(\Xi_\ba)$ $W_{c\star \ba} = \rho_c^{-1}(W_\ba)$.
\end{alphenumerate}
\end{lemma}

\begin{proof}
Part (a) is straightforward from \eqref{E:XiW} and the $\ZZ$-linearity of diamond brackets on~$\cA_f$. For (b) we need only check the cases $\Xi_{[a/f]}$ and $W_{[a/f]}$. Choosing $e \in A$ with $ec\equiv 1 \pmod{f}$, it follows from~\eqref{E:XiWaf} that
\[
\rho_c^{-1} (\Xi_{[a/f]}) = \sum_{\substack{b \in A \\ \deg b < \deg f \\ (b,f)=1}} \biggl\langle \frac{ba}{f} \biggr\rangle \xi_{eb} = \sum_{\substack{b \in A \\ \deg b < \deg f \\ (b,f)=1}} \biggl\langle \frac{bca}{f} \biggr\rangle \xi_b =  \Xi_{c\star [a/f]}.
\]
The result for $W_{[a/f]}$ is similar.
\end{proof}

In particular, since $\Xi_{[a/f]}$ and $W_{[a/f]}$ are effective, Lemma~\ref{L:XiWprops}(a) implies that $\Xi_{\ba}$ and~$W_{\ba}$ are effective  whenever $\ba$ is effective. Furthermore, for $\ba \in \cA_f$,
\[
\deg \Xi_{\ba} = \frac{r \deg \ba}{q-1}, \quad r = [K:k].
\]

\begin{subsubsec}{$f$-dual families}
A pair of sequences $\{ a_i \}_{i=1}^{\deg f}$, $\{b_i \}_{i=1}^{\deg f}$, of polynomials in~$A$ is called an \emph{$f$-dual family} if
\[
\Res \biggl( \frac{a_ib_j}{f} \biggr) = \delta_{ij}, \quad 1 \leqslant i,j \leqslant \deg f,
\]
where $\delta_{ij}$ is the Kronecker delta. In particular $\{ a_i \}_{i=1}^{\deg f}$, $\{b_i \}_{i=1}^{\deg f}$ both represent $\FF_q$-bases of $A/fA$. Moreover, the pairing
\[
(b \bmod f,c \bmod f) \mapsto \Res(bc/f) : A/fA \times A/fA \to \FF_q
\]
is perfect (see~\cite{ABP04}*{\S 5.4.3}), which provides the existence of $f$-dual families.
\end{subsubsec}

\begin{example} \label{Ex:fdual}
Let $d= \deg f$. As outlined in~\cite{ABP04}*{\S 5.4.3}, we note that for $1 \leqslant i, j \leqslant d$,
\[
\Res \bigl( \theta^{i+d - j -1}/f \bigr) = \begin{cases}
0 & \textup{if $i < j$,} \\
1 & \textup{if $i=j$,} \\
* & \textup{if $i > j$.}
\end{cases}
\]
Thus if $r_{ij} = \Res(\theta^{i+d-j-1}/f)$, then the matrix $(r_{ij}) \in \GL_d(\FF_q)$ is lower triangular with $1$'s on the diagonal. Denoting its inverse by $(s_{ij})$, we find that taking
\[
a_i = \sum_{m=1}^i s_{im}\theta^{m-1} \quad (1 \leqslant i \leqslant d), \quad
b_j = \theta^{d-j} \quad (1 \leqslant j \leqslant d),
\]
we obtain an $f$-dual family. Each $a_i$, $b_j$ is monic with $\deg a_i = i-1$ and $\deg b_j = d-j$.
\end{example}

\begin{lemma} \label{L:duality}
Given an $f$-dual family $\{ a_i \}_{i=1}^{\deg f}$, $\{b_i \}_{i=1}^{\deg f}$, for any $c \in A$,
\[
\sum_{i=1}^{\deg f} \Res(ca_i/f) b_i \equiv c \pmod{f}.
\]
\end{lemma}

\begin{proof}
Note that if $c \equiv c' \pmod{f}$, then for each $i$ we have $\Res(ca_i/f) = \Res(c'a_i/f)$.
There are unique $\gamma_j \in \FF_q$ so that $c \equiv \sum_{j=1}^{\deg f} \gamma_j b_j \pmod{f}$. Thus for each $i$,
\[
\Res(ca_i/f) = \sum_{j=1}^{\deg f} \gamma_j \Res( a_i b_j/f) = \gamma_i,
\]
where the last equality follows from the choice of $\{ a_i \}_{i=1}^{\deg f}$, $\{b_i \}_{i=1}^{\deg f}$ as an $f$-dual family. The result then follows from the choice of $\{\gamma_j\}$.
\end{proof}

Our primary application of $f$-dual families is an extension of Lemma~\ref{L:estarepair} due to Anderson, Brownawell, and the second author.

\begin{theorem}[{\cite{ABP04}*{Thm.~5.4.4}}] \label{T:fdualthm}
Let $f \in A_+$, $\deg f \geqslant 1$, be fixed, and let $\{ a_i \}_{i=1}^{\deg f}$, $\{b_i \}_{i=1}^{\deg f}$ be an $f$-dual family. For $a \in A$ with $\deg a < \deg f$, the following hold.
\begin{alphenumerate}
\item For $N \geqslant 0$,
\[
\sum_{i=1}^{\deg f} \be^*(a_i/f)^{q^{N+1}} \be(b_ia/f) = -\Psi_N(a/f).
\]
\item Furthermore, if $a \in A_+$, then
\[
\sum_{i=1}^{\deg f} \be^*(a_i/f) \be(b_ia/f)^{q^{\deg f - \deg a - 1}} = 1.
\]
\end{alphenumerate}
\end{theorem}

The connection with Lemma~\ref{L:estarepair} is obtained through the following calculation~\cite{ABP04}*{p.~290}. Note that for $n \geqslant 0$ and $x \in k_{\infty}$, $\Res(\theta^n x)$ is the coefficient of $\theta^{-n-1}$ in~$x$, and so as in~\cite{ABP04}*{\S 5.2.5}, we have
\begin{equation} \label{E:estarRes}
\be^*(x) = \sum_{n=0}^{\infty} \Res(\theta^n x) \be^*(1/\theta^{n+1}).
\end{equation}
Therefore, for $N \in \ZZ$,
\begin{align*}
\sum_{i=1}^{\deg f} \be^*(a_i/f)^{q^{N+1}} \be(b_ia/f) &= \sum_{i=1}^{\deg f} \sum_{n=0}^{\infty} \Res(\theta^n a_i/f) \be^*(\theta^{-n-1})^{q^{N+1}} \be(b_ia/f) \\
&=  \sum_{n=0}^{\infty} \be^*(\theta^{-n-1})^{q^{N+1}} \be\biggl( \sum_{i=1}^{\deg f} \Res(\theta^n a_i/f) b_i \cdot \frac{a}{f} \biggr) \\
&= \sum_{n=0}^{\infty} \be^*(\theta^{-n-1})^{q^{N+1}} \be(\theta^n a/f),
\end{align*}
where Lemma~\ref{L:duality} implies the last equality, and the rest follows from Lemma~\ref{L:estarepair}.

Theorem~\ref{T:fdualthm}(a) plays the role here that Anderson's interpolation formula~\cite{And92}*{Thm.~2} played in Sinha~\cite{Sinha97b}, but without the need of soliton machinery. Theorem \ref{T:fdualthm}(b) is instrumental in determining the divisor of the Coleman function below (see \cite{ABP04}*{\S 6.3.6}).

\begin{subsubsec}{Definition of Coleman functions}
Select an $f$-dual family $\{ a_i \}_{i=1}^{\deg f}$, $\{b_i \}_{i=1}^{\deg f}$. For $[a/f] \in \cA_f$, the \emph{Coleman function} $g_{[a/f]} \in \ok(\bX)$ is defined as in~\cite{ABP04}*{\S 6.3.5} by
\begin{equation} \label{E:gdef}
g_{[a/f]} \assign 1 - \sum_{i=1}^{\deg f} \be^*(a_i/f)\cC_{ab_i}(t,z).
\end{equation}
One checks that this definition is independent of the choice of $f$-dual family~\citelist{\cite{ABP04}*{\S 6.3.5} \cite{DavisPhD}*{Prop.~3.5.2.2}} and the representation of $[a/f]$ in $\cA_f$. We note that $g_{[0/f]} = 1$. For $[a/f] \neq [0/f]$, letting $I = \infty_1 + \dots + \infty_n$ as in~\S\ref{SSS:infinity}, by \cite{ABP04}*{\S 6.3.6} the divisor of $g_{[a/f]}$ is
\[
\divf\bigl( g_{[a/f]} \bigr) =  -I + \sum_{\substack{b \in A \\ (b,f)=1 \\ \deg b < \deg f}} \sum_{N=0}^{\infty} \biggl\langle \frac{ba}{f} \biggr\rangle_N \,\xi_{b}^{(N)} = -I +\sum_{\substack{b \in A \\ (b,f)=1 \\ \deg b < \deg f}} \biggl\langle \frac{ba}{f} \biggr\rangle \xi_b^{(\deg f - \deg (ba \bmod f) -1)}.
\]
\end{subsubsec}

\begin{remark}
When originally defined by Sinha~\cite{Sinha97b} using solitons, Sinha used the equivalent formulation
\begin{equation} \label{E:divgaf}
\divf\bigl( g_{[a/f]} \bigr)  = W_{[a/f]}^{(1)} - W_{[a/f]} + \Xi_{[a/f]} - I, \quad ([a/f] \neq [0/f]),
\end{equation}
which will be important for us in defining Sinha modules in \S\ref{S:Sinha}. There is a natural comparison to make with divisors of shtuka functions for Drinfeld-Hayes modules (see \citelist{\cite{Goss}*{\S 7.11} \cite{Thakur93} \cite{Thakur}*{\S 8.2}}).
\end{remark}

\begin{proposition} \label{P:coldef}
Let $g_{[a/f]} \in \ok(\bX)$ be defined as above. Then
\[
g_{[a/f]} \in B[f^{-1}][t,z].
\]
\end{proposition}

\begin{proof}
Since $\cC_a(t,z) \in \FF_q[t,z]$ for all $a \in A$, the result follows from \eqref{E:gdef} combined with Proposition~\ref{P:estarval}.
\end{proof}

For $\ba \in \cA_f$, $\ba = \sum m_a [a/f]$, we set
\begin{equation} \label{E:coleman}
g_{\ba} \assign \prod_{\substack{a \in A \\ \deg a < \deg f}} g_{[a/f]}^{m_a},
\end{equation}
and so $g_0 = 1$. By construction (and Lemma~\ref{L:XiWprops}(a)),
\begin{equation} \label{E:divgba}
\divf(g_{\ba}) = W_{\ba}^{(1)} - W_{\ba} + \Xi_{\ba} - I_{\ba},
\end{equation}
where for uniformity of notation we write $I_{\ba} \assign \deg(\ba)\cdot I$.

\begin{subsubsec}{Infinite products of Frobenius twists}
For $a \in A$ with $\deg a < \deg f$, we consider $g_{[a/f]} \in K[t,z]$ as an element of the affinoid algebra $\BB$. We see from~\eqref{E:gdef}, since $\cC_{ab_i}(t,z) \in \FF_q[t,z]$, that the Gauss norm of $g_{[a/f]}$ is governed by the quantities $\inorm{\be^*(a_i/f)}$. Moreover, using~\eqref{E:estarnorm} we see that $\dnorm{g_{[a/f]}} = 1$,
and more granularly that
\begin{equation} \label{E:normgaf-1}
\dnorm{g_{[a/f]}-1} \leqslant \max_{i=1}^{\deg f} \bigl\{ \inorm{\be^*(a_i/f)} \bigr\} = \inorm{\ttheta}^{-1} < 1,
\end{equation}
where the maximum in the middle occurs when $\deg a_i = \deg f - 1$.  We thus obtain the following proposition, improving a result of Sinha~\cite{Sinha97b}*{\S 5.1}.
\end{subsubsec}

\begin{proposition} \label{P:cGaf}
For $a\in A$ with $\deg a < \deg f$, the infinite product
\[
\cG_{[a/f]} \assign \prod_{N=1}^{\infty} g_{[a/f]}^{(N)}
\]
converges in $\BB_{\theta}$ with respect to $\dnorm{\,\cdot\,}_{\theta}$. Moreover, $\cG_{[a/f]} \in \BB_{\theta}^{\times}$.
\end{proposition}

\begin{proof}
We see from~\eqref{E:normgaf-1} that the product for $\cG_{[a/f]}$ converges in $\BB$. Sinha showed this in~\cite{Sinha97b}*{\S 5.1} using Green's functions and verified that $\cG_{[a/f]} \in \BB^{\times}$. To obtain the results in $\BB_{\theta}$ we analyze~\eqref{E:gdef} with respect to $\dnorm{\,\cdot\,}_{\theta}$ (cf.\ Wei~\cite{Wei22}*{Lem.~5.3.2, Rem.~5.3.3}).

Each $\cC_{ab_i}(t,z) \in \FF_q[t,z] = \bB$ that appears in \eqref{E:gdef} can be expressed uniquely as an $\FF_q[t]$-linear combination of $1$, $z, \dots, z^{r-1}$. Let $M \geqslant 0$ be the maximum of the degrees in~$t$ of these coefficients as we vary over all $\cC_{ab_i}(t,z)$. It then follows from~\eqref{E:gdef}, \eqref{E:normgaf-1}, and the definition of $\dnorm{\,\cdot\,}_{\theta}$, that for $N \geqslant 0$,
\[
\bigl\lVert g_{[a/f]}^{(N)} - 1 \bigr\rVert_{\theta} \leqslant \inorm{\ttheta}^{-q^N} \cdot \inorm{\theta}^{M}.
\]
(See also \S\ref{SSS:precursorRAT}.) Thus $\dnorm{g_{[a/f]}^{(N)} - 1}_{\theta} \to 0$ as $N \to \infty$, which shows that $\cG_{[a/f]} \in \BB_{\theta}$. Now by~\eqref{E:divgaf} we see that, for $N \geqslant 1$, the zeros of $g_{[a/f]}^{(N)}$ lie outside the inverse image under $t:\bX \to \PP^1$ of the disk of radius $\inorm{\theta}$ in $\C$. In fact the $t$-coordinates of the zeros of $g_{[a/f]}^{(N)}$ have norm going to infinity as $N \to \infty$. Thus each $g_{[a/f]}^{(N)}$, $N \geqslant 1$, is invertible in $\BB_{\theta}$, and also their product is in $\BB_{\theta}^{\times}$.
\end{proof}

\begin{remark}
Since for each $b \in (A/fA)^{\times}$ we have $t(\xi_b) = \theta$, this proposition implies that $\cG_{[a/f]}$ is regular and non-vanishing at each $\xi_b \in \bX(K)$.
\end{remark}

For $\ba \in \cA_f$ with $\ba = \sum m_a[a/f]$, we then define $\cG_{\ba} \assign \prod \cG_{[a/f]}^{m_a}$, which lies in the fraction field of $\BB_{\theta}$.

\begin{subsubsec}{Interpolation formulas} \label{SSS:interp}
The main application of Coleman functions to special $\Pi$-values was proved by Anderson~\cite{And92}*{\S 5.3}, but see also \citelist{\cite{ABP04}*{\S 6.3} \cite{Sinha97b}*{\S 5.3}}. Let $a$, $b \in A$, with $(b,f)=1$, and of degrees $< \deg f$. For $N \geqslant 0$, we see from \eqref{E:gdef} that
\begin{multline} \label{E:interp}
g_{[a/f]}^{(N+1)}(\xi_b)  = 1 - \sum_{i=1}^{\deg f} \be^*(a_i/f)^{q^{N+1}} C_{ab_i}(\be(b/f)) \\
= 1 - \sum_{i=1}^{\deg f} \be^*(a_i/f)^{q^{N+1}} \be(abb_i/f)
=1 + \Psi_{N} \biggl( \frac{ab \bmod f}{f} \biggr),
\end{multline}
where the last equality follows from Theorem~\ref{T:fdualthm}(a). Thus by~\eqref{E:PiPsiN},
\begin{equation}
\Pi\bigl( (ab \bmod f)/f \bigr) = \cG_{[a/f]}(\xi_b)^{-1} = \cG_{[ba/f]}(\xi)^{-1},
\end{equation}
which equates specializations of $\cG_{[a/f]}$ with special $\Pi$-values. This prompts the definition for $a \in A$, not necessarily of degree $< \deg f$,
\[
\Pi\bigl( [a/f] \bigr) \assign \Pi \bigl( (a \bmod f)/f \bigr),
\]
and also for $\ba = \sum_a m_a[a/f]$ we take $\Pi(\ba) = \prod_a \Pi ([a/f])^{m_a}$. In this way, we have
\begin{equation} \label{E:Pibstara}
\Pi(b \star \ba) = \cG_{\ba}(\xi_b)^{-1} = \cG_{b \star \ba}(\xi)^{-1}.
\end{equation}
for $b \in A$ with $\deg b < \deg f$, $(b,f)=1$.
\end{subsubsec}

\begin{remark}
There is a slight difference between this formula and the corresponding one in \cite{ABP04}*{\S 6.3.9}, which is stated as
\[
\Pi( b \star \ba) = \prod_{N=1}^{\infty} \bigl( g_{\ba}^{(N)}(\xi_b) \bigr)^{-1}.
\]
This latter one is an infinite product of values in $\CC_{\infty}$, whereas \eqref{E:Pibstara} arises from infinite products of functions in $\BB$ which are then evaluated at $\xi_b$.
\end{remark}

\begin{subsubsec}{Galois actions on $g_{\ba}$} \label{SSS:gaGalois}
For $\ba \in \cA_f$, there are two natural Galois actions on $g_{\ba}$. First $\Gal(K/k)$ acts on the coefficients of $g_{\ba}$, and secondly $\Gal(\bK(t,z)/\bk(t,z)) \cong \Gal(\bK/\bk)$, so there is a complementary $\Gal(\bK/\bk)$-action as well.
\end{subsubsec}

\begin{proposition} \label{P:gaGalois}
Let $\ba \in \cA_f$. For $b \in (A/fA)^{\times}$ the following hold.
\begin{alphenumerate}
\item $\rho_b^{-1} ( g_{\ba}) = g_{b \star \ba}$.
\item $\trho_b(g_{\ba}) = g_{b\star \ba}$.
\end{alphenumerate}
\end{proposition}

\begin{proof}
It suffices to consider the case $\ba=[a/f]$. For (b), if we let $\{a_i\}$, $\{b_i\}$ be $f$-dual bases of $A/fA$, then \eqref{E:trhodef} and~\eqref{E:gdef} imply
\[
\trho_b(g_{[a/f]}) = 1 - \sum_{i=1}^{\deg f} \be^*(a_i/f) \cC_{a b_i}\bigl( t, \cC_b(t,z) \bigr) =
1 - \sum_{i=1}^{\deg f} \be^*(a_i/f) \cC_{ba b_i}(t,z)
=g_{[ba/f]},
\]
and the result follows. For (a), we observe that
\[
\rho_b^{-1} (g_{[a/f]}) = 1 - \sum_{i=1}^{\deg f} \be^*(ba_i/f) \cC_{ab_i}(t,z) = 1 - \sum_{n=0}^{\infty} \be^*(1/\theta^{n+1}) \sum_{i=1}^{\deg f} \Res (\theta^n a_i b/f) \cC_{ab_i}(t,z),
\]
where the first equality follows from Proposition~\ref{P:estargal} and~\eqref{E:gdef} and the second from~\eqref{E:estarRes}. We then find
\[
\rho_b^{-1}(g_{[a/f]}) = 1 - \sum_{n=0}^{\infty} \be^*(1/\theta^{n+1}) \cC_{a \sum_{i=1}^{\deg f} \Res(\theta^n a_i b/f)b_i}(t,z)
= 1 - \sum_{n=0}^{\infty} \be^*(1/\theta^{n+1}) \cC_{a\theta^n b}(t,z),
\]
where here the first equality follows from fact that $\Res$ takes values in $\FF_q$ and the second from Lemma~\ref{L:duality}. Then we reverse the calculation, while keeping $b$ on the right, and
\[
\rho^{-1}_b(g_{[a/f]}) = 1 - \sum_{i=1}^{\deg f} \be^*(a_i/f) \cC_{abb_i}(t,z) = g_{b \star [a/f]} \] as desired. \end{proof}

\subsection{Anderson Hecke characters} \label{SS:AndHecke}
We review the construction Hecke characters of Anderson~\cite{And92}, extended by Sinha~\cite{Sinha97a},
\[
\chi_{\ba} : \cI_{K,\ff} \to K^{\times}, \quad \ba \in \cA_f,
\]
of conductor dividing the radical $\ff$ of $fB$. We first define $\chi_{\ba}$ for $\ba$ basic and then extend multiplicatively. We should note that Anderson's definition is achieved through specializations of his soliton function on the surface $X \times X$, whereas we phrase things equivalently in terms of Coleman functions on $\bX$ itself.

\begin{subsubsec}{Reduction of Coleman functions}
Fix a finite prime~$\fp$ of $B$ that is relatively prime to~$f$, let $B_{\fp}$ be the localization of $B$ at $\fp$, and let $\FF_{\fp} = B_{\fp}/\fp B_{\fp} \cong B/\fp$ be its residue field. The extensions of scalars $\cU = B_{\fp} \times_{\FF_q} U$ and $\cX = B_{\fp} \times_{\FF_q} X$ define curves over $\Spec B_{\fp}$. Since $X$ is smooth and proper over $\FF_q$, the resulting base extension $\cX \to \Spec B_{\fp}$ is also smooth and proper (e.g., see \cite{SilvermanATAEC}*{Prop.~IV.2.9}). Now $B_{\fp}[t,z]$ is the affine coordinate ring of $\cU$, which is irreducible, and its fraction field is $K(t,z) = K(\cX)$.

As $\cX$ is a smooth and proper model of $X$ over $B_{\fp}$, we have $\cX(B_{\fp}) = \bX(K)$, and so there is a well-defined reduction map $Q \mapsto \overline{Q}:  \bX(K) \to X(\FF_{\fp})$ modulo~$\fp$, which extends to divisors $\Div_K(\bX) \to \Div_{\FF_{\fp}}(X)$. We also have the natural reduction map
\[
h \mapsto \oh : B_{\fp}[t,z] \to \FF_{\fp}[t,z].
\]
Because the special fiber of $\cX$, which is isomorphic to $\FF_{\fp} \times_{\FF_q} X$, is irreducible, a function $h \in B_{\fp}[t,z]$ vanishes on the special fiber if and only if $\oh = 0$. Thus if $\oh \neq 0$, the divisor of~$h$ as an element of $\Div(\cX)$ is supported on horizontal divisors (e.g., see \cite{SilvermanATAEC}*{\S IV.7}). Each of these horizontal divisors intersect the special fiber at some closed point of $\FF_{\fp} \times_{\FF_q} X$. All of this is compatible with taking divisors of $h$ as a rational function on the generic fiber of $\cX$, yielding the identity in $\Div_{\FF_{\fp}}(X)$,
\begin{equation}\label{E:divred}
\divf(\oh) = \overline{\divf(h)}, \quad (\forall\, h \in B_{\fp}[t,z],\ \oh \neq 0).
\end{equation}
\end{subsubsec}

\begin{proposition} \label{P:ogaf}
Let $\fp$ be a finite prime of $B$ that is relatively prime to~$f$. For $[a/f] \in \cA_f$, we have $g_{[a/f]} \in B_{\fp}[t,z]$ and
\[
\og_{[a/f]} \neq 0.
\]
Moreover,
\[
\divf \bigl( \og_{[a/f]} \bigr) = \oW_{[a/f]}^{(1)} - \oW_{[a/f]} + \oXi_{[a/f]} - I_{[a/f]}.
\]
\end{proposition}

\begin{proof}
The proposition is trivially true if $[a/f]=[0/f]$, since then $g_{[0/f]}=1$. Assume then that $[a/f] \neq [0/f]$.
That $g_{[a/f]} \in B_{\fp}[t,z]$ follows directly from Proposition~\ref{P:coldef}. To show that $\og_{[a/f]} \neq 0$ we use~\eqref{E:interp}. In particular, taking $N=0$ in \eqref{E:interp} and noting that $\Psi_0(x) = x$, it follows that for $b \in A$ with $(b,f)=1$ and $\deg b < \deg f$, we have $g_{[a/f]}^{(1)}(\xi_b) = 1 + (ab \bmod f)/f$.
Now $\xi_b = (\theta,C_b(\zeta))$ has coordinates in $B_{\fp}$, so
\[
\og_{[a/f]}^{(1)}(\oxi_b) = \overline{g_{[a/f]}^{(1)}(\xi_b)} \equiv 1 + \frac{ab \bmod f}{f} \pmod{\fp}.
\]
We claim that there exists a choice of $b$ so that this expression is non-zero modulo $\fp$. Assuming this claim, we conclude that $\og_{[a/f]}^{(1)} \neq 0$, whence $\og_{[a/f]} \neq 0$, and then the statement on divisors then follows from~\eqref{E:divred}.
To address the claim, if $f = cf'$, $a = ca'$, with $(a',f')=1$ and $c \in A_+$, then $\{ ab \bmod f \mid b \in (A/fA)^{\times} \} = \{ cb' \bmod f \mid b' \in (A/f'A)^{\times} \}$. Moreover, if we let $\wp \in A_+$ be the unique monic prime of $A$ contained in $\fp$, then we wish to find $b' \in (A/f'A)^{\times}$ such that $f+ (cb' \bmod f) \not\equiv 0 \pmod{\wp}$. Since $\wp \nmid c$ and $\wp \nmid f$, taking $b' = \wp$ does the trick.
\end{proof}

\begin{subsubsec}{Anderson's Hecke character}
Letting $[a/f] \in \cA_f$, for $\fp$ a finite prime of $B$ relatively prime to~$f$ we obtain $\og_{[a/f]} \in \FF_{\fp}[t,z]$ as in Proposition~\ref{P:ogaf}. Letting $\ell = [\FF_{\fp}:\FF_q]$,
\begin{equation} \label{E:Gafdef}
G_{[a/f],\fp} \assign \prod_{i=0}^{\ell-1} \og_{[a/f]}^{(i)} \in \FF_q[t,z] = \bB,
\end{equation}
since $\Gal(\FF_{\fp}(t,z)/\FF_q(t,z)) \cong \Gal(\FF_{\fp}/\FF_q)$. Thus we can define
\begin{equation} \label{E:chiafdef}
\chi_{[a/f]}(\fp) \assign G_{[a/f],\fp}(\xi) = \prod_{i=0}^{\ell-1} \og_{[a/f]}^{(i)}(\xi) \in \FF_q[\theta,\zeta]=B,
\end{equation}
where we recall that $\xi = \xi_1 = (\theta,\zeta) \in \bX(K)$. In the following theorem, Anderson~\cite{And92} proved that $\chi_{[a/f]}$ is an algebraic Hecke character. Sinha~\cites{SinhaPhD, Sinha97a} then investigated~$\chi_{\ba}$ for general parameters. For a version of Anderson's proof relying on the properties of Coleman functions above, see \cite{DavisPhD}*{Thm.~4.2.2.2}.
\end{subsubsec}

\begin{theorem}[{Anderson~\cite{And92}*{\S 3}, Sinha~\cite{Sinha97a}*{\S 2.2}}] \label{T:AndHecke}
For $[a/f] \in \cA_f$, the assignment $\fp \mapsto \chi_{[a/f]}(\fp)$ from \eqref{E:chiafdef} extends multiplicatively to a function $\chi_{[a/f]} : \cI_{K,\ff} \to K^{\times}$ such that the following hold.
\begin{alphenumerate}
\item Letting
\[
\Theta_{[a/f]} = \sum_{\substack{b \in A \\ \deg b < \deg f \\ (b,f)=1}} \biggl\langle \frac{ba}{f} \biggr\rangle \rho_b^{-1} \quad \in \ZZ[\Gal(K/k)],
\]
for any finite prime $\fp$ of $B$ we have the equality of ideals in $B$,
\[
\chi_{[a/f]}(\fp) \cdot B = \fp^{\Theta_{[a/f]}}.
\]
\item For a finite prime $\fp$ of $B$ and $b \in (A/fA)^{\times}$, we have
\[
\bigl| \rho_b \bigl( \chi_{[a/f]}(\fp) \bigr) \bigr|_{\infty} = \inorm{\ttheta}^{\ell}, \quad \ell = [\FF_{\fp}:\FF_q].
\]
\item Moreover, $\chi_{[a/f]} : \cI_{K,\ff} \to K^{\times}$ is an algebraic Hecke character of conductor dividing~$\ff$ and infinity type $\Theta_{[a/f]}$. In particular, for $\alpha \in K^{\times}$ with $\alpha \equiv 1 \pmod{\ff}$,
\[
\chi_{[a/f]}(\alpha B) = \alpha^{\Theta_{[a/f]}}.
\]
\end{alphenumerate}
\end{theorem}

\begin{remark}
(i) Comparing notation we see that $\Theta_{[1/f]}$ here is the same as $\Theta_f(1)^*$ in~\cite{And92}. (ii) Sinha~\cite{SinhaPhD}*{Prop.~3.1} has proved a refined version of part (c) for all $\alpha \in K^{\times}$.
\end{remark}

\begin{lemma} \label{L:chiGalaction}
Let $a$, $b \in A$ with $(b,f) = 1$. For a finite prime $\fp$ of $B$ prime to~$f$,
\[
\chi_{[ba/p]}(\fp) = \rho_b \bigl( \chi_{[a/f]}(\fp) \bigr).
\]
\end{lemma}

\begin{proof}
By~\eqref{E:gdef} and~\eqref{E:chiafdef}, we see that for an $f$-dual family $\{ a_i \}$, $\{ b_j \}$,
\[
\chi_{[ab/f]}(\fp) = \prod_{i=1}^{\ell-1} \biggl( 1 - \sum_{j=1}^{\deg f} \overline{\be^*(a_j/f)}^{\,q^i} C_{ab_j}(C_b(\zeta)) \biggr)
= G_{[a/f],\fp}(\xi_b) = \rho_b \bigl( G_{[a/f],\fp}(\xi) \bigr),
\]
and the result follows.
\end{proof}

\begin{subsubsec}{General Anderson Hecke characters} \label{SSS:GenAndHecke}
As investigated by Sinha~\cite{Sinha97a}, for $\ba \in \cA_f$ we define
$\chi_{\ba} : \cI_{K,\ff} \to K^{\times}$, by setting
\[
\chi_{\ba}(\fb) = \prod_a \chi_{[a/f]}(\fb)^{m_a}, \quad \ba = \sum_a m_a [a/f].
\]
Then $\chi_{\ba}$ is an algebraic Hecke character of conductor dividing $\ff$ with infinity type $\Theta_{\ba} = \sum_a m_a \Theta_{[a/f]}$. For $\rho_b \in \Gal(K/k)$, Lemma~\ref{L:chiGalaction} shows that
\[
\chi_{\ba}^{\rho_b} \assign \rho_b \circ \chi_{\ba} = \chi_{b \star \ba}
\]
is a Hecke character. The character $\chi_{\ba}$ is also Galois symmetric.
\end{subsubsec}

\begin{lemma}[{Sinha~\cite{SinhaPhD}*{\S 3.2}}] \label{L:chiaGalsymm}
Let $\ba \in \cA_f$. For $\fb \in \cI_{K,\ff}$ and $\rho \in \Gal(K/k)$,
\[
\chi_{\ba}^{\rho}(\fb) = \chi_{\ba}(\fb^{\rho}).
\]
\end{lemma}

\begin{proof}
It suffices to prove this identity for $\ba = [a/f]$, and $\fb = \fp$ a maximal ideal of $B$. We claim that $(\rho^{-1} \circ \chi_{[a/f]} \circ \rho)(\fp) = \chi_{[a/f]}(\fp)$. Consider the following commutative diagram:
\begin{center}
\begin{tikzcd}
B_{\fp} \arrow[d, "\rho"'] \arrow[r, "x\,\mapsto\,\ox", two heads] & 
B/\fp \arrow[d, "\rho"] \arrow[r, "i", hook] & \oFF_q \arrow[r, hook] &  \C, \\
B_{\fp^{\rho}} \arrow[r, "x\,\mapsto\,\tx"', two heads] & 
B/\fp^{\rho} \arrow[ru, "j", hook]
\end{tikzcd}
\end{center}
where $i$ is a chosen embedding of $\FF_\fp$ in $\oFF_q$ and $j$ is the embedding of $\FF_{\fp^{\rho}}$ then induced by~$\rho^{-1}$.
We observe that for $x \in B_{\fp}$,
\[
i(\ox) = j(\widetilde{\rho(x)}).
\]
Let $\ell = [\FF_{\fp}:\FF_q] = [\FF_{\fp^{\rho}}:\FF_q]$. Suppose that $\rho = \rho_b$ for $b \in (A/fA)^{\times}$, and take $c \in A$ such that $bc \equiv 1 \pmod{f}$. Then~\eqref{E:chiafdef} implies
\[
(\rho^{-1} \circ \chi_{[a/f]} \circ \rho)(\fp) = (\rho_c \circ \chi_{[a/f]} )(\fp^{\rho_b}) = \rho_c \biggl( \prod_{m=0}^{\ell - 1} \tg_{[a/f]}^{(m)}(\xi_1) \biggr) = \prod_{m=0}^{\ell - 1}\tg_{[a/f]}^{(m)}(\xi_c),
\]
where $\tg_{[a/f]}$ is the reduction of $g_{[a/f]}$ modulo $\fp^{\rho_b}$. Taking $\{ a_n \}$, $\{ b_n \}$ to be an $f$-dual family, continuing this calculation, and using~\eqref{E:gdef} with the diagram above, yields
\begin{align*}
(\rho^{-1} \circ \chi_{[a/f]} \circ \rho)(\fp) &=
\prod_{m=0}^{\ell-1} \biggl(1 - \sum_{n=1}^{\deg f} j \Bigl( \widetilde{\be^*(a_n/f)} \Bigr)^{q^m} \rho_c ( \be( a b_n/f)) \biggr) \\
&= \prod_{m = 0}^{\ell-1} \biggl(1 - \sum_{n = 1}^{\deg f} i \bigl( \overline{ \rho_c( \be^*(a_n/f))} \bigr)^{q^m} \be (a b_n c/f) \biggr), \\
\intertext{where now applying Proposition~\ref{P:estargal},}
&= \prod_{m = 0}^{\ell-1} \biggl(1 - \sum_{n=1}^{\deg f} i \bigl( \overline{ \be^\star(a_n b/f) } \bigr)^{q^m} \be(a b_n c/f) \biggr).
\end{align*}
Since $\{ a_n b\}$, $\{ b_n c\}$ is also an $f$-dual family, we finally obtain
\[
(\rho^{-1} \circ \chi_{[a/f]} \circ \rho)(\fp)
= \prod_{m=0}^{\ell-1} \biggl(1 - \sum_{n=1}^{\deg f} i \bigl( \overline{\be^* (a_n/f)} \bigr)^{q^m} \be (a b_n/f) \biggr)^{-1}
= \chi_{[a/f]}(\fp),
\]
as desired.
\end{proof}

\section{Sinha modules} \label{S:Sinha}

\subsection{Sinha \texorpdfstring{$t$}{t}-comotives and \texorpdfstring{$t$}{t}-modules} \label{SS:Sinhadefs}
As in \S\ref{S:cyclo}, fix $f \in A_+$, $\deg f \geqslant 1$, and pick effective $\ba \in \cA_f$ with $\deg \ba > 0$. Let $\KK$ be an algebraically closed intermediate field $k \subseteq \KK \subseteq \C$. As in \S\ref{SS:cyclocurves}, we take $\bU = \KK \times_{\FF_q} U$ and $\bX = \KK \times_{\FF_q} X$. We then set
\begin{equation} \label{E:Hadef}
H_{\ba} \assign H^0(\bU,\cO_{\bX}(-W_{\ba}^{(1)})),
\end{equation}
which is a non-zero ideal of $\KK[t,z]$, the affine coordinate ring of $\bU$. If we wish to indicate the dependence on $\KK$, we will write $H_{\ba}(\KK)$.
Following \cite{ABP04}*{\S 6.4}, we make $H_{\ba}$ into a left $\KK[t,\sigma]$-module using the inherent left $\KK[t]$-module structure and setting
\[
\sigma h \assign g_{\ba} h^{(-1)}, \quad h \in H_{\ba}.
\]
Then $H_{\ba}$ is a coabelian $t$-comotive by~\cite{ABP04}*{\S 6.4.2}. The following lemma is then a consequence of~\cite{BCPW22}*{Prop.~4.2.3}, but we provide proofs of (a)--(c) in this case for completeness. That $H_{\ba}$ has a finite $\KK[\sigma]$-basis will be revisited in \S\ref{SS:Fodef}.

\begin{lemma} \label{L:Haprops}
The $t$-comotive $H_{\ba}$ is coabelian. Moreover, the following hold.
\begin{alphenumerate}
\item The rank of $H_{\ba}$ is $r = [K:k] = \#(A/fA)^{\times}$.
\item The dimension of $H_{\ba}$ is $d = \deg(\Xi_{\ba}) = \deg(\ba) \cdot r/(q-1)$.
\item We have $\sigma H_{\ba} = H^0(\bU,\cO_{\bX}(-W_{\ba}^{(1)} - \Xi_{\ba}))$.
\end{alphenumerate}
\end{lemma}

\begin{proof}
The ring $\KK[t,z]$ is the extension of scalars of $\bB = \FF_q[t,z]$ to $\KK$. Since $U/\FF_q$ is smooth and absolutely irreducible, $\KK[t,z]$ is also a Dedekind domain, and $[\KK[t,z]:\KK[t]] = [\bB:\bA] = r$. As an ideal of $\KK[t,z]$, $H_{\ba}$ also has rank $r$ over $\KK[t]$, which establishes~(a).

We next prove (c). We see from~\eqref{E:divgba} and~\eqref{E:Hadef} that $\sigma H_{\ba} \subseteq H^0(\bU,\cO_{\bX}(-W_{\ba}^{(1)} - \Xi_{\ba}))$. For the opposite containment, suppose that $\beta \in \KK[t,z]$ satisfies
$\divf(\beta)|_{\bU} \geqslant W_{\ba}^{(1)} + \Xi_{\ba}$. Then letting $\alpha = \beta^{(1)}/g_{\ba}^{(1)}$, we find that $\divf(\alpha)|_{\bU} \geqslant W_{\ba}^{(1)}$. Thus $\alpha \in H_{\ba}$, and $\sigma \alpha = \beta$, which proves (c) (see also \cite{ABP04}*{\S 6.4.2}).
For (b), we have
\[
\rank_{\KK[\sigma]} (H_{\ba}) = \dim_{\KK} \biggl( \frac{H_{\ba}}{\sigma H_{\ba}} \biggr) = \dim_{\KK} \left( \frac{H^0(\bU,\cO_{\bX}(-W_{\ba}^{(1)}))}{H^0(\bU,\cO_{\bX}(-W_{\ba}^{(1)} - \Xi_{\ba}))} \right).
\]
Through a short calculation the right-hand side is~$\deg(\Xi_{\ba})$ by \cite{BCPW22}*{Lem.~4.2.4}.
\end{proof}

\begin{subsubsec}{Definition of Sinha modules}
As in \S\ref{SS:tmoddualtmot}, we use the $t$-comotive $H_{\ba}(\KK)$ to define the \emph{Sinha module}
\[
E_{\ba} : \bA \to \Mat_d(\KK[\tau]),
\]
where $d= \deg(\Xi_{\ba})$ as in Lemma~\ref{L:Haprops}. The Sinha module $E_{\ba}$ is an abelian Anderson $t$-module, and since we can take $\KK = \ok$, we can find a model for $E_{\ba}$ defined over~$\ok$. A few remarks are in order.
\begin{itemize}
\item Since multiplication by $b \in \bB$ on $H_{\ba}$ is $\KK[t,\sigma]$-linear, we obtain a natural $\FF_q$-algebra homomorphism $\bB \to \End_{\KK[t,\sigma]}(H_\ba(\KK))$. Through functoriality there is a natural extension
\[
E_{\ba} : \bB \to \Mat_d(\KK[\tau]).
\]
\item The model for $E_{\ba}$ depends on the choice of $\KK[\sigma]$-basis for $H_{\ba}(\KK)$ as in~\eqref{E:Etdef}--\eqref{E:Etprime}. We will explore different models in \S\ref{SS:reduction}.
\item The $t$-comotives $H_{\ba}$ were originally defined in \cite{ABP04}*{\S 6.4}, based on a similar $t$-motivic construction of Sinha~\cites{SinhaPhD,Sinha97b}. Sinha also investigated the $t$-modules associated to these $t$-motives and determined a number of their properties, including the calculation of their period lattices in terms of special $\Gamma$-values. The $t$-modules $E_{\ba}$ defined here are not strictly speaking the ones that he worked with but are isogenous to them. Sinha modules are occasionally referred to in terms of ``soliton $t$-modules'' or ``soliton $t$-motives'' (e.g., see \citelist{\cite{BCPW22}*{\S 1.5.3} \cite{BP02}*{\S 4} \cite{Wei22}*{\S 5.3}}), based as they are on Anderson's soliton functions~\cite{And92}.
\item If we write $\ba = \sum m_a[a/f]$, then $g_{\ba}$, $H_{\ba}$, and $E_{\ba}$ do not depend on $m_0$, and we will without loss of generality assume that $m_0 = 0$.
\end{itemize}
\end{subsubsec}

\begin{subsubsec}{Hyperderivatives with respect to~$t$} \label{SSS:hyper}
We recall the sequence of $\C$-linear operators $\pd_t^j : \C[t] \to \C[t]$, $j \geqslant 0$, defined by $\pd_t^j(t^m) = \binom{m}{j} t^{m-j}$. These operators extend uniquely to both $\power{\C}{t}$ and $\C(t)^{\sep}$ (e.g., see \citelist{\cite{Conrad00}*{\S 4} \cite{Jeong11} \cite{NamoijamP24}*{\S 2.4}}). Then each $\pd_t^j$ leaves $\TT$ and $\TT_{\theta}$ invariant. Since $\C(\bX) = \C(t,z)$ is separable over $\C(t)$, we can uniquely define $\pd_t^j(z)$ for each $j \geqslant 0$. Moreover, by \cite{Conrad00}*{Thm.~5},
\[
\pd_t^j : \C(\bX) \to \C(\bX), \quad j \geqslant 0.
\]
In this way we can extend $\pd_t^j$ to $\C$-linear operators on the fraction fields of $\BB = \TT[z]$ and $\BB_{\theta} = \TT_{\theta}[z]$. To be a bit more precise, we note that $\pd_t^j(\cC_f(t,z)) = 0$, and if we write
\[
\cC_f(t,z) = f(t)z + \sum_{i=1}^{\deg f} a_i(t) z^{q^i}, \quad a_i(t) \in \FF_q[t],
\]
then the product rule \citelist{\cite{Jeong11}*{\S 2.2} \cite{NamoijamP24}*{Prop.~2.14(a)}} implies that
\[
\pd_t^j\bigl( \cC_f(t,z) \bigr) = 
\sum_{\ell=0}^j \pd_t^\ell(f(t))\pd_t^{j-\ell}(z)  + \sum_{i=1}^{\deg f} \sum_{\ell=0}^{j} \pd_t^{\ell}(a_i(t)) \pd_t^{j-\ell}\bigl( z^{q^i} \bigr) = 0.
\]
By the $p$-th power rule \citelist{\cite{Jeong11}*{\S 2.2} \cite{NamoijamP24}*{Prop.~2.14(b)}}, for $1 \leqslant i \leqslant \deg f$ the $\pd_t^{j-\ell}(z^{q^i})$ terms are either $0$ or a power of $\pd_t^{k}(z)$ with $k \leqslant (j-\ell)/q$. In particular the only occurrence of $\pd_t^j(z)$ in the above expression is in the first sum with $\ell=0$, and all other derivatives of $z$ occur with order strictly less than $j$. Thus an induction argument implies that $\pd_t^j(z) \in \FF_q(t)[z] = \bK$, and more generally
\[
\pd_t^j : \bk \otimes_{\bA} \BB \to \bk \otimes_{\bA} \BB, \quad
\pd_t^j : \bk \otimes_{\bA} \BB_{\theta} \to \bk \otimes_{\bA} \BB_{\theta}, \quad
j \geqslant 0.
\]
We think of $\bk\otimes_{\bA} \BB$, $\bk \otimes_{\bA} \BB_{\theta}$ as subrings of the fraction fields of $\BB$, $\BB_{\theta}$. For any $g \in \bk \otimes_{\bA} \BB_{\theta}$, we can evaluate $\pd_t^j(g)$ at any of the points $\xi_a \in \bX(\C)$ for $a \in (A/fA)^{\times}$.
\end{subsubsec}

\begin{subsubsec}{Expansions at $\xi_a$} \label{SSS:xiaexpansions}
Let $a \in (A/fA)^{\times}$, and let $h \in \C(\bX)$ be regular at $\xi_a$. Since $t-\theta$ is a uniformizer at $\xi_a$ by~\eqref{E:divt-theta}, in the completion of the local ring at $\xi_a$ we can expand~$h$ as power series,
\[
h = \sum_{j=0}^{\infty} \pd_t^j(h)|_{\xi_a} \cdot (t-\theta)^{j}.
\]
In particular $\pd_t^j(h)|_{\xi_a}=0$ for $0 \leqslant j < \ord_{\xi_a}(h)$, and $\pd_t^j(h)|_{\xi_a} \neq 0$ for $j = \ord_{\xi_a}(h)$.
\end{subsubsec}

\begin{subsubsec}{Normalized $\KK[\sigma]$-bases} \label{SSS:normalized}
We make the following definitions (cf.~Sinha~\cite{Sinha97b}*{\S 5.2.4}). For a divisor $D \in \Div(\bX)$ and $P \in \bX(\KK)$, we let $\mult_P(D) \in \ZZ$ be the multiplicity of~$P$ occurring in~$D$. We then define the index sets
\begin{gather*}
\cS_{\ba} = \left\{ (a,j) \in A \times \ZZ \biggm| \begin{gathered}
\deg a < \deg f,\ (a,f) = 1,\ \mult_{\xi_a}(\Xi_{\ba}) \neq 0,\\ 0 \leqslant j \leqslant \mult_{\xi_a}(\Xi_{\ba})-1 \end{gathered} \right\}, \\[5pt]
\cP_{\ba} = \{ a \in A \mid \deg a < \deg f,\ (a,f) = 1,\ \mult_{\xi_a}(\Xi_{\ba}) \neq 0 \}.
\end{gather*}
We observe that $\#\cS_{\ba} = \deg(\Xi_{\ba}) = d$ as in Lemma~\ref{L:Haprops}(b), and $\cP_{\ba}$ is simply the projection of $\cS_{\ba}$ to $A$. Suppose that $\{ \tilh_1, \dots, \tilh_d\}$ is a $\KK[\sigma]$-basis of $H_{\ba}(\KK)$. By Lemma~\ref{L:Haprops}(c) and~\cite{BCPW22}*{Lem.~4.2.4}, we can take a $\KK$-linear change of variables on $\{ \tilh_i \}$ to obtain a new $\KK[\sigma]$-basis $\{ h_{(a,j)} \mid (a,j) \in S_{\ba} \}$ so that
\begin{equation} \label{E:norm1}
h_{(a,j)} \in H^0( \bU, \cO_{\bX}(-W_{\ba}^{(1)} - \Xi_{\ba}+ (j+1)\xi_{a})) \setminus
H^0( \bU, \cO_{\bX}(-W_{\ba}^{(1)} - \Xi_{\ba}+ j\xi_{a})),
\end{equation}
and, taking
\[
\ell_a = \mult_{\xi_a}(\Xi_{\ba})-1,
\]
we have
\begin{equation} \label{E:norm2}
\frac{h_{(a,j)}}{(t-\theta)^{\ell_a - j}}\bigg|_{\xi_a} = 1.
\end{equation}
We note that we can do this since $t-\theta$ is a uniformizer at $\xi_a$, and
\[
\ord_{\xi_a}(h_{(a,j)}) = \mult_{\xi_a}(W_{\ba}^{(1)} + \Xi_{\ba}) - j - 1 = \ell_a - j,
\]
this last equality following because $W_{\ba}^{(1)}$ and $\Xi_{\ba}$ have disjoint supports.
Suppose further that for $a \in \cP_{\ba}$ and $0 \leqslant j, k \leqslant \ell_a$, we have
\begin{equation} \label{E:norm3}
\pd_t^k\bigl( h_{(a,j)} \bigr)\big|_{\xi_a} = \begin{cases}
1 & \textup{if $k = \ell_a - j$,} \\
0 & \textup{if $k \neq \ell_a-j$.}
\end{cases}
\end{equation}
In this case, we call a $\KK[t]$-basis $\{ h_{(a,j)} \mid (a,j) \in \cS_{\ba} \}$ for $H_{\ba}(\KK)$ a \emph{normalized basis}.
\end{subsubsec}

\begin{subsubsec}{Constructing normalized bases} \label{SSS:Wronskian}
As observed in \S\ref{SSS:normalized}, we can find a $\KK[t]$-basis $\{ h_{(a,j)} \mid (a,j) \in \cS_{\ba} \}$ satisfying \eqref{E:norm1}--\eqref{E:norm2}, but condition~\eqref{E:norm3} is not automatic. However, \eqref{E:norm1}--\eqref{E:norm2} do imply, as in \S\ref{SSS:xiaexpansions}, that
for $a \in \cP_{\ba}$ and $0 \leqslant j,k \leqslant \ell_a$,
\begin{equation} \label{E:diffhajeval}
\pd_t^k\bigl( h_{(a,j)} \bigr)\big|_{\xi_a} = \begin{cases}
1 & \textup{if $k=\ell_a-j$,} \\
0 & \textup{if $0 \leqslant k < \ell_a-j$.}
\end{cases}
\end{equation}
Thus for each $a \in \cP_{\ba}$ the Wronskian matrix
\[
M_a = \left. \begin{pmatrix}
\pd_t^{\ell_a}(h_{(a,0)}) & \cdots & \pd_t^{\ell_a}(h_{(a,\ell_a)}) \\
\vdots & & \vdots \\
\pd_t^{1}(h_{(a,0)}) & \cdots & \pd_t^{1}(h_{(a,\ell_a)}) \\
h_{(a,0)} & \cdots & h_{(a,\ell_a)}
\end{pmatrix} \right|_{\xi_a}
=
\begin{pmatrix}
1 & \pd_t^{\ell_a}(h_{(a,1)})|_{\xi_a} & \cdots & \pd_t^{\ell_a}(h_{(a,\ell_a)})|_{\xi_a} \\
& 1 & \ddots & \vdots \\
& & 1 & \pd_t^1(h_{(a,\ell_a)})|_{\xi_a} \\
& & & 1
\end{pmatrix}
\]
is upper triangular with $1$'s along the diagonal. Notably, $M_a^{-1}$ has the same shape, and
\[
M_a,\ M_a^{-1} \in \GL_{\ell_a+1}\bigl( \FF_q \bigl[ \pd_t^k \bigl( h_{(a,j)} \bigr)\big|_{\xi_a} : 0 \leqslant j \leqslant \ell_a,\ \ell_a-j < k \leqslant \ell_a \bigr] \bigr).
\]
Now for each $a \in \cP_{\ba}$ let
\[
\bigl( h'_{(a,0)}, \ldots, h'_{(a,\ell_a)} \bigr)
\assign 
\bigl (h_{(a,0)}, \dots, h_{(a,\ell_a)} \bigr) M_a^{-1}.
\]
We claim that $\{ h'_{(a,j)} \mid (a,j) \in \cS_{\ba} \}$ is a normalized $\KK[\sigma]$-basis for $H_{\ba}(\KK)$. Indeed for each $a \in \cP_{\ba}$, because $M_a^{-1}$ is upper triangular with $1$'s on the diagonal, $\{ h'_{(a,j)} \}$ satisfies \eqref{E:norm1}--\eqref{E:norm2}. Also, for $0 \leqslant j \leqslant \ell_a$,
\[
\bigl( \pd_t^j(h'_{(a,0)}), \ldots, \pd_t(h'_{(a,\ell_a)}) \bigr)\big|_{\xi_a}
= \bigl( \pd_t^j(h_{(a,0)}, \ldots, \pd_t^j(h_{(a,\ell_a)}) \bigr) \big|_{\xi_a}  \cdot M_a^{-1}.
\]
Letting $M_a'$ be the Wronskian matrix for $h'_{(a,0}, \dots, h'_{(a,\ell_a)}$ as above,
\[
M_a' = M_a \cdot M_a^{-1} = \rI.
\]
Thus $\{ h'_{(a,j)} \mid (a,j) \in \cS_{\ba} \}$ satisfies \eqref{E:norm3} and forms a normalized $\KK[\sigma]$-basis of $H_{\ba}(\KK)$.
\end{subsubsec}

\begin{lemma} \label{L:normdecomp}
Fix a normalized $\KK[\sigma]$-basis $\{ h_{(a,j)} \}$ for $H_{\ba}(\KK)$. For $h \in H_{\ba}(\C)$,
\[
h = \sum_{(a,j) \in \cS_{\ba}} \pd_t^{\ell_a - j}(h)|_{\xi_a} \cdot h_{(a,j)}
\pmod{\sigma H_{\ba}(\C)}.
\]
\end{lemma}

\begin{proof}
Pick $c_{(a,j)} \in \C$ so that $h = h' + \sum_{(a,j) \in \cS_{\ba}} c_{(a,j)} \cdot h_{(a,j)}$, with $h' \in \sigma H_{\ba}(\C)$. Then for $k \geqslant 0$,
\[
\pd_t^k(h) = \pd_t^k(h') + \sum_{a \in \cP_{\ba}} \sum_{j=0}^{\ell_a} c_{(a,j)} \cdot \pd_t^k \bigl( h_{(a,j)} \bigr).
\]
For $b \in \cP_{\ba}$, we note that (i) $\ord_{\xi_b}(h') > \ell_b$; (ii) $\ord_{\xi_b}(h_{(a,j)}) > \ell_b$ when $a \neq b$; and as before (iii) $\ord_{\xi_b}(h_{(b,j)}) = \ell_b-j$. Thus if $k \leqslant \ell_b$, then replacing $k$ with $\ell_b-k$ and evaluating at~$\xi_b$, we have
\[
\pd_t^{\ell_b-k}(h)|_{\xi_b} = \sum_{j=0}^{\ell_b} c_{(b,j)} \cdot \pd_t^{\ell_b-k} \bigl( h_{(b,j)} \bigr) \big|_{\xi_b} = c_{(b,k)},
\]
where the last equality follows from~\eqref{E:norm3}.
\end{proof}

\begin{subsubsec}{The map $\cE_0$ for $H_{\ba}$} \label{SSS:E0}
Continuing with the notation above, assume we have chosen a normalized $\KK[\sigma]$-basis $\{ h_{(a,j)} \}$ for $H_{\ba}(\KK)$. We define
\[
\teps_0 : H_{\ba}(\C) \to \C^{d}
\]
by
\[
\teps_0(h) \assign \begin{pmatrix}
\vdots \\ \pd_t^{\ell_a-j}(h)|_{\xi_a}  \\ \vdots
\end{pmatrix}_{(a,j) \in \cS_{\ba}},
\]
which is $\C$-linear and
whose kernel is precisely $\sigma H_{\ba}(\C)$. For a $\C[t]$-basis $\{ v_1, \dots, v_r\}$ of $H_{\ba}(\C)$, we 
let $\tiota : \Mat_{1 \times r}(\C[t]) \to H_{\ba}(\C)$ be the evident $\C[t]$-isomorphism. We then have a commutative diagram, where the vertical isomorphism is the $\C[\sigma]$-linear map induced by the basis $\{ h_{(a,j)} \}$ and $\iota$ is then the map induced by $\tiota$:
\begin{center}
\begin{tikzcd}
\Mat_{1 \times r}(\C[t]) \arrow[d, equal] \arrow[r, "\iota"] & \Mat_{1 \times d}(\C[\sigma]) \arrow[d, "\rotatebox{90}{$\sim$}"] \arrow[r, "\varepsilon_0"] & \C^d. \\
\Mat_{1 \times r}(\C[t]) \arrow[r, "\tiota"] & 
H_{\ba}(\C) \arrow[ru, "\teps_0"']
\end{tikzcd}
\end{center}
In this way $\cE_0 : \Mat_{1 \times r}(\TT_{\theta}) \to \C^d$ is defined continuously via $\teps_0 \circ \tiota = \varepsilon_0 \circ \iota$. Therefore, for $(\beta_1, \dots, \beta_r) \in \Mat_{1 \times r}(\TT_{\theta})$, if we let $\beta = \beta_1 v_1 + \dots + \beta_r v_r$, then
\begin{equation} \label{E:E0}
\cE_0(\beta_1, \dots, \beta_r) = \begin{pmatrix}
\vdots \\ \pd_t^{\ell_a-j}(\beta)|_{\xi_a}  \\ \vdots
\end{pmatrix}_{(a,j) \in \cS_{\ba}}.
\end{equation}
In particular, in the case that $\mult_{\xi_a}(\Xi_{\ba}) = 1$ for all $a \in \cP_{\ba}$, e.g., when $\ba \in \cA_f$ is basic, things are even simpler. In this case $\ell_a=0$ for all $a \in \cP_{\ba}$, and we obtain
\begin{equation} \label{E:E0basic}
\cE_0(\beta_1, \dots, \beta_r) = \begin{pmatrix}
\vdots \\ \beta(\xi_a) \\ \vdots
\end{pmatrix}_{a \in \cP_{\ba}}.
\end{equation}
\end{subsubsec}

\begin{lemma} \label{L:dEab}
Fix a normalized $\KK[\sigma]$-basis of $H_{\ba}(\KK)$ as above. For $b \in \bB$, the matrix $\rd E_{\ba,b}$ is a block diagonal matrix, where for each $a \in \cP_{\ba}$, there is an $(\ell_a+1) \times (\ell_a+1)$ upper-triangular block in $\rd E_{\ba,b}$ of the form
\[
\begin{pmatrix}
b(\xi_a) & \pd_t^1(b)|_{\xi_a} & \cdots & \pd_t^{\ell_a}(b)|_{\xi_a} \\
& \ddots & \ddots & \vdots \\
& & b(\xi_a) & \pd_t^1(b)|_{\xi_a} \\
& & & b(\xi_a)
\end{pmatrix} = \rho_a \begin{pmatrix}
b(\xi) & \pd_t^1(b)|_{\xi} & \cdots & \pd_t^{\ell_a}(b)|_{\xi} \\
& \ddots & \ddots & \vdots \\
& & b(\xi) & \pd_t^1(b)|_{\xi} \\
& & & b(\xi)
\end{pmatrix}.
\]
\end{lemma}

\begin{proof}
As in \eqref{E:LieEFandEFisoms}, the induced map
\[
\teps_0 : \frac{H_{\ba}(\C)}{\sigma H_{\ba}(\C)} \iso \Lie(E_{\ba})(\C)
\]
is an isomorphism of $\C[t]$-modules. And in fact it is an isomorphism of $\C[t,z]$-modules. Now for $b \in \bB \subseteq \C[t,z]$ and $h \in H_{\ba}(\C)$, we see from the product rule,
\[
\begin{pmatrix}
\pd_t^{\ell_a}(bh)|_{\xi_a} \\
\vdots \\
\pd_t^1(bh)|_{\xi_a} \\
(bh)|_{\xi_a}
\end{pmatrix}
= \begin{pmatrix}
b(\xi_a) & \pd_t^1(b)|_{\xi_a} & \cdots & \pd_t^{\ell_a}(b)|_{\xi_a} \\
& \ddots & \ddots & \vdots \\
& & b(\xi_a) & \pd_t^1(b)|_{\xi_a} \\
& & & b(\xi_a)
\end{pmatrix}
\begin{pmatrix}
\pd_t^{\ell_a}(h)|_{\xi_a} \\
\vdots \\
\pd_t^1(h)|_{\xi_a} \\
(h)|_{\xi_a}
\end{pmatrix}.
\]
The conclusion then follows from Lemma~\ref{L:normdecomp} (and its proof) with respect to our normalized $\KK[\sigma]$-basis.
\end{proof}

\begin{remark} \label{R:EaCM}
In the language of \cite{BCPW22}*{\S 4.2}, $H_{\ba}$ is a CM $t$-comotive with generalized CM type $(\bK,\Xi_{\ba})$. In this way $E_{\ba}$ is a $t$-module with CM by $\bB$ and generalized CM type $(\bK,\Xi_{\ba})$ as in~\cite{BCPW22}*{\S 8}.
\end{remark}

\subsection{Analytic properties of Sinha modules} \label{SS:analyticSinha}
Fix effective $\ba \in \cA_f$ with $\deg \ba > 0$ as in \S\ref{SS:Sinhadefs}. It was shown in \cite{ABP04}*{Lem.~6.4.3} that $H_{\ba}$ is rigid analytically trivial, and as such the $t$-module $E_{\ba}$ is uniformizable.
This was also demonstrated in \cite{BCPW22}*{\S 4.3} using Hilbert-Blumenthal $t$-modules. As we will need more refined information about the rigid analytic trivialization for $H_{\ba}$, we present a modified version of the constructions of~\cite{ABP04}*{\S 6.4} here. One should also compare to Wei~\cite{Wei22}*{\S 5.3}.

\begin{subsubsec}{Precursor to rigid analytic trivialization} \label{SSS:precursorRAT}
We begin by following the setup of~\cite{ABP04}*{Lem.~6.4.3} but pay attention when possible to fields of definition. The set $\{1, z, \dots, z^{r-1} \}$ is a $K[t]$-basis of $K[t,z]$. Since $g_{\ba} \in K[t,z]$ by Proposition~\ref{P:coldef} (as $\ba$ is effective), we can define $\Phi_{\ba} \in \Mat_{r}(K[t])$ as the matrix representing multiplication by~$g_{\ba}$ with respect to $\{1, z, \dots, z^{r-1} \}$. In other words, letting $\bz \assign (1,z, \dots, z^{r-1})^{\tr}$,
\begin{equation} \label{E:gaeigen}
g_{\ba}\cdot \bz \equiv \Phi_{\ba}\cdot \bz \pmod{\cD_f(t,z)}.
\end{equation}
Likewise, since $\{1, z, \dots, z^{r-1} \}$ is an $\FF_q[t]$-basis of $\bB=\FF_q[t,z]$, we can define $Z(t) \in \Mat_r(\FF_q[t])$ to represent multiplication by $z$ on $\bB$ with respect to this basis,
\begin{equation} \label{E:zeigen}
z\cdot \bz \equiv Z(t) \cdot \bz \pmod{\cD_f(t,z)}.
\end{equation}

Now as in Proposition~\ref{P:coldef} and the proof of Proposition~\ref{P:cGaf}, we can write
\begin{equation} \label{E:gasum}
g_{\ba} = 1 + \sum_{i=0}^M \sum_{j=0}^{r-1} c_{ij} t^i z^j, \quad c_{ij} \in B[f^{-1}],
\end{equation}
for some $M \geqslant 0$. It then follows from~\eqref{E:gaeigen}--\eqref{E:zeigen} that
\begin{equation} \label{E:Phiadef}
\Phi_{\ba} = \rI_r + \sum_{i=0}^M \sum_{j=0}^{r-1} c_{ij} t^i Z(t)^j \quad \in \Mat_r(B[f^{-1}][t]).
\end{equation}
By~\eqref{E:estarnorm} and~\eqref{E:normgaf-1}, we see that $\inorm{c_{ij}} \leqslant \inorm{\ttheta}^{-1} < 1$ for all $i$, $j$, and so $\dnorm{\Phi_{\ba} - \rI_r} < 1$. Thus we can define
\[
\Psi_{\ba} \assign \prod_{N=1}^{\infty} \Phi_{\ba}^{(N)}, \quad \Bigl( = \Phi_{\ba}^{(1)}\Phi_{\ba}^{(2)} \cdots \Bigr),
\]
which converges in $\Mat_r(\TT)$. We observe that $\Psi_{\ba}^{(-1)} = \Phi_{\ba} \Psi_{\ba}$, and so from~\cite{ABP04}*{Prop.~3.1.3} it follows that the entries of $\Psi_{\ba}$ are entire functions. Moreover, $\Psi_{\ba} \in \GL_r(\TT_{\theta}(k_{\infty}(\zeta)))$.
\end{subsubsec}

\begin{proposition} \label{P:PhiaPsiadiag}
For effective $\ba \in \cA_f$ with $\deg \ba > 0$, let $\Phi_{\ba}$ be defined as in~\eqref{E:Phiadef}. Then there exists $\Delta \in \GL_r(\bB[f(t)^{-1}])$ so that the following diagonalizations hold.
\begin{alphenumerate}
\item $\displaystyle \Delta^{-1}\Phi_{\ba}\Delta = \begin{pmatrix}
\ddots & & \\
& g_{b\star \ba} & \\
& & \ddots
\end{pmatrix}_{b \in (A/fA)^{\times}}$.
\item $\displaystyle \Delta^{-1}\Psi_{\ba}\Delta = \begin{pmatrix}
\ddots & & \\
& \cG_{b\star\ba} & \\
& & \ddots
\end{pmatrix}_{b \in (A/fA)^{\times}}$.
\end{alphenumerate}
\end{proposition}

\begin{proof}
For $b \in (A/fA)^{\times}$, let $z_b \assign \trho_b(z)$ in $\bK$. If we apply $\trho_b$ to the identity in~\eqref{E:zeigen}, we see that $(1,z_b, \dots, z_b^{r-1})^{\tr}$ is an eigenvector of $Z(t)$ with eigenvalue~$z_b$. We define $\Delta \in \GL_r(\bB[f(t)^{-1}])$ so that
\begin{equation} \label{E:Deltainv}
\Delta \assign \begin{pmatrix}
\cdots & 1 & \cdots \\ \cdots  & z_b & \cdots \\  & \vdots & \\ \cdots & z_b^{r-1} & \cdots  
\end{pmatrix}_{b \in (A/fA)^{\times}},
\end{equation}
where the columns are indexed by $b \in (A/fA)^{\times}$ in some order. To confirm that $\Delta \in \GL_r(B[f(t)^{-1}])$, we observe that $\det \Delta^2$ is the discriminant of the $\bA$-basis $\{1, z, \dots, z^{r-1}\}$ of $\bB$. In particular, it is invertible and its determinant is divisible only by ramified primes of $\bB$, namely those dividing $f(t)$. Therefore, we have the diagonalization
\[
\Delta^{-1} Z(t) \Delta = \diag \bigl( z_b \mid b \in (A/fA)^{\times} \bigr).
\]
Then~\eqref{E:gasum} and~\eqref{E:Phiadef} imply that
\[
\Delta^{-1} \Phi_{\ba} \Delta = \diag \bigl( \trho_b(g_{\ba}) \mid b \in (A/fA)^{\times} \bigr),
\]
and (a) follows from Proposition~\ref{P:gaGalois}(b). To prove (b) let $\Psi_R = \Phi_{\ba}^{(1)} \cdots \Phi_{\ba}^{(R)}$ for $R \geqslant 1$. Then (a) implies
\[
\Delta^{-1} \Psi_R \Delta = \prod_{N=1}^R \diag \bigl( g_{b\star \ba}^{(N)} \mid b \in (A/fA)^{\times} \bigr).
\]
As $R \to \infty$, the left-hand side has limit $\Delta^{-1} \Psi_{\ba} \Delta$ by the definition of $\Psi_{\ba}$. The right-hand side has the desired limit by Proposition~\ref{P:cGaf}.
\end{proof}

\begin{subsubsec}{Rigid analytic trivialization of $H_{\ba}$} \label{SSS:HaRAT}
As $H_{\ba}(\KK) = H^0(\bU/\KK,\cO_{\bX/\KK}(-W_{\ba}^{(1)}))$, we note that since the support of $W_{\ba}^{(1)}$ is $K^{(1)}$-rational (as are the points at infinity), we can find a $\KK[t]$-basis $\{n_1, \dots, n_r \}$ of $H_{\ba}(\KK)$ such that $\{n_1, \dots, n_r \} \subseteq K^{(1)}[t,z]$, (e.g., see~\cite{Rosen}*{Thm.~5.4}). Let $\bn \assign (n_1, \dots, n_r)^{\tr}$, and choose $\Phi \in \Mat_r(K^{\perf}[t])$ so that $\sigma \bn = \Phi\bn$. Since $\{ \sigma n_1, \dots, \sigma n_r\} \subseteq K[t,z]$, we see that under these choices,
\begin{equation} \label{E:Phirational}
\Phi \in \Mat_r(K[t]).
\end{equation}
Now there exists a unique matrix $Q \in \Mat_r(K^{(1)}[t])$ such that $\bn \equiv Q \cdot \bz \pmod{\cD_f(t,z)}$. As noted in~\cite{ABP04}*{Pf.~of Prop.~6.4.4}, $Q \in \GL_r(\TT_{\theta})$ and $\det Q(\theta) \neq 0$. On the one hand, $\sigma \bn = \Phi \bn \equiv \Phi Q \cdot \bz \pmod{\cD_f(t,z)}$, but on the other, using~\eqref{E:Phiadef},
\[
\sigma \bn \equiv \sigma Q \cdot \bz \equiv Q^{(-1)} g_{\ba} \cdot \bz \equiv Q^{(-1)} \Phi_{\ba} \cdot \bz \pmod{\cD_f(t,z)}.
\]
Therefore, $\Phi Q = Q^{(-1)} \Phi_{\ba}$. If we set $\Psi \assign Q\Psi_{\ba} \in \GL_r(\TT_{\theta})$, then
\[
\Psi^{(-1)} = Q^{(-1)}\Psi_{\ba}^{(-1)} = Q^{(-1)}\Phi_{\ba}\Psi_{\ba} = \Phi Q \Psi_{\ba} = \Phi\Psi,
\]
and so $\Psi$ is a rigid analytic trivialization for $H_{\ba}$ and has entries in $\TT_{\theta}(k_{\infty}(\zeta))$.
\end{subsubsec}

\begin{subsubsec}{Periods of $E_{\ba}$}
Using Anderson's exponentiation theorem (Theorem~\ref{T:Anderson}), we can calculate the period lattice $\Lambda_{\ba} \subseteq \C^d$ for $E_{\ba}$ in terms of special $\Pi$-values. We obtain an extension of one of Sinha's main theorems from~\cite{Sinha97b}*{\S 5.3.3, \S 5.3.9}. 
\end{subsubsec}

\begin{theorem} \label{T:periods}
Let $\ba \in \cA_f$ be effective with $\deg \ba > 0$. Let $E_{\ba}$ be constructed with respect to a normalized $\KK[\sigma]$-basis on $H_{\ba}(\KK)$. The period lattice $\Lambda_{\ba}$ of $E_{\ba}$ has an $\bA$-basis $\pi_1, \dots, \pi_r$ so that the following hold.
\begin{alphenumerate}
\item For $1 \leqslant i \leqslant r$,
\[
\pi_i = \begin{pmatrix}
\vdots \\ \pd_t^{\ell_b - j} \bigl( z^{i-1} \cG_{\ba}^{-1} \bigr)\big|_{\xi_b} \\ \vdots
\end{pmatrix}_{(b,j) \in \cS_{\ba}}.
\]
The coordinate corresponding to $(b,\ell_b) \in \cS_{\ba}$ is $\rho_b(\zeta)^{i-1} \Pi (b \star \ba)$ \textup{(cf.~\cite{Sinha97b}*{\S 5.3})}.
\item When $\ba = [a/f]$ is basic, for $1 \leqslant i \leqslant r$,
\[
\pi_i = \begin{pmatrix}
\vdots \\ \rho_b(\zeta)^{i-1} \Pi\bigl( [ba/f] \bigr) \\ \vdots
\end{pmatrix}_{b \in \cP_{\ba}}.
\]
\end{alphenumerate}
\end{theorem}

\begin{proof}
Fix a normalized $\KK[\sigma]$-basis $\{ h_{(a,j)} \mid (a,j) \in \cS_{\ba} \}$ of $H_{\ba}(\KK)$ as in \S\ref{SSS:normalized}, and let $E_{\ba}: \bA \to \Mat_d(\KK[\sigma])$ be the corresponding $t$-module. We fix all of the same notation from earlier in \S\ref{SS:analyticSinha}. As $\Psi = Q\Psi_{\ba}$ is a rigid analytic trivialization of $E_{\ba}$, if we take $\br \in \Mat_{1\times r}(\TT_{\theta})$ to be a row of $\Psi^{-1}$, we have $\br^{(-1)} \Phi - \br = 0$,
and so $\cE_0(\br) \in \Lambda_{\ba}$ by Theorem~\ref{T:Anderson}(a). Furthermore, as $\cE_0(\br)$ ranges over all rows of $\Psi^{-1}$, we obtain a $\rd E_{\ba}(\bA)$-basis of $\Lambda_{\ba}$ by Theorem~\ref{T:Anderson}(c).

Now the entries of $\bn = Q\bz \in \Mat_{r\times 1} (K^{(1)}[t,z])$ form a $\KK[t]$-basis of $H_{\ba}(\KK)$. If $\bsalpha = (\alpha_1, \dots, \alpha_r) \in \Mat_{1 \times r}(\C[t])$, then by \S\ref{SSS:E0}, $(\varepsilon_0 \circ \iota)(\bsalpha) = (\teps_0 \circ \tiota)(\bsalpha) = \teps_0 (\bsalpha \cdot Q \cdot \bz)$.
Extending $\teps_0$ continuously to $\tcE_0 : \BB_{\theta} \to \C$, we have for $\bsalpha = (\alpha_1, \dots, \alpha_r) \in \Mat_{1 \times r}(\TT_{\theta})$,
\[
\cE_0(\bsalpha) = \tcE_0 ( \bsalpha \cdot Q \cdot \bz ),
\]
as in~\eqref{E:E0}. Recalling from Proposition~\ref{P:PhiaPsiadiag}(b) that $\Psi = Q\Psi_{\ba} = Q\Delta \Gamma_{\ba}\Delta^{-1}$, where we set $\Gamma_{\ba} \assign \diag( \cG_{b \star \ba} \mid b \in (A/fA)^{\times} )$, we have
\[
\Psi^{-1} = \Delta \Gamma_{\ba}^{-1}\Delta^{-1} Q^{-1}.
\]
Therefore, if $\bs_i \in \Mat_{1 \times r}(\TT_{\theta})$ is the standard basis vector with $1$ in the $i$-th entry, then taking $\br_i = \bs_i\Psi^{-1}$, we set
\[
\pi_i \assign \cE_0(\br_i) = \tcE_0 \bigl( \bs_i \cdot \Delta \Gamma_{\ba}^{-1} \Delta^{-1} \cdot \bz\bigr),
\]
as defined in~\eqref{E:E0}. By \eqref{E:Deltainv}, if we index the columns of $\Delta$ so that the first one is~$\bz$ (and corresponds to $b=1 \in (A/fA)^{\times}$), we have $\Delta^{-1} \cdot \bz = (1, 0, \dots, 0)^{\tr}$. It follows that
\[
\pi_i = \tcE_0 \Bigl( \bs_i \bigl( 
\cG_{\ba}^{-1}, z\cG_{\ba}^{-1}, \dots, z^{r-1} \cG_{\ba}^{-1} \bigr)^{\tr} \Bigr) = \tcE_0 \bigl( z^{i-1}\cG_{\ba}^{-1} \bigr).
\]
Thus~\eqref{E:E0} implies the formula for $\pi_i$ in~(a) as desired. The coordinates of $\pi_i$ corresponding to $(b,\ell_b) \in \cS_{\ba}$ are $(z^{i-1}\cG_{\ba}^{-1})|_{\xi_b} = \rho_b(\zeta)^{i-1} \Pi(b\star \ba)$ from \eqref{E:Pibstara}.
\end{proof}

\begin{remark}
Sinha's identities in~\cite{Sinha97b}*{\S 5.3.3, 5.3.9} apply to the case $\ba = [1/f]$, and if we compare his formulas to Theorem~\ref{T:periods}(b), we see that his period is almost exactly~$\pi_1$, with only a discrepancy of a factor $b/f$ when $\deg b = \deg f - 1$.
\end{remark}

\subsection{Fields of definition} \label{SS:Fodef}
Again we fix effective $\ba \in \cA_f$ with $\deg \ba > 0$. The goal of this section is to show that the associated $t$-module $E_{\ba}$ has a model defined over $K$. As previously observed, we see from \citelist{\cite{ABP04}*{\S 6.4} \cite{BCPW22}*{Thm.~4.2.2}}, using arguments that go back to~\cite{Sinha97b}*{Thm.~3.2.5}, that $H_{\ba}(\ok)$ has a $\ok[\sigma]$-basis and that therefore $E_{\ba}$ has a model over~$\ok$. However, in all of these arguments the $\ok[\sigma]$-basis is not explicitly constructed, but rather relies on an existential result of Anderson (see~\citelist{\cite{And86}*{Lem.~1.4.5} \cite{ABP04}*{Prop.~4.3.2}}) to show that it suffices to prove that $H_{\ba}(\ok)$ is finitely generated over~$\ok[\sigma]$.

\begin{subsubsec}{The definability problem over $K$}
To show that $E_{\ba}$ has a model defined over $K$, we first establish that $H_{\ba}(\KK)$ has a (normalized) $\KK[\sigma]$-basis consisting of functions in $K[t,z]$ (see Proposition~\ref{P:HaKdef}). Then we analyze the construction of $E_{\ba}$ with respect to this basis as in~\eqref{E:Etdef} to show that this model of $E_{\ba}$ is defined over~$K$ (see Theorem~\ref{T:EaKdef}).

We make two definitions:
\begin{itemize}
\item A $\KK$-subspace of $\KK[t,z]$ (in particular a subspace of $H_{\ba}(\KK)$) is \emph{defined over $K$ as a $\KK$-vector space} if it has a $\KK$-basis contained in $K[t,z]$.
\item A free left $\KK[\sigma]$-submodule of $H_{\ba}(\KK)$ is \emph{defined over $K$ as a $\KK[\sigma]$-module} if it has a $\KK[\sigma]$-basis contained in $K[t,z]$.
\end{itemize}
We note that a free $\KK[\sigma]$-submodule may be defined over $K$ as a $\KK[\sigma]$-module but not as a $\KK$-vector space due to the action of $\sigma$. Nevertheless, when it will not cause confusion we will refer to $\KK$-subspaces and $\KK[\sigma]$-submodules as ``defined over $K$'' without additional qualifiers when it is clear which situation we are in from the context.
\end{subsubsec}

\begin{proposition} \label{P:HaKdef}
For $\ba \in \cA_f$ effective with $\deg \ba > 0$, the following hold.
\begin{alphenumerate}
\item $H_{\ba}(\KK)$ is defined over $K$ as a $\KK[\sigma]$-module.
\item $H_{\ba}(\KK)$ has a normalized $\KK[\sigma]$-basis defined over $K$.
\end{alphenumerate}
\end{proposition}

\begin{remark}
(i) A major issue with showing that $H_{\ba}(\KK)$ is defined over~$K$ is that~$K$ is not perfect. For example, one can use the arguments in \citelist{\cite{ABP04}*{\S 6.4} \cite{BCPW22}*{\S 4.2}} to show that $H_{\ba}(\KK)$ is defined over $K^{\perf}$, but descending to $K$ itself is more tricky.
(ii) The proof of Proposition~\ref{P:HaKdef} takes several steps. It is essentially a more laborious version of the proofs of \citelist{\cite{ABP04}*{Lem.~6.4.1} \cite{BCPW22}*{Thm.~4.2.2} \cite{Sinha97b}*{Thm.~3.2.5}}, but where we pay attention at each stage to obtaining bases of $\KK$-subspaces and $\KK[\sigma]$-submodules defined over $K$.
\end{remark}

\begin{subsubsec}{Riemann-Roch spaces}
For a divisor $D \in \Div_{\KK}(\bX)$, we define
\[
\cL(D) \assign H^0(\bX,\cO_{\bX}(D)),
\]
which is a finite dimensional $\KK$-vector space. We note that if $D \in \Div_K(\bX)$ and $\cL(D) \subseteq \KK[t,z]$, then $\cL(D)$ is defined over $K$ by \cite{Rosen}*{Thm.~5.4}. For $i \in \ZZ$, we set
\begin{equation} \label{E:Midef}
M_i \assign \cL(-W_{\ba}^{(1)} + i I_{\ba}) \subseteq H_{\ba}(\KK),
\end{equation}
which is then finite dimensional over $\KK$ and defined over~$K$. 
\end{subsubsec}

\begin{lemma} \label{L:Miprops}
The following hold.
\begin{alphenumerate}
\item For $i < 0$, $M_i=\{0\}$. We have $M_0 \subseteq M_1 \subseteq M_2 \subseteq \cdots$, and $H_{\ba}(\KK) = \cup_{i=0}^{\infty} M_i$.
\item For $i$, $j \in \ZZ$, $M_i^{(-j)} = \cL(-W_{\ba}^{(1-j)} + i I_{\ba})$.
\item For $i \in \ZZ$ and $j \geqslant 1$,
\begin{align*}
\sigma^j M_{i} &= g_{\ba}g_{\ba}^{(-1)} \cdots g_{\ba}^{(-j+1)} M_{i}^{(-j)} \\
&= \cL(-W_{\ba}^{(1)} - \Xi_{\ba} - \Xi_{\ba}^{(-1)} - \cdots - \Xi_{\ba}^{(-j+1)} + (i + j)I_{\ba}).
\end{align*}
\item For $i$, $k \in \ZZ$ and $j$, $\ell \geqslant 0$,
\[
\sigma^j M_i \cap \sigma^{\ell} M_k
= \cL \bigl(-W_{\ba}^{(1)} - \Xi_{\ba} - \Xi_{\ba}^{(-1)} - \cdots - \Xi_{\ba}^{(-\max(j,\ell)+1)} + \min(i+j,k+\ell)I_{\ba} \bigr).
\]
\item For $i$, $k \in \ZZ$ and $j$, $\ell \geqslant 0$,
\[
\sigma^j M_i \cap \sigma^\ell M_k = \begin{cases}
\{0\}, & \text{if $\max(j,\ell) > \min(i+j,k+\ell)$,} \\
\sigma^{\max(j,\ell)} M_{\min(i+j,k+\ell) - \max(j,\ell)}, & \text{if $\max(j,\ell) \leqslant \min(i+j,k+\ell)$.}
\end{cases}
\]
Equivalently, for $i$, $k \in \ZZ$ and $\ell \geqslant 0$, $M_i \cap \sigma^{\ell} M_k = \sigma^{\ell} M_{\min(i,k+\ell) - \ell}$. 
\item For $i \in \ZZ$ and $j \geqslant 0$,
$M_i + \sigma M_i + \cdots + \sigma^j M_i \subseteq M_{i+j}$.
\item There exists $N \geqslant 0$ such that for $i > N$, $M_{i} = M_{i-1} + \sigma M_{i-1}$.
\end{alphenumerate}
\end{lemma}

\begin{proof}
Parts (a)--(c) are straightforward applications of the Riemann-Roch theorem together with the definition of $H_{\ba}(\KK)$ and its $\sigma$-action. Part (d) follows from (c) and the definition of $\cL(D)$. The identity in the second case of (e) also holds for the first, as $M_i = \{0\}$ if $i< 0$, and this identity follows from (d). The equivalence in (e) is then a simple calculation, using that if $\ell \geqslant j$, then $\sigma^j M_i \cap \sigma^{\ell} M_k =  \sigma^{j} (M_i \cap \sigma^{\ell-j}M_k)$. Part (f) follows from (c). For (g), we note that (e) implies
\[
M_{i-1} \cap \sigma M_{i-1} = \sigma M_{i-2}.
\]
Therefore, as $\dim_{\KK} M_{i-1} = \dim_{\KK} \sigma M_{i-1}$, we have $\dim_{\KK} (M_{i-1} + \sigma M_{i-1}) = 2 \dim_{\KK} M_{i-1} - \dim_{\KK} M_{i-2}$. Now the Riemann-Roch theorem implies that there exists $N \geqslant 0$ such that for $i > N$ we have $\dim_{\KK} M_{i-2} = -\deg W_{\ba}^{(1)} + (i-2) \deg I_{\ba} - \gamma + 1$, where $\gamma$ is the genus of $\bX$. In particular,
$\dim_{\KK} M_{i-1} = \dim_{\KK} M_{i-2} + \deg I_{\ba}$. Thus for $i > N$,
\[
\dim_{\KK}(M_{i-1} + \sigma M_{i-1}) = \dim_{\KK} M_{i-1} + \deg I_{\ba} = \dim_{\KK} M_{i}.
\]
As $M_{i-1} + \sigma M_{i-1} \subseteq M_{i}$ by (f), we are done.
\end{proof}

\begin{remark}
We recall that Lemma~\ref{L:Miprops}(g) is one of the essential tools in showing that $H_{\ba}(\KK)$ is finitely generated over $\KK[\sigma]$ in \citelist{\cite{ABP04}*{\S 6.4} \cite{BCPW22}*{Thm.~4.2.2}}. It will also play a crucial role for us at the end of the proof of Proposition~\ref{P:HaKdef}.
\end{remark}

\begin{criteria} \label{Cr:Mi}
For more flexible use in \S\ref{SS:reduction}, we formalize some of these properties. Let us assume we have a sequence $\{ M_i \}_{i \in \ZZ}$ of finite dimensional $\KK$-vector subspaces of $H_{\ba}(\KK)$ satisfying the following.
\begin{alphenumerate}
\item For $i < 0$, $M_i = \{0 \}$. We have $M_0 \subseteq M_1 \subseteq M_2 \subseteq \cdots$, and $H_{\ba}(\KK) = \cup_{i=0}^{\infty} M_i$.
\item For $i$, $k \in \ZZ$ and $\ell \geqslant 0$,
\[
M_i \cap \sigma^{\ell} M_k = \sigma^{\ell} M_{\min(i,k+\ell) - \ell}.
\]
\item For $i \in \ZZ$ and $j \geqslant 0$, $M_i + \sigma M_i + \cdots + \sigma^j M_i \subseteq M_{i+j}$.
\item There exists $N \geqslant 0$ such that for $i > N$, $M_i = M_{i-1} + \sigma M_{i-1}$.
\end{alphenumerate}
That is, $\{ M_i \}$ satisfies the properties in Lemma~\ref{L:Miprops}(a), (e)--(g), and in particular our original choice of $\{ M_i \}$ in \eqref{E:Midef} satisfies these criteria.
\end{criteria}

\begin{subsubsec}{Definitions of spaces and bases} \label{SSS:UVTSdefs}
Fix a sequence $\{ M_i \}$ of $\KK$-vector subspaces of $H_{\ba}(\KK)$ as in Criteria~\ref{Cr:Mi}.
For $i \geqslant 0$, we recursively construct $\KK$-vector spaces,
\[
U_i \subseteq V_i \subseteq M_i,
\]
together with respective $\KK$-bases $T_i \subseteq U_i$ and $T_i \subseteq S_i \subseteq V_i$.

We set $U_0\assign V_0 \assign M_0$. If $M_0 = \{ 0 \}$, we set $S_0 \assign T_0 \assign \emptyset$. Otherwise, we let $S_0 \assign T_0 \assign \{ u_{0,1}, \dots, u_{0,s_0} \}$ be any $\KK$-basis for $M_0$.

For $i > 0$, if $M_i = M_{i-1} + \sigma M_{i-1}$, we set $T_i \assign \emptyset$. Otherwise we let $T_i$ be a set of $\KK$-linearly independent elements of $M_i$ such that $T_i$ induces a $\KK$-basis of $M_i/(M_{i-1} + \sigma M_{i-1})$. We write $T_i = \{u_{i,1}, \dots, u_{i,s_i} \}$, and we set
\[
U_i \assign \Span_{\KK}(T_i), \quad V_i \assign V_{i-1} + U_i, \quad S_i \assign S_{i-1} \cup T_i.
\]
Since $M_i = (M_{i-1} + \sigma M_{i-1}) \oplus U_i$ by construction, we see that the sum $V_{i-1} \oplus U_i$ is direct (as $V_{i-1} \subseteq M_{i-1}$), and so $S_i$ is a $\KK$-basis of $V_i$. Moreover,
\begin{equation} \label{E:ViUl}
V_i = \bigoplus_{\ell=0}^i U_\ell,
\end{equation}
where we will use ``$\oplus$'' to indicate that an internal sum is direct.
\end{subsubsec}

\begin{remark} \label{R:Kdef}
If $M_0$ is defined over $K$, then we can find $T_0= S_0 \subseteq K[t,z]$.
Suppose further that for each $i > 0$, both $M_i$ and $M_{i-1} + \sigma M_{i-1}$ are defined over $K$. Then $M_i/(M_{i-1} + \sigma M_{i-1})$ is defined over~$K$, and so at each stage we can choose $T_i \subseteq K[t,z]$. Therefore, under these additional conditions we can select $T_i \subseteq S_i \subseteq K[t,z]$ for all $i \geqslant 0$.
\end{remark}

\begin{lemma} \label{L:Silinind}
For $i \geqslant 0$, the following hold.
\begin{alphenumerate}
\item $S_i$ is a $\KK[\sigma]$-linearly independent subset of $H_{\ba}(\KK)$.
\item $\displaystyle M_i = \bigoplus_{\ell=0}^i \bigoplus_{j=0}^{i-\ell} \sigma^j U_{\ell} = \bigoplus_{j=0}^i \sigma^{i-j}V_j$.
\end{alphenumerate}
\end{lemma}

\begin{proof}
We proceed by induction on $i$, and consider first the case $i=0$. Then (b) holds automatically since $M_0=U_0=V_0$. If $S_0 = \emptyset$, then (a) holds trivially. So we suppose $S_0 \neq \emptyset$, and suppose
\begin{equation} \label{E:i=0basesum}
\alpha_1 u_{0,1} + \cdots + \alpha_{s_0} u_{0,s_0} = 0, \quad \alpha_1, \dots, \alpha_{s_0} \in \KK[\sigma].
\end{equation}
For a finite subset $X \subseteq \KK[\sigma]$, we set
\[
\max\deg_{\sigma}(X) \assign \max_{\alpha \in X} \bigl( \deg_\sigma (\alpha) \bigr).
\]
We proceed by induction on $\max\deg_{\sigma}(\alpha_1, \dots, \alpha_{s_0})$. Suppose $\max \deg_{\sigma}(\alpha_1, \dots, \alpha_{s_0}) \leqslant 0$, in which case $\alpha_n \in \KK$ for all $1 \leqslant n \leqslant s_0$. The $\KK$-linear independence of $S_0$ implies $\alpha_1 = \cdots = \alpha_{s_0} = 0$. Next suppose that $\max \deg_{\sigma}(\alpha_1, \dots, \alpha_{s_0}) \leqslant j$ for $j > 0$. For each $n$ we write
\[
\alpha_n = \beta_n + b_n \sigma^j, \quad \deg_{\sigma}(\beta_n) \leqslant j-1,\ b_n \in \KK.
\]
Then \eqref{E:i=0basesum} implies
\begin{equation} \label{E:i=0second}
\beta_1 u_{0,1} + \dots + \beta_{s_0} u_{0,s_0} = -\sigma^j \bigl( b_1^{q^j} u_{0,1} + \dots + b_{s_0}^{q^j}u_{0,s_0} \bigr)
\end{equation}
The left-hand side of this expression is in
\[
U_0 + \sigma U_0 + \dots + \sigma^{j-1} U_0 =
M_0 + \sigma M_0 + \dots + \sigma^{j-1}M_0 \subseteq M_{j-1},
\]
where the final containment follows from Criterion~\ref{Cr:Mi}(c). The right-hand side of \eqref{E:i=0second} is in $\sigma^j U_0 = \sigma^j M_0$. By Criterion~\ref{Cr:Mi}(b), $M_{j-1} \cap \sigma^j M_0 = \{ 0 \}$, so both sides of \eqref{E:i=0second} are~$0$. By the induction hypothesis, $\beta_1 = \dots = \beta_{s_0} = 0$, and the $\KK$-linear independence of $S_0$ implies that $b_1 = \dots = b_{s_0} = 0$. Thus $S_0$ is $\KK[\sigma]$-linearly independent.

Now we consider the case $i > 0$. The induction hypothesis implies two things:
\begin{equation} \label{E:IH1}
\bigoplus_{j=0}^{\infty} \sigma^j V_m = \bigoplus_{j=0}^{\infty} \biggl( \bigoplus_{\ell=0}^m \sigma^j U_{\ell} \biggr) = \bigoplus_{\ell=0}^m \bigoplus_{j=0}^{\infty} \sigma^j U_{\ell}, \quad (0 \leqslant m \leqslant i-1),
\end{equation}
where all of these internal sums of $\KK$-vector spaces are direct;
also,
\begin{equation} \label{E:IH2}
M_m = \bigoplus_{\ell=0}^m \bigoplus_{j=0}^{m-\ell} \sigma^j U_{\ell}, \quad (0 \leqslant m \leqslant i-1).
\end{equation}
Notably this second identity implies that
\begin{equation} \label{E:Mi-1sigmaMi-1}
M_{i-1} + \sigma M_{i-1} =
\biggl( \bigoplus_{\ell=0}^{i-1} \bigoplus_{j=0}^{i-1-\ell} \sigma^j U_{\ell} \biggr) + \biggl( \bigoplus_{\ell=0}^{i-1} \bigoplus_{j=1}^{i-\ell} \sigma^j U_{\ell} \biggr)
= \bigoplus_{\ell=0}^{i-1} \bigoplus_{j=0}^{i-\ell} \sigma^j U_{\ell},
\end{equation}
where the final term is a direct sum by~\eqref{E:IH1}. Now because $M_i = (M_{i-1} + \sigma M_i) \oplus U_i$ by the definition of~$U_i$, we see that the first identity in part~(b) for $i$ follows immediately, and the second from~\eqref{E:ViUl}. The rest of the proof is devoted to proving~(a).

Suppose that $T_i = \emptyset$. In this case $S_i = S_{i-1}$ and $U_i=\{0\}$, and then (a) holds by the induction hypothesis. 
Now assume that $T_i = \{u_{i,1}, \dots, u_{i,s_i} \} \neq \emptyset$, and suppose that
\begin{equation} \label{E:ibasesum}
\sum_{m=0}^i \sum_{n=1}^{s_m} \alpha_{m,n} u_{m,n} = 0, \quad \alpha_{m,n} \in \KK[\sigma].
\end{equation}
Here we use the convention that $s_m=0$ if $T_m = \emptyset$, whereby the corresponding sum is empty. We make the following \textbf{claim}: if for $j \geqslant i$,
\begin{equation} \label{E:jcondition}
\max\deg_{\sigma}(\alpha_{m,n} \mid 1 \leqslant n \leqslant s_m) \leqslant j-m, \quad \forall\,m,\ 0 \leqslant m \leqslant i,
\end{equation}
then $\alpha_{m,n}=0$ for all $0 \leqslant m \leqslant i$, $1 \leqslant n \leqslant s_m$.

We prove the claim by induction on $j$. For the base case we assume $j=i$, so that $\deg_{\sigma} \alpha_{m,n} \leqslant i-m$. Then the left-hand side of~\eqref{E:ibasesum} is an element of $\oplus_{m=0}^i \oplus_{\ell=0}^{i-m} \sigma^{\ell} U_m$, which is direct by the proof of (b) for $i$ above. Thus $\alpha_{m,n} = 0$ for all $0 \leqslant m \leqslant i$, $1 \leqslant n \leqslant s_m$. Now take $j > i$, and assume the claim holds up to $j-1$. First suppose
\begin{gather} \label{E:ijfirstconditions}
\max \deg_{\sigma}( \alpha_{0,1}, \dots, \alpha_{0,s_0}) \leqslant j, \\
\max \deg_{\sigma}(\alpha_{m,n} \mid 1 \leqslant n \leqslant s_m) \leqslant j-1-m, \quad \forall\,m,\ 1 \leqslant m \leqslant i. \notag
\end{gather}
For $1 \leqslant n \leqslant s_0$, write $\alpha_{0,n} = \beta_n + b_n \sigma^j$, $\deg_{\sigma}(\beta_n) \leqslant j-1$, $b_n \in \KK$. We rewrite \eqref{E:ibasesum} as
\begin{equation} \label{E:ijsecondsum}
\sum_{n=1}^{s_0} \beta_n u_{0,n} + \sum_{m=1}^i \sum_{n=1}^{s_m} \alpha_{m,n} u_{m,n} = - \sigma^j \biggl( \sum_{n=1}^{s_0} b_n^{q^j} u_{0,n} \biggr).
\end{equation}
Using that $U_m \subseteq M_m$ together with Criterion~\ref{Cr:Mi}(c), we find that the left-hand side is in $M_{j-1}$. The right-hand side is in $\sigma^j M_0$. Thus the expressions on both sides of~\eqref{E:ijsecondsum} are in $M_{j-1} \cap \sigma^j M_0 = \{0\}$ (using Criterion~\ref{Cr:Mi}(b)). Therefore, the induction hypothesis implies that $\beta_n=0$ for $1 \leqslant n \leqslant s_0$ and that $\alpha_{m,n}=0$ for $1 \leqslant m \leqslant i$, $1 \leqslant n \leqslant s_m$. Likewise the $\KK$-linear independence of $u_{0,1}, \dots, u_{0,s_0}$ implies that $b_n=0$ for $1 \leqslant n \leqslant s_0$. We conclude that under the conditions in~\eqref{E:ijfirstconditions}, the claim is proved.

We now suppose that $1 \leqslant \ell \leqslant i$. We will show that if the claim holds when
\begin{gather} \label{E:ijlconditions}
\max \deg_{\sigma}(\alpha_{m,n} \mid 1 \leqslant n \leqslant s_m) \leqslant j-m, \quad \forall\,m,\ 0 \leqslant m \leqslant \ell-1, \\
\max \deg_{\sigma}(\alpha_{m,n} \mid 1 \leqslant n \leqslant s_m) \leqslant j-1-m, \quad \forall\,m,\ \ell \leqslant m \leqslant i, \notag
\end{gather}
then it also holds when
\begin{gather} \label{E:ijl+1conditions}
\max \deg_{\sigma}(\alpha_{m,n} \mid 1 \leqslant n \leqslant s_m) \leqslant j-m, \quad \forall\,m,\ 0 \leqslant m \leqslant \ell, \\
\max \deg_{\sigma}(\alpha_{m,n} \mid 1 \leqslant n \leqslant s_m) \leqslant j-1-m, \quad \forall\,m,\ \ell+1 \leqslant m \leqslant i. \notag
\end{gather}
Once we show this then the claim is proved, thus concluding the induction on $j$ as well as the induction on $i$ and completing the proof of part~(a).

We proceed as in the previous cases. For $0 \leqslant m \leqslant \ell$ and $1 \leqslant n \leqslant s_m$, we write $\alpha_{m,n} = \beta_{m,n} + b_{m,n} \sigma^{j-m}$, $\deg_{\sigma} \beta_{m,n} \leqslant j-m-1$, $b_{m,n} \in \KK$. Rearranging~\eqref{E:ibasesum}, we have
\begin{equation} \label{E:ijlsecondsum}
\sum_{m=0}^{\ell} \sum_{n=1}^{s_m} \beta_{m,n} u_{m,n} + \sum_{m=\ell+1}^i \sum_{n=1}^{s_m} \alpha_{m,n} u_{m,n} = -\sum_{m=0}^{\ell} \sigma^{j-m} \biggl( \sum_{n=1}^{s_m} b_{m,n}^{q^{j-m}} u_{m,n} \biggr).
\end{equation}
As in previous cases, the left-hand side is in~$M_{j-1}$, and the right is visibly in
\[
\sigma^{j-\ell} \cdot \sum_{m=0}^{\ell} \sigma^{\ell-m} M_{m} \subseteq \sigma^{j-\ell} M_{\ell}, \quad \textup{(Criterion~\ref{Cr:Mi}(c)).}
\]
Thus both sides of~\eqref{E:ijlsecondsum} are in
\[
M_{j-1} \cap \sigma^{j-\ell} M_{\ell} = \sigma^{j-\ell} M_{\ell-1}, \quad \textup{(Criterion~\ref{Cr:Mi}(b)).}
\]
Since $\ell - 1 \leqslant i-1$, we see from~\eqref{E:IH2} that
\begin{equation} \label{E:final1}
\sigma^{j-\ell} M_{\ell-1} = \sigma^{j-\ell} \Biggl(
\biggl( \bigoplus_{\lambda=0}^{\ell-1} \sigma^{\lambda} U_0 \biggr) \oplus \biggl( \bigoplus_{\lambda=0}^{\ell-2} \sigma^{\lambda} U_{1} \biggr) \oplus \cdots \oplus \bigl( U_{\ell-2} \oplus \sigma U_{\ell-2} \bigr) \oplus U_{\ell-1} \Biggr).
\end{equation}
On the other hand, the right-hand side of~\eqref{E:ijlsecondsum} is in
\begin{equation} \label{E:final2}
\sigma^{j-\ell} \bigl( \sigma^{\ell}U_0 \oplus \sigma^{\ell-1} U_1 \oplus \cdots \oplus \sigma U_{\ell-1} \oplus U_{\ell} \bigr).
\end{equation}
If $\ell \leqslant i-1$, then this sum is indeed direct, as all terms are then sub-terms of~\eqref{E:IH1}. If $\ell=i$, then this sum is direct by our previous proof of part~(b) for~$i$. When we compare the terms that appear in~\eqref{E:final1} and~\eqref{E:final2}, there is no overlap, so it must be that both sides of~\eqref{E:ijlsecondsum} are~$0$, and then the vanishing of coefficients follows as in previous cases. Thus, we have completed the proof of~(a).
\end{proof}

\begin{remark} \label{R:SiKdef}
Returning to the case that $M_i = \cL(-W_{\ba}^{(1)} + i I_{\ba})$ as in~\eqref{E:Midef}, we have already observed each $M_i$ is defined over $K$. By Lemma~\ref{L:Miprops}(c), for $i > 0$,
\[
M_{i-1} + \sigma M_{i-1} = \cL(-W_{\ba}^{(1)} + (i-1)I_{\ba}) + \cL(-W_{\ba}^{(1)} - \Xi_{\ba} + iI_{\ba}),
\]
and so it follows that $M_{i-1} + \sigma M_{i-1}$ is defined over $K$. By Remark~\ref{R:Kdef}, we can choose $T_i \subseteq K[t,z]$ for each $i \geqslant 0$. Therefore, in this case the $\KK[\sigma]$-linear span of $S_i$ is defined over $K$ as a $\KK[\sigma]$-module.
\end{remark}

\begin{proof}[Proof of Proposition~\ref{P:HaKdef}]
Let $N \geqslant 0$ be chosen as in Criterion~\ref{Cr:Mi}(d). By construction then, for $i > N$ we have $U_i=\{0\}$ and $V_i = V_{i-1}$. We claim that $S_N$ is a $\KK[\sigma]$-basis of $H_{\ba}(\KK)$. By Lemma~\ref{L:Silinind}(a), $S_N$ is $\KK[\sigma]$-linearly independent. From Lemma~\ref{L:Silinind}(b), we find that for $i > N$,
\[
M_i = \bigoplus_{j=0}^N \sigma^{i-j} V_j \oplus \bigoplus_{j=N+1}^{i} \sigma^{i-j} V_N\subseteq \bigoplus_{j=0}^i \sigma^{i-j}V_N,
\]
and as $H_{\ba}(\KK) = \cup_{i=0}^{\infty} M_i$, we see that $V_N$, and hence $S_N$, generate $H_{\ba}(\KK)$ as $\KK[\sigma]$-modules.

As noted in Remark~\ref{R:SiKdef}, if we take $\{ M_i \}$ defined as in~\eqref{E:Midef}, then we can choose the elements of $S_N$ to be in $K[t,z]$, and so $H_{\ba}(\KK)$ is defined over~$K$. Furthermore, in the construction of a normalized $\KK[\sigma]$-basis in~\S\ref{SSS:normalized}, \S\ref{SSS:Wronskian}, if the given basis $\{ \tilh_1, \dots, \tilh_d \}$ is defined over $K$, then their values at the points $\xi_a$ are all in $K$, and thus the linear changes of variables to obtain the normalized $\{ h'_{(a,j)} \}_{a \in \cS_{\ba}}$ are defined over $K$.
\end{proof}

\begin{remark} \label{R:algorithm}
The proofs of Proposition~\ref{P:HaKdef} and Lemma~\ref{L:Silinind} show that the recursive definitions of $S_0 \subseteq S_1 \subseteq \cdots \subseteq S_N$ in \S\ref{SSS:UVTSdefs} provide an algorithm for calculating $\KK[\sigma]$-bases of $H_{\ba}(\KK)$.
\end{remark}

\begin{theorem} \label{T:EaKdef}
Let $\ba \in \cA_f$ be effective with $\deg \ba > 0$. Then $E_{\ba}$ has a model defined over $K$.
\end{theorem}

\begin{proof}
We may as well switch to the case that $\KK=\ok$. By Proposition~\ref{P:HaKdef}, we can choose a $\ok[\sigma]$-basis of $H_{\ba}(\ok)$ defined over $K$. Since transforming such a basis to a normalized one can be done over $K$, as in \S\ref{SSS:Wronskian}, it suffices to assume our basis is normalized. To that end, let $\{h_{(a,j)} \mid (a,j) \in \cS_{\ba} \} \subseteq H_{\ba}(\ok)$ be a normalized $\ok[\sigma]$-basis defined over~$K$. We wish to determine the matrix $Q^*$ as in~\eqref{E:Etdef} and show that the entries of $Q^*$ are all in $K[\tau]$. Fix $(a_0,j_0) \in \cS_{\ba}$. Then
\begin{equation} \label{E:tha0j0}
t \cdot h_{(a_0,j_0)} = \sum_{i \geqslant 0} \sum_{(a,j) \in \cS_{\ba}} \alpha_{i,(a,j)} \sigma^i h_{(a,j)}
= \sum_{i \geqslant 0} \sum_{(a,j) \in \cS_{\ba}}\alpha_{i,(a,j)} g_{\ba} g_{\ba}^{(-1)} \cdots g_{\ba}^{(-i+1)} h_{(a,j)}^{(-i)},
\end{equation}
for unique $\alpha_{i,(a,j)} \in \ok$, almost all of which are~$0$. We will be done after proving the following \textbf{claim}: for $i \geqslant 0$ and $(a,j) \in \cS_{\ba}$, we have $\alpha_{i,(a,j)}^{(i)} \in K$.

We proceed by induction on $i$. When $i=0$, Lemma~\ref{L:normdecomp} implies for $(a,j) \in \cS_{\ba}$ that
\[
\alpha_{0,(a,j)} = \pd_t^{\ell_a-j} \bigl( t \cdot h_{(a_0,j_0)} \bigr) \big|_{\xi_a},
\]
which is an element of $K$ since $t \cdot h_{(a_0,j_0)} \in K[t,z]$ and differentiation with respect to $t$ leaves $K(t,z)$ invariant, as in \S\ref{SSS:hyper}. This concludes the base case.

Now let $m \geqslant 1$ and suppose that $\alpha_{i,(a,j)}^{(i)} \in K$ for all $0 \leqslant i \leqslant m-1$ and $(a,j) \in \cS_{\ba}$. Let $b \in \cP_{\ba}$. For $0 \leqslant \ell \leqslant \ell_b$, we apply $\pd_t^\ell$ to~\eqref{E:tha0j0} and obtain,
\[
\pd_t^{\ell} \bigl(t\cdot h_{(a_0,j_0)} \bigr)
= \sum_{a \in \cP_{\ba}} \sum_{j=0}^{\ell_a} \sum_{i \geqslant 0} \alpha_{i,(a,j)} \pd_t^{\ell} \bigl( g_{\ba} g_{\ba}^{(-1)}\cdots g_{\ba}^{(-i+1)} h_{(a,j)}^{(-i)} \bigr).
\]
For $i \geqslant 0$, we note that $g_{\ba} g_{\ba}^{(-1)} \cdots g_{\ba}^{(-i+1)} h_{(a,j)}^{(-i)}$
vanishes along the divisor $W_{\ba}^{(1)} + \Xi_{\ba} + \Xi_{\ba}^{(-1)} + \dots + \Xi_{\ba}^{(-i+1)}$. Therefore, when $i > m$, we have $\ord_{\xi_b}(g_{\ba}^{(m)} \cdots g_{\ba}^{(m-i+1)}h_{(a,j)}^{(m-i)}) > \ell_b$, which implies that
$\pd_t^{\ell}(g_{\ba}^{(m)} \cdots g_{\ba}^{(m-i+1)}h_{(a,j)}^{(m-i)})|_{\xi_b} = 0$. Thus, using the product rule, twisting $m$ times, and evaluating at~$\xi_b$, we obtain
\begin{equation} \label{E:mtwist2}
\pd_t^{\ell} \bigl( t \cdot h_{(a_0,j_0)}^{(m)} \bigr) \big|_{\xi_b}
= \sum_{i=0}^m \sum_{a \in \cP_{\ba}} \sum_{j=0}^{\ell_a} \alpha_{i,(a,j)}^{(m)} \sum_{k=0}^{\ell} \pd_t^{\ell-k}\bigl( g_{\ba}^{(m)} \cdots g_{\ba}^{(m-i+1)}\bigr) \pd_t^{k} \bigl(h_{(a,j)}^{(m-i)} \bigr) \big|_{\xi_b}.
\end{equation}
Before evaluating we have used that differentiation with respect to $t$ commutes with Frobenius twisting. Again differentiation with respect to $t$ leaves $K(t,z)$ invariant, so the left-hand side of this equation is in $K$. Likewise, by the induction hypothesis,
\[
\sum_{i=0}^{m-1} \sum_{a \in \cP_{\ba}} \sum_{j=0}^{\ell_a} \alpha_{i,(a,j)}^{(m)} \sum_{k=0}^{\ell} \pd_t^{\ell-k}\bigl( g_{\ba}^{(m)} \cdots g_{\ba}^{(m-i+1)}\bigr) \pd_t^{k} \bigl(h_{(a,j)}^{(m-i)} \bigr) \big|_{\xi_b} \in K.
\]
The remaining terms to account for in \eqref{E:mtwist2} have $i=m$, and are thus
\begin{equation} \label{E:remain}
\sum_{a \in \cP_{\ba}} \sum_{j=0}^{\ell_a} \alpha_{m,(a,j)}^{(m)} \sum_{k=0}^{\ell} \pd_t^{\ell-k}\bigl( g_{\ba}^{(m)} \cdots g_{\ba}^{(1)}\bigr) \pd_t^{k} \bigl(h_{(a,j)} \bigr) \big|_{\xi_b}.
\end{equation}
If $a \neq b$, then as observed in the proof of Lemma~\ref{L:normdecomp}, $\ord_{\xi_b}(h_{(a,j)}) > \ell_b$, so as $k \leqslant \ell_b$ in~\eqref{E:remain}, we have $\pd_t^k(h_{(b,j)})|_{\xi_b} = 0$. When $a=b$, we apply \eqref{E:norm3}, and see that the only possibly non-zero terms in \eqref{E:remain} occur when $a=b$, $k=\ell_b-j$. This condition on~$k$ forces $\ell_b-\ell \leqslant j \leqslant \ell_b$. Bringing everything together and using~\eqref{E:norm3}, we find
\begin{equation} \label{E:finalsteps}
\sum_{j=\ell_b - \ell}^{\ell_b} \alpha_{m,(b,j)}^{(m)} \pd_t^{\ell+j-\ell_b} \bigl( g_{\ba}^{(m)} \dots g_{\ba}^{(1)} \bigr) \big|_{\xi_b} \in K.
\end{equation}
If $\ell = 0$, then there is only one term, and we obtain
\[
\alpha_{m,(b,\ell_b)}^{(m)} \bigl( g_{\ba}^{(m)} \cdots g_{\ba}^{(1)} \bigr)\big|_{\xi_b} \in K.
\]
For $n > 0$, the zeros of $g_{\ba}^{(n)}$ are supported on $\{ \xi_a^{(k)} \mid a \in \cP_{\ba},\ k \geqslant n \}$, whence $(g_{\ba}^{(m)} \cdots g_{\ba}^{(1)})|_{\xi_b} \in K^{\times}$. Thus $\alpha_{m,(b,\ell_b)}^{(m)} \in K$. Suppose that we have shown that $\alpha_{m,(b,\ell_b)}^{(m)}$, $\alpha_{m,(b,\ell_b-1)}^{(m)}, \dots, \alpha_{m,(b,\ell_b-\ell+1)}^{(m)} \in K$. Then \eqref{E:finalsteps} implies that $\alpha_{m,(b,\ell_b-\ell)}^{(m)} ( g_{\ba}^{(m)} \cdots g_{\ba}^{(1)} )|_{\xi_b} \in K$, and so as in the $\ell=0$ case, we obtain $\alpha_{m,(b,\ell_b-\ell)}^{(m)} \in K$. Therefore, by induction on $\ell$ we have shown that $\alpha_{m,(b,j)}^{(m)} \in K$ for all $0 \leqslant j \leqslant \ell_b$, and thus also for all $b \in \cP_{\ba}$. This completes the proof of the claim and the theorem.
\end{proof}

\subsection{Reduction modulo finite primes} \label{SS:reduction}
Let $f \in A_+$ with $\deg f \geqslant 1$, and fix $\ba \in \cA_f$ effective with $\deg \ba > 0$. As Theorem~\ref{T:EaKdef} fashions a model of $E_{\ba}$ defined over~$K$, we can conjugate by a suitable matrix $b\rI$, $b \in B$, to clear denominators and obtain a model
\[
E_{\ba} : \bA \to \Mat_{d}(B[\tau])
\]
defined over $B$. In this section we explore different models for $E_{\ba}$ and investigate good reduction at finite primes of~$B$. We more fully determine primes of good reduction in~\S\ref{App:GoodRed}.

\begin{subsubsec}{Primes in $B^{\perf}$}
Temporarily let $F/k$ be a finite separable extension, and let $B \subseteq F$ be the integral closure of $A$ in $F$.
For a maximal ideal $\fp \subseteq B$, let $B_{\fp}$ be the localization of $B$ at $\fp$. We take $\FF_{\fp} = B_{\fp}/\fp B_{\fp} \cong B/\fp$. We also have $\fp^{\perf} \subseteq B^{\perf} \subseteq (B_{\fp})^{\perf}  \subseteq F^{\perf}$, where for $S \subseteq F$, $S^{\perf} = \{ \alpha \in F^{\perf} \mid \exists\,n \geqslant 0,\ \alpha^{p^n} \in S \}$. It is straightforward to check that $B^{\perf}$ is a perfect ring and that $\fp^{\perf}$ is an ideal of $B^{\perf}$. The following properties can be verified (cf.~\cite{GazdaMaurischat25}*{Lem.~2.3}).
\begin{itemize}
\item $\fp^{\perf}$ is a maximal ideal of $B^{\perf}$.
\item $\fp^{\perf} \cap B = \fp$.
\item $B^{\perf}/\fp^{\perf} \cong \FF_{\fp}$.
\item The localization $(B^{\perf})_{\fp^{\perf}}$ is the same as $(B_{\fp})^{\perf}$, and will be denoted $B_{\fp}^{\perf}$.
\end{itemize}
\end{subsubsec}

\begin{subsubsec}{Reduction of $t$-modules} \label{SSS:redtmods}
Suppose that $E : \bA \to \Mat_{d}(B_{\fp}[\tau])$ is a $t$-module. We will denote the reduction of $E$ modulo~$\fp$ by
\[
\oE : \bA \to \Mat_{d}(\FF_{\fp}[\tau]),
\]
which is a $t$-module over $\FF_\fp$ defined in the natural way. That is, if $E_a = \sum_{i=0}^s C_i \tau^i$, $C_i \in \Mat_d(B_{\fp})$, for $a \in \bA$, then $\oE_a \assign \sum_{i=0}^s \overline{C}_i \tau^i$, $\overline{C}_i \in \Mat_d(\FF_\fp)$, where $\overline{C}_i$ is obtained from $C_i$ by reducing its entries modulo~$\fp$. For any $a \in \bA$, we have commutative diagrams,
\begin{center}
\begin{tikzcd}
E(B_{\fp}) \arrow[d, "E_a"'] \arrow[r, two heads] & 
\oE(\FF_{\fp}), \arrow[d, "\oE_a"] & E(B_{\fp}^{\perf}) \arrow[d, "E_a"'] \arrow[r, two heads] & 
\oE(\FF_{\fp}), \arrow[d, "\oE_a"] \\
E(B_{\fp}) \arrow[r, two heads] & 
\oE(\FF_{\fp}) & E(B_{\fp}^{\perf}) \arrow[r, two heads] & 
\oE(\FF_{\fp})
\end{tikzcd}
\end{center}
where the horizontal maps are the natural reduction maps. In other words,
$E(B_{\fp}) \twoheadrightarrow \oE(\FF_{\fp})$ and $E(B_{\fp}^{\perf}) \twoheadrightarrow \oE(\FF_{\fp})$ are $\bA$-module homomorphisms.
\end{subsubsec}

\begin{subsubsec}{Reduction of $t$-comotives}
Fix a maximal ideal $\fp \subseteq B$. Suppose that $E : \bA \to \Mat_d(B_{\fp}[\tau])$ is a $t$-module over $B_{\fp}$. We define the $t$-comotive associated to $E$ over $B_{\fp}^{\perf}$,
\[
H(E)(B_{\fp}^{\perf}) \assign \Mat_{1 \times d}(B_{\fp}^{\perf}[\sigma]) \subseteq H(E)(F^{\perf}) \assign \Mat_{1 \times d}(F^{\perf}[\sigma]).
\]
Then $H(E)(F^{\perf})$ is the $t$-comotive of $E$ over $F^{\perf}$ as in~\S\ref{SSS:tmod-dualtmot}, and $H(E)(B_{\fp}^{\perf})$ inherits the structure of a left $B_{\fp}^{\perf}[t,\sigma]$-module.

Likewise, we have the $t$-comotive of the reduction of $E$ modulo $\fp$, $H(\oE)(\FF_\fp) \assign \Mat_{1 \times d}(\FF_\fp[\sigma])$. Then for any $a \in \bA$, we have a commutative diagram,
\begin{center}
\begin{tikzcd}
H(E)(B_{\fp}^{\perf}) \arrow[d, "(\cdot)E_a^*"'] \arrow[r, two heads] & 
H(\oE)(\FF_{\fp}). \arrow[d, "(\cdot)\oE_a^*"] \\
H(E)(B_{\fp}^{\perf}) \arrow[r, two heads] & 
H(\oE)(\FF_{\fp})
\end{tikzcd}
\end{center}
\end{subsubsec}

\begin{subsubsec}{Good reduction} \label{SSS:goodred}
When $E : \bA \to \Mat_{d}(B_{\fp}[\tau])$ is coabelian, we say that $E$ has \emph{good reduction} if $\oE$ is coabelian and the rank of $\oE$, i.e., the rank of $H(\oE)(\FF_\fp)$ as an $\FF_{\fp}[t]$-module, is the same as the rank of $E$. Otherwise, $E$ has \emph{bad reduction}.

If $E : \bA \to \Mat_{d}(K[\tau])$ is a coabelian $t$-module, then we say that $E$ has \emph{good reduction at $\fp$} if $E$ is isomorphic over $K[\tau]$ to a $t$-module $E' : \bA \to \Mat_d(B_{\fp}[\tau])$ that has good reduction. Otherwise, $E$ has \emph{bad reduction at $\fp$}.
\end{subsubsec}

\begin{subsubsec}{Good reduction of $E_{\ba}$}
Returning to the setting of Sinha modules, let $\ba \in \cA_f$ be effective with $\deg \ba >0$. Suppose $\fp \in B$ is a maximal ideal prime to~$f$, and let
\[
H_{\ba,\fp} = H_{\ba,\fp}(\oFF_{\fp}) = H^0 \bigl(U,\cO_{X}(-\oW_{\ba}^{(1)}) \bigr) \subseteq \oFF_{\fp}[t,z],
\]
where $\oW_{\ba} \in \Div_{\FF_{\fp}}(X)$ is the reduction of $W_{\ba}$ modulo~$\fp$ as in \S\ref{SS:AndHecke}. We make $H_{\ba,\fp}$ into a left $\oFF_{\fp}[t,\sigma]$-module by setting $\sigma h \assign \og_{\ba} h^{(-1)}$ for $h \in H_{\ba,\fp}$, which is well-defined by Proposition~\ref{P:ogaf}.
Moreover, $H_{\ba,\fp}$ is an integral ideal of $\oFF_{\fp}[t,z]$ and so is free and finitely generated as an $\oFF_{\fp}[t]$-module of rank~$r = [K:k]$. As $\oW_{\ba}$ and $\oXi_{\ba}$ have disjoint supports from $I_{\ba}$, the same proof as in \cite{ABP04}*{\S 6.4.2} demonstrates that $H_{\ba,\fp}$ is a $t$-comotive over~$\oFF_{\fp}$. As in Lemma~\ref{L:Haprops}, it satisfies
\begin{gather} \label{E:Haprankdim}
\rank H_{\ba,\fp} = \rank_{\oFF_{\fp}[t]} H_{\ba,\fp} = \rank H_{\ba} = r, \\
\dim H_{\ba,\fp} = \rank_{\oFF_{\fp}[\sigma]} H_{\ba,\fp} = \deg \oXi_{\ba} = \deg \Xi_{\ba} =\dim H_{\ba} = d. \notag
\end{gather}
We let $H_{\ba,\fp}(\FF_{\fp}) = H_{\ba,\fp}(\oFF_{\fp}) \cap \FF_{\fp}[t,z]$, and note that $H_{\ba,\fp}(\FF_{\fp})$ is a $t$-comotive as a left~$\FF_{\fp}[t,\sigma]$-module with the same rank and dimension as $H_{\ba}$. Furthermore, the algorithm mapped out in \S\ref{SS:Fodef}, and in particular in Proposition~\ref{P:HaKdef}, works equally well for $H_{\ba,\fp}(\FF_{\fp})$, from which we can obtain an $\FF_{\fp}[\sigma]$-basis of $H_{\ba,\fp}(\FF_{\fp})$ that is defined over $\FF_{\fp}$.
\end{subsubsec}

\begin{proposition} \label{P:goodbasis}
Let $\ba \in \cA_f$ be effective with $\deg \ba > 0$, and let $\fp$ be a maximal ideal of $B$ relatively prime to $f$. Suppose that $\{ h_1, \dots, h_d\}$ is a $\ok[\sigma]$-basis of $H_{\ba}(\ok)$ such that
\begin{alphenumerate}
\item $\{ h_1, \dots, h_d \}$ is contained in $B_{\fp}^{\perf}[t,z]$,
\item $\{ \oh_1, \dots, \oh_d \}$ generates $H_{\ba,\fp}(\FF_{\fp})$ as an $\FF_{\fp}[\sigma]$-module.
\end{alphenumerate}
Then $\{ \oh_1, \dots, \oh_d \}$ is an $\FF_{\fp}[\sigma]$-basis of $H_{\ba,\fp}(\FF_{\fp})$, and $\{ h_1, \dots, h_d \}$ is a $B_{\fp}^{\perf}[\sigma]$-basis of $H_{\ba}(B_{\fp}^{\perf}) \assign H_{\ba}(\ok) \cap B_{\fp}^{\perf}[t,z]$.
\end{proposition}

\begin{proof}
We consider the reduction map $\pi : H_{\ba}(B_{\fp}^{\perf}) \to H_{\ba,\fp}(\FF_{\fp})$, which commutes with the action of $\sigma$. The kernel of $\pi$ is $\fp^{\perf} \cdot H_{\ba}(B_{\fp}^{\perf}) = H_{\ba}(K^{\perf}) \cap \fp^{\perf} \cdot B_{\fp}^{\perf}[t,z]$. Also,
\[
\sigma \bigl( \fp^{\perf} \cdot H_{\ba}(B_{\fp}^{\perf}) \bigr) = \fp^{\perf} \cdot \sigma \bigl( H_{\ba}(B_{\fp}^{\perf}) \bigr),
\]
since $\fp^{\perf}$ is invariant under Frobenius twisting. We take $S = \{ h_1, \dots, h_d\}$ and $\oS = \{ \oh_1, \dots, \oh_d \}$. We note that as $\oS$ generates $H_{\ba,\fp}(\FF_{\fp})/\sigma H_{\ba,\fp}(\FF_{\fp})$, which has dimension~$d$ over $\FF_{\fp}$, it follows that $\oS$ is $\FF_{\fp}$-linearly independent. Nakayama's lemma~\cite{LangAlg}*{\S X.4} implies that $S$ induces a $B_{\fp}^{\perf}$-basis of $H_{\ba}(B_{\fp}^{\perf})/\sigma H_{\ba}(B_{\fp}^{\perf})$. By induction, for $\ell \geqslant 1$, $\oS \cup \sigma \oS \cup \dots \cup \sigma^{\ell-1} \oS$ is an $\FF_{\fp}$-basis of $H_{\ba,\fp}(\FF_{\fp})/\sigma^{\ell} H_{\ba,\fp}(\FF_{\fp})$, implying that $S \cup \sigma S \cup \dots \cup \sigma^{\ell-1} S$ induces a $B_{\fp}^{\perf}$-basis of $H_{\ba}(B_{\fp}^{\perf})/\sigma^{\ell} H_{\ba}(B_{\fp}^{\perf})$. Taking a direct limit over $\ell$, we find that $S$ is a $B_{\fp}^{\perf}[\sigma]$-basis of $H_{\ba}(B_{\fp}^{\perf})$.
\end{proof}

\begin{proposition} \label{P:Eagoodred}
For $\ba \in \cA_f$ with $\deg \ba >0$, let $E_{\ba} : \bA \to \Mat_d(K[\tau])$ be constructed as in Theorem~\ref{T:EaKdef} with respect to a $\ok[\sigma]$-basis $\{ h_1, \dots, h_d \}$ of $H_{\ba}(\ok)$ defined over $K$. Suppose $\fp$ is a prime of $B$, not dividing $f$, such that $\{ h_1, \dots, h_d\}$ satisfies the conditions of Proposition~\ref{P:goodbasis}. Then $E_{\ba}$ has good reduction at~$\fp$.
\end{proposition}

\begin{proof}
As in Theorem~\ref{T:EaKdef}, we can use $S = \{h_1, \dots, h_d\}$ together with~\eqref{E:Etdef} to construct a model $E_{\ba} : \bA \to \Mat_d(K[\tau])$ defined over $K$. Likewise, we use $\oS = \{ \oh_1, \dots, \oh_d \}$ and \eqref{E:Etdef} to construct a $t$-module $E_{\ba,\fp} : \bA \to \Mat_d(\FF_{\fp}[\tau])$. Using the $\bA$-stability of $H_{\ba}(B_{\fp}^{\perf})$ and $H_{\ba,\fp}(\FF_{\fp})$,  we have a commutative diagram via~\eqref{E:Etdef},
\begin{equation}
\begin{tikzcd}
\Mat_{1 \times d}(B_{\fp}^{\perf}[\sigma]) \arrow[d, "(\cdot)E_{\ba,t}^{*}"'] \arrow[r, "\sim"]
& H_{\ba}(B_{\fp}^{\perf}) \arrow[d, "t(\cdot)"'] \arrow[r, "\pi", two heads]
& H_{\ba,\fp}(\FF_{\fp}) \arrow[d, "t(\cdot)"] \arrow[r, "\sim"]
& \Mat_{1 \times d}(\FF_{\fp}[\sigma]) \arrow[d, "(\cdot) E_{\ba,\fp,t}^*"] \\
\Mat_{1 \times d}(B_{\fp}^{\perf}[\sigma]) \arrow[r, "\sim"]
& H_{\ba}(B_{\fp}^{\perf}) \arrow[r, "\pi", two heads]
& H_{\ba,\fp}(\FF_{\fp}) \arrow[r, "\sim"]
& \Mat_{1 \times d}(\FF_{\fp}[\sigma]).
\end{tikzcd}
\end{equation}
We note from the first square that
\[
E_{\ba,t} \in \Mat_{d}(K[\tau]) \cap \Mat_d(B_{\fp}^{\perf}[\tau]) = \Mat_d(B_{\fp}[\tau]),
\]
and thus $E_{\ba}$ is defined over $B_{\fp}$. A short calculation then yields
\begin{equation} \label{E:EatbarisEapt}
\oE_{\ba,t} = E_{\ba,\fp,t},
\end{equation}
where $\oE_{\ba}$ is the reduction of $E_{\ba}$ modulo~$\fp$ as in~\S\ref{SSS:redtmods}. Since the rank of $E_{\ba,\fp}$ is the same as the rank of $E_{\ba}$, we see that $E_{\ba}$ has good reduction at~$\fp$.
\end{proof}

\begin{corollary} \label{C:EaSigma}
Let $\ba \in \cA_f$ with $\deg \ba > 0$, and let $E_{\ba}: \bA \to \Mat_d(K[\tau])$ be constructed as in Theorem~\ref{T:EaKdef}. Then there is a finite set of primes $\Sigma$ of $B$, containing the primes above $f$, such that $E_{\ba}$ has good reduction modulo $\fp \notin \Sigma$.
\end{corollary}

\begin{proof}
As in the proof of Theorem~\ref{T:EaKdef}, we obtain a model for $E_{\ba}$ over $K$ by selecting a normalized $\ok[\sigma]$-basis $\{ h_{(a,j)} \mid (a,j) \in \cS_{\ba} \}$ of $H_{\ba}(\ok)$ defined over~$K$. We can therefore find a set of primes $\Sigma$ of $B$, containing the primes above~$f$, such that for $\fp \notin \Sigma$, we have $h_{(a,j)} \in B_{\fp}[t,z]$ for all $(a,j) \in \cS_{\ba}$. Moreover, by possibly increasing $\Sigma$, we find for $\fp \notin \Sigma$ that $\{ \oh_{(a,j)} \mid (a,j) \in \cS_{\ba}\}$ satisfy conditions~\eqref{E:norm1}--\eqref{E:norm3} in~\S\ref{SSS:normalized} modulo~$\fp$. (There is a subtle point here, noted after~\eqref{E:norm2}, that we need $\oW_{\ba}^{(1)}$ and $\oXi_{\ba}$ to have disjoint supports modulo~$\fp$, but having common supports can only happen for finitely many primes.) The reduction map
\[
H_{\ba}(\ok) \cap B_{\fp}^{\perf}[t,z] \to H_{\ba,\fp}(\FF_{\fp})
\]
is surjective and commutes with $\sigma$. For $\fp \notin \Sigma$, the normalization of $\{ \oh_{(a,j)} \}$ ensures it generates $H_{\ba,\fp}(\FF_{\fp})$ as an $\FF_{\fp}[\sigma]$-module. Proposition~\ref{P:Eagoodred} completes the proof.
\end{proof}

\begin{remark}
It is a natural question to characterize primes of good reduction for $E_{\ba}$. Although Corollary~\ref{C:EaSigma} is sufficient for our main results on $L$-values, Theorem~\ref{T:GossLMain} and Corollary~\ref{C:HeckeLMain}, we show in \S\ref{App:GoodRed} that for each prime $\fp$ of $B$ not dividing $f$, $E_{\ba}$ is isomorphic over $K$ to a model over $B_{\fp}$ with good reduction modulo~$\fp$. See Corollary~\ref{C:Eagoodprimes}.
\end{remark}

\section{Goss and Hecke \texorpdfstring{$L$}{L}-functions} \label{S:GossHeckeL}

We now consider Goss $L$-functions attached to Sinha modules and identities for them in terms of Hecke $L$-functions for Anderson's Hecke character. Our main results focus on special values of these $L$-functions at $s=0$.

\subsection{\texorpdfstring{$L$}{L}-functions for Sinha modules} \label{SS:GossLSinha}
We fix effective $\ba \in \cA_f$ with $\deg \ba > 0$, and let $E_{\ba} : \bA \to \Mat_d(K[\tau])$ be its associated Sinha module defined over $K$ as in \S\ref{S:Sinha}. We let $\Sigma_{\ba}$ denote a finite set of primes of $B$, containing all primes above $f$, such that $E_{\ba}$ has good reduction at $\fp \notin \Sigma_{\ba}$ as in Corollary~\ref{C:EaSigma}. As in \S\ref{SSS:GossLtmod}, in order to define the Goss $L$-functions for~$E_{\ba}$, we need to analyze the Galois action on its associated Tate modules.

Fix a prime $\fp \notin \Sigma_{\ba}$, and choose a model
\[
E_{\ba} : \bA \to \Mat_{d}(B_{\fp}[\tau])
\]
for which $E_{\ba}$ has good reduction at~$\fp$. Then as in \eqref{E:EatbarisEapt}, we have $\oE_{\ba} \cong E_{\ba,\fp}$ over $\FF_{\fp}$.

Let $\lambda \in \bA_+$ be irreducible with $\fp \nmid \lambda(\theta)$. Following \S\ref{SSS:GossLtmod}, we let $\alpha_{\fp} \in \Gal(K^{\sep}/K)$ be a Frobenius element and set
\begin{align} \label{E:Papdef}
P_{\ba,\fp}(X) &\assign \Char(\alpha_\fp, T_{\lambda}(E_{\ba}),X)|_{t = \theta} \in A_{\lambda(\theta)}[X], \\
P_{\ba,\fp}^{\vee}(X) &\assign \Char(\alpha_\fp, T_{\lambda}(E_{\ba})^{\vee},X)|_{t = \theta} \in k_{\lambda(\theta)}[X], \notag
\end{align}
where $T_{\lambda}(E_{\ba})$ is the $\lambda$-adic Tate module of $E_{\ba}$ and $T_{\lambda}(E_{\ba})^{\vee}$ is its dual. As one may expect since we are in the situation of good reduction, $T_{\lambda}(E_{\ba})$ and $T_{\lambda}(E_{\ba})^{\vee}$ form families of strictly compatible representations in the sense of \cite{Goss}*{Def.~8.6.5}. We prove in Proposition~\ref{P:charpolys} that these polynomials are independent of the choice of~$\lambda$, and furthermore $P_{\ba,\fp}(X) \in A[X]$, $P_{\ba,\fp}^{\vee}(X) \in k[X]$. For Drinfeld modules with good reduction, such results were proved originally by Takahashi~\cite{Takahashi82} (see also~\cite{Goss}*{Ex.~8.6.6.2}), and our proof is similar. However, using results from~\cite{HuangP22} we also obtain more detailed information.

The reduction map $\pi : H_{\ba}(B_{\fp}^{\perf}) \to H_{\ba,\fp}(\FF_{\fp})$ from Proposition~\ref{P:goodbasis} also commutes with the operation of $\bA$, and moreover, as in \S\ref{SSS:HaRAT}, we can pick a $\ok[t]$-basis of $H_{\ba}(\ok)$, $\{ n_1, \dots, n_r \} \subseteq B_{\fp}^{(1)}[t,z]$, such that their reductions $\{ \on_1, \dots, \on_r \} \subseteq \FF_{\fp}[t,z]$ form a $\FF_{\fp}[t]$-basis of $H_{\ba,\fp}(\FF_{\fp})$. Letting $\bn = (n_1, \dots, n_r)^{\tr}$ and $\obn = (\on_1, \dots, \on_r)^{\tr}$, we see that the matrix $\Phi$ from~\eqref{E:Phirational} can be chosen in $\Mat_r(B_{\fp}[t])$, and satisfies,
\[
\sigma \bn = \Phi \bn, \quad \sigma \obn = \oPhi \obn,
\]
where $\oPhi \in \Mat_r(\FF_{\fp}[t])$ is the reduction of $\Phi$ modulo $\fp$. Moreover, the matrix $Q$ from \S\ref{SSS:HaRAT} satisfies $Q \in \Mat_r(B_{\fp}^{(1)}[t])$, and
\begin{equation} \label{E:Phiconj}
\Phi = Q^{(-1)} \Phi_{\ba} Q^{-1}, \quad \oPhi = \oQ^{(-1)} \oPhi_{\ba} \oQ^{-1}.
\end{equation}
Recall Anderson's Hecke character $\chi_{\ba} : \cI_{K,f} \to K^{\times}$ from \S\ref{SS:AndHecke}. For convenience define
\begin{equation} \label{E:psiadef}
\psi_{\ba} \assign \chi_{-\ba} = \chi_{\ba}^{-1} : \cI_{K,\ff} \to K^{\times},
\end{equation}
which is a Hecke character of infinity type $\Theta_{-\ba} = -\Theta_{\ba}$ as in \S\ref{SSS:GenAndHecke}.

\begin{proposition} \label{P:charpolys}
Let $\ba \in \cA_{f}$ be effective with $\deg \ba > 0$. Let $\fp$ be a prime of $B$, $\fp \notin \Sigma_{\ba}$. Then
\[
P_{\ba,\fp}(X) = \prod_{b \in (A/fA)^{\times}} \bigl( X - \chi_{\ba}^{\rho_b}(\fp) \bigr) \in A[X], \quad
P_{\ba,\fp}^{\vee}(X) = \prod_{b \in (A/fA)^{\times}} \bigl( X - \psi_{\ba}^{\rho_{b}}(\fp) \bigr) \in k[X].
\]
These characteristic polynomials are independent of the choice of $\lambda \in \bA_+$ from~\eqref{E:Papdef}.
\end{proposition}

\begin{proof}
Pick $\lambda \in \bA_+$ irreducible such that $\fp \nmid \lambda(\theta)$. Since $\fp$ is a prime of good reduction for $E_{\ba}$ and $\fp \nmid \lambda(\theta)$, there is an isomorphism of Tate modules, $T_{\lambda}(E_{\ba}) \to T_{\lambda}(\oE_{\ba})$, compatible with the Galois action. If $\ell = [\FF_{\fp}:\FF_q]$, by the definition in~\eqref{E:Papdef},
\[
P_{\ba,\fp}(X) = \Char(\tau^\ell,T_{\lambda}(\oE_{\ba}),X)|_{t = \theta} = \Char(\tau^\ell,T_{\lambda}(E_{\ba,\fp}),X)|_{t = \theta},
\]
where the second equality follows from the isomorphism $\oE_{\ba} \cong E_{\ba,\fp}$ over $\FF_{\fp}$. On the other hand it follows from \cite{HuangP22}*{Eq.~(3.7.1), Cor.~3.7.3} (with $Z = \emptyset$) (cf.~\cite{Taelman09}*{Prop.~7}) that
\begin{equation} \label{E:chartwists}
\Char(\tau^\ell, T_{\lambda}(E_{\ba,\fp}),X) = \Char \bigl( \oPhi^{(-\ell+1)} \cdots \oPhi^{(-1)} \oPhi,X \bigr) \in \bA[X].
\end{equation}
Strictly speaking, the results in \cite{HuangP22}*{\S 3.7} do not appear to immediately imply this identity, as the proofs there center on $t$-motives rather than $t$-comotives. It is possible to rewrite these results from the standpoint of $t$-comotives instead, but there is a more straightforward explanation. Maurischat~\cite{Maurischat21} has shown that abelian and coabelian are equivalent conditions for Anderson $t$-modules. By Hartl and Juschka~\cite{HartlJuschka20}*{Thm.~2.5.13}, this implies that if $\oGamma \in \Mat_{r}(\FF_{\fp}[t])$ represents multiplication by~$\tau$ on the $t$-motive of~$E_{\ba,\fp}$, then there is $V \in \GL_r(\FF_{\fp}[t])$ such that $\oGamma^{\tr} V = V^{(-1)} \oPhi$ (see also~\cite{NamoijamP24}*{Eq.~(4.65)}). In \cite{HuangP22}*{\S 3.7} it is shown that $\Char(\tau^\ell, T_{\lambda}(E_{\ba,\fp}),X) = \Char( \oGamma^{(\ell-1)} \cdots \oGamma^{(1)} \oGamma,X)$. As matrices with entries in $\FF_{\fp}[t]$, it follows that $V^{(\ell)} = V$ and $\oPhi^{(\ell)} = \oPhi$, and so
\begin{multline*}
\bigl( \oGamma^{(\ell-1)} \cdots \oGamma^{(1)} \oGamma \bigr)^{\tr}
= \oGamma^{\tr} \bigl( \oGamma^{\tr} \bigr)^{(1)} \cdots \bigl( \oGamma^{\tr} \bigr)^{(\ell-1)} \\
= V^{(-1)} \oPhi\, \oPhi^{(1)} \cdots \oPhi^{(\ell-1)} \bigl( V^{-1} \bigr)^{(\ell-1)}
= V^{(-1)} \oPhi\, \oPhi^{(-\ell+1)} \cdots \oPhi^{(-1)} \bigl( V^{(-1)} \bigr)^{-1}.
\end{multline*}
Thus $\oGamma^{(\ell-1)} \cdots \oGamma^{(1)} \oGamma$ and $\oPhi\, \oPhi^{(-\ell+1)} \cdots \oPhi^{(-1)}$ have the same characteristic polynomial, and the latter has the same characteristic polynomial as $\oPhi^{(-\ell+1)} \cdots \oPhi^{(-1)} \oPhi$, justifying~\eqref{E:chartwists}.

Now $\oPhi = \oQ^{(-1)} \oPhi_{\ba} \oQ^{-1}$, where $\oQ \in \GL_r(\FF_{\fp}[t])$. It follows that $\oPhi^{(-\ell+1)} \cdots \oPhi^{(-1)} \oPhi = \oQ^{(-1)} \oPhi_{\ba}^{(-\ell+1)} \cdots \oPhi_{\ba}^{(-1)} \oPhi_{\ba} (\oQ^{(-1)})^{-1}$, and thus
\[
P_{\ba,\fp}(X)= \Char \bigl(\oPhi_{\ba}^{(-\ell+1)} \cdots \oPhi_{\ba}^{(-1)} \oPhi_{\ba},X \bigr) \big|_{t = \theta}
= \prod_{b \in (A/fA)^{\times}} \biggl(X - \prod_{i=0}^{\ell-1} \og_{b \star \ba}^{(i)} \biggr) \bigg|_{\xi},
\]
with the second equality following from Proposition~\ref{P:PhiaPsiadiag}(a). By the definition of $\chi_{\ba}$ from~\S\ref{SS:AndHecke} and Lemma~\ref{L:chiGalaction}, we see that
\[
P_{\ba,\fp}(X) = \prod_{b \in (A/fA)^{\times}} \bigl( X - \chi_{b \star \ba}(\fp) \bigr) = \prod_{b \in (A/fA)^{\times}} \bigl( X - \chi_{\ba}^{\rho_{b}}(\fp) \bigr).
\]
Finally the result for $P_{\ba,\fp}^{\vee}(X)$ follows directly from this one and the definition of $\psi_{\ba}$.
\end{proof}

As in \S\ref{SSS:GossLtmod}, we define reciprocal polynomials $Q_{\ba,\fp}(X)$ and $Q_{\ba,\fp}^{\vee}(X)$ for $P_{\ba,\fp}(X)$ and $P_{\ba,\fp}^{\vee}(X)$ and Goss $L$-functions, for $s \in \ZZ$,
\begin{equation} \label{E:LEaKSigmadef}
L(E_{\ba}/K,\Sigma_{\ba},s) = \prod_{\fp \notin \Sigma_{\ba}} Q_{\ba,\fp}\bigl( \cN(\fp)^{-s} \bigr)^{-1}, \
L(E_{\ba}^{\vee}/K,\Sigma_{\ba},s) = \prod_{\fp \notin \Sigma_{\ba}} Q_{\ba,\fp}^{\vee} \bigl( \cN(\fp)^{-s} \bigr)^{-1}.
\end{equation}
When they converge, these $L$-functions take values in $k_{\infty}$. As defined these $L$-values depend on the model $E_{\ba}$ we have chosen.

\begin{remark}
We will show in Corollary~\ref{C:Eagoodprimes} that $E_{\ba}$ has good reduction at all primes $\fp \nmid f$, so it makes sense to define
\begin{equation} \label{E:LEaKdef}
L(E_{\ba}/K,s) = \prod_{\fp \nmid f} Q_{\ba,\fp}\bigl( \cN(\fp)^{-s} \bigr)^{-1}, \quad
L(E_{\ba}^{\vee}/K,s) = \prod_{\fp \nmid f} Q_{\ba,\fp}^{\vee} \bigl( \cN(\fp)^{-s} \bigr)^{-1},
\end{equation}
though there may not be a single model for $E_{\ba}$ over $B$ with good reduction at every $\fp \nmid f$.
\end{remark}

We also have Hecke $L$-functions as in \S\ref{SSS:HeckeL}, for $s \in \ZZ$,
\begin{equation} \label{E:AndHeckeLdef}
L(\chi_\ba,s) = \prod_{\fp \nmid f} \biggl( 1 - \frac{\chi_{\ba}(\fp)}{\cN(\fp)^s} \biggr)^{-1}, \quad
L(\psi_\ba,s) = \prod_{\fp \nmid f} \biggl( 1 - \frac{\psi_{\ba}(\fp)}{\cN(\fp)^s} \biggr)^{-1}.
\end{equation}
When they converge, these Hecke $L$-functions take values in $k_{\infty}(\zeta)$. We further set
\[
L(\chi_\ba,\Sigma_{\ba},s) = \prod_{\fp \notin \Sigma_{\ba}} \biggl( 1 - \frac{\chi_{\ba}(\fp)}{\cN(\fp)^s} \biggr)^{-1}, \quad
L(\psi_\ba,\Sigma_{\ba}, s) = \prod_{\fp \notin \Sigma_{\ba}} \biggl( 1 - \frac{\psi_{\ba}(\fp)}{\cN(\fp)^s} \biggr)^{-1}.
\]
The restricted $L$-function $L(\chi_{\ba},\Sigma_{\ba},s)$ differs from $L(\chi_{\ba},s)$ by finitely many Euler factors.

\begin{proposition} \label{P:Lidentities}
Let $\ba \in \cA_{f}$ be effective with $\deg \ba > 0$, and let $D_{\ba} = \deg(\ba)/(q-1)$. The following hold.
\begin{alphenumerate}
\item $L(\chi_{\ba},s)$ converges for $s > D_{\ba}$, and $L(\psi_{\ba},s)$ converges for $s > -D_{\ba}$.
\item $L(E_{\ba}/K,\Sigma_{\ba},s)$ converges for $s > D_{\ba}$, and $L(E_{\ba}^{\vee}/K,\Sigma_{\ba},s)$ converges for $s > -D_{\ba}$.
\item $\displaystyle L(E_{\ba}/K,\Sigma_{\ba},s) = \prod_{b \in (A/fA)^{\times}} L(\chi_{\ba}^{\rho_b},\Sigma_{\ba},s) = L(\chi_{\ba},\Sigma_{\ba},s)^r$.
\item $\displaystyle L(E_{\ba}^{\vee}/K,\Sigma_{\ba},s) = \prod_{b \in (A/fA)^{\times}} L(\psi_{\ba}^{\rho_b},\Sigma_{\ba},s) = L(\psi_{\ba},\Sigma_{\ba},s)^r$.
\end{alphenumerate}
\end{proposition}

\begin{proof}
Let $\fp$ be a prime of $B$, $\fp \notin \Sigma_{\ba}$, and set $\ell = [\FF_{\fp}:\FF_q]$. By Theorem~\ref{T:AndHecke}(b) and \S\ref{SSS:GenAndHecke}, we see that for any $b \in (A/fA)^{\times}$,
\[
\bigl\lvert \chi_{\ba}^{\rho_b}(\fp) \bigr\rvert_{\infty} = \inorm{\ttheta}^{\ell \deg \ba} \quad \Rightarrow \quad
\biggl\lvert \frac{\chi_{\ba}^{\rho_b}(\fp)}{\cN(\fp)^{s}}\biggr\rvert_{\infty} = \frac{\inorm{\ttheta}^{\ell \deg \ba}}{\inorm{\theta}^{s \ell}} = \inorm{\theta}^{(1/(q-1) \cdot \deg \ba - s)\ell}.
\]
Taking $b=1$, this implies the first part of~(a). Since $\psi_{\ba} = \chi_{-\ba}$, the second part also follows. The first equalities in both (c) and~(d) follow immediately from Proposition~\ref{P:charpolys}, and these equalities imply~(b). It remains to show the second equalities in~(c) and~(d). Fix $\wp \in A_+$ irreducible and relatively prime to the primes in~$\Sigma_{\ba}$. Using Proposition~\ref{P:charpolys}, consider the (reciprocals of the) Euler factors in~\eqref{E:LEaKSigmadef},
\begin{equation} \label{E:Eulerfactors}
\prod_{\fp \mid \wp} \prod_{b \in (A/fA)^{\times}} \bigl( 1 - \chi_{\ba}^{\rho_b}(\fp) \cN(\fp)^{-s}\bigr) =
\prod_{\fp \mid \wp} \prod_{b \in (A/fA)^{\times}} \Bigl( 1 - \chi_{\ba}(\rho_{b}(\fp)) \wp^{-s [\FF_{\wp}:\FF_q]}\Bigr),
\end{equation}
which follows from the definition of $\cN(\fp)$ in \S\ref{SSS:GossLtmod} and Lemma~\ref{L:chiaGalsymm}. The Galois group $\Gal(K/k)$ acts transitively on the primes above $\wp$ in $B$ with stabilizers of size $\delta = [\FF_{\fp}:\FF_{\wp}]$ for any $\fp \mid \wp$. Thus for each $\fp \mid \wp$, the multiset $\{ \rho_b(\fp) \mid b \in (A/fA)^{\times} \}$ consists of the $r/\delta$ primes above $\wp$, each with multiplicity $\delta$ (as $\wp$ is unramified in $B$). Therefore, the expression in~\eqref{E:Eulerfactors} becomes
\[
\prod_{\fp \mid \wp} \prod_{\fq \mid \wp} \Bigl( 1 - \chi_{\ba}(\fq) \wp^{-s [\FF_{\wp}:\FF_q]}\Bigr)^{\delta} = \prod_{\fq \mid \wp} \Bigl( 1 - \chi_{\ba}(\fq) \cN(\fq)^{-s} \Bigr)^r,
\]
where this equality follows from the observation that the inner product is independent of the outer product over $\fp \mid \wp$. This proves the second equality in~(c), and the second equality in~(d) follows immediately since $\psi_{\ba} = \chi_{-\ba}$.
\end{proof}

\begin{remark} \label{R:LnoSigmas}
(a) As mentioned above, by Corollary~\ref{C:Eagoodprimes} we can take $\Sigma_{\ba}$ to be exactly the primes dividing $f$, in which case identities in Proposition~\ref{P:Lidentities} simplify as
\[
L(E_{\ba}/K,s) = L(\chi_{\ba},s)^r, \quad L(E_{\ba}^{\vee}/K,s) = L(\psi_{\ba},s)^r.
\]
(b) The proof of Proposition~\ref{P:Lidentities}(a) does not depend on $\ba$ being effective, and the conclusion holds for any $L(\chi_{\ba},s)$.
\end{remark}

\subsection{Taelman class module formulas} \label{SS:Taelman}
Continuing with the notation of the previous section, in \S\ref{SS:specialLvalues} we will utilize a theorem of Fang~\cite{Fang15}*{Thm.~1.10} to derive identities for $L(E_{\ba}^{\vee},0)$ and $L(\psi_{\ba},0)$ when $\ba = [a/f]$ is basic. In this section we review the essentials of Taelman's results and Fang's subsequent extensions.

In breakthrough work, Taelman~\cites{Taelman09, Taelman10, Taelman12} proved formulas for $L(\phi^{\vee}/F,0)$, where $\phi$ is a Drinfeld module defined over a finite extension $F/k$. Taelman's identities were generalized to abelian $t$-modules by Fang~\cite{Fang15}, and later Angl\`es, Ngo Dac, Pellarin, and Tavares Ribeiro \cites{ANT17b, ANT20, ANT22, APT18, AT17} have introduced and developed the theory of Stark units to extend these results greatly to the wider class of admissible $\bA$-modules, for $\bA$ more general than $\FF_{q}[t]$.

The formalism of these class formulas utilize two fundamental objects introduced by Taelman~\cites{Taelman10,Taelman12}, namely the regulator lattice and class module of a Drinfeld module, which were later generalized to $t$-modules. We will follow the setup from Angl\`es, Ngo Dac, and Tavares Ribeiro~\cites{ANT17b, ANT20, ANT22}. A third important object for $L$-values is the module of Stark units~\cites{ANT17b, ANT20, ANT22, APT18, AT17}, though we will ultimately not need it for our identities.

\begin{subsubsec}{Fitting ideals}
For a finite $A$-module $M$, we let $[M]_A \in A_+$ denote the unique monic generator of the Fitting ideal of $M$. If $M \cong \oplus_{i=1}^{\ell} A/g_iA$ for $g_1, \dots, g_\ell \in A_+$, then $[M]_A = g_1 \dots g_\ell$, which is independent of the choice of $g_1, \dots, g_{\ell}$. If $M$ is instead a finite $\bA$-module, we will use the convention $[M]_{A} \assign [M]_{\bA}|_{t = \theta}$, which coerces the value into~$A$. See for example~\cite{ANT22}*{\S 1.2} for more general properties.
\end{subsubsec}

\begin{subsubsec}{Lattices and covolumes} \label{SSS:covol}
We follow the constructions in \citelist{\cite{ANT20}*{\S 2.3} \cite{ANT22}*{\S 1.3}}, but see also \cite{Taelman12}. Let $V$ be a finite dimensional $k_{\infty}$-vector space of dimension $n \geqslant 1$. An \emph{$A$-lattice} in~$V$ is an $A$-submodule that is discrete and cocompact. An $A$-lattice $M \subseteq V$ is necessarily free of rank~$n$ over $A$, and any $A$-basis of $M$ is also a $k_{\infty}$-basis of~$V$. If $M$, $N \subseteq V$ are $A$-lattices, then we can find a $k_{\infty}$-isomorphism $\alpha : V \to V$ such that $\alpha(M) \subseteq N$. The \emph{covolume} of $N$ in $M$ is
\[
[M:N]_A \assign \frac{\det_{k_{\infty}} \alpha}{\sgn(\det_{k_{\infty}} \alpha)} \cdot \biggl[ \frac{N}{\alpha(M)} \biggr]_A^{-1} \in k_{\infty}^{\times}.
\]
By~\cite{ANT17b}*{\S 2.3}, $[M:N]_A$ is independent of the choice of $\alpha$ and matches the definition in \cite{Taelman12}*{Prop.~4}. If $N \subseteq M$, then $[M:N]_A = [ M/N ]_A$. Generally, $[M:N]_A = [N:M]_A^{-1}$, and for any three $A$-lattices, $[M:N]_A = [M:P]_A \cdot [P:N]_A$. Furthermore, suppose $v_1, \dots, v_n$ form a $k_\infty$-basis of $V$ and $m_i = \sum_{j} \beta_{ij} v_j$, $\beta_{ij} \in k_{\infty}$, form an $A$-basis of $N$. Then
\begin{equation} \label{E:covoldet}
[\oplus_{i=1}^n A v_i:N]_A = \frac{\det (\beta_{ij})}{\sgn(\det(\beta_{ij}))}.
\end{equation}
If we let $\bk_{\infty} = \laurent{\FF_q}{t^{-1}}$, then for a finite dimensional $\bk_{\infty}$-vector space $V$, we can define \emph{$\bA$-lattices} and covolumes in the same way by converting $\theta$ to $t$ in the above definitions.
\end{subsubsec}

\begin{subsubsec}{Taelman regulator for $E_{\ba}$} \label{SSS:regulators}
For $\ba \in \cA_f$ effective with $\deg \ba > 0$, let $E_{\ba} : \bA \to \Mat_d(B[\tau])$ be a fixed model defined over~$B$ via Theorem~\ref{T:EaKdef}. The unit module $\rU(E_{\ba}/B)$ and regulator $\Reg(E_{\ba}/B) \in k_{\infty}^{\times}$ are defined as in \citelist{\cite{ANT22}*{\S 2.1} \cite{Taelman10}*{\S 5}}. Let
\begin{equation} \label{E:Kinfty}
K_{\infty} \assign K \otimes_k k_{\infty} \cong \bigoplus_{i=1}^n K_{\infty_i}.
\end{equation}
Letting $I_f \assign \{  b \in A_+ \mid \deg b < \deg f, (b,f)=1 \}$, we have $\{ \rho_b|_{K^+} : b \in I_f \} = \Gal(K^+/k)$ as in~\S\ref{SSS:infinity}. We use the convention that $\infty_1$-adic topology on $K$ coincides with the $\infty$-adic subspace topology from $\C$ (so that $K_{\infty_1} = k_{\infty}(\zeta) \subseteq \C$), and we find $\{ \infty_1, \dots, \infty_n \} = \{ \rho_b(\infty_1) \mid b \in I_f \}$. Thus $K_{\infty} \cong \oplus_{b \in I_f} K_{\rho_b(\infty_1)}$. In the usual way, $K \hookrightarrow K_{\infty} \to \oplus_{b \in I_f} K_{\rho_b(\infty_1)}$ via $\alpha \mapsto \alpha \otimes 1 \mapsto (\rho_b(\alpha) \mid b \in I_f)$. We let
\[
\Exp_{E_{\ba}} : \Lie(E_{\ba})(K_{\infty}) \to E_{\ba}(K_{\infty})
\]
be the $\bA$-module map induced by $\exp_{E_{\ba}}(\bz)$. For variables $\bz_b$ on $\Lie(E_{\ba})(K_{\infty})$,
\[
\Exp_{E_{\ba}} ( \bz_b \mid b \in I_f) = \bigl( \exp_{E_{\ba}}^{\rho_b}(\bz_b) \mid b \in I_f \bigr),
\]
where $\exp_{E_{\ba}}^{\rho_b}(\bz) \in \power{K}{z_1, \dots, z_d}^d$ is obtained from $\exp_{E_{\ba}}(\bz)$ by applying $\rho_b$ to its coefficients. As in~\cite{ANT22}*{\S 2.1}, the \emph{unit module} is
\[
\rU(E_{\ba}/B)  \assign \Exp_{E_{\ba}}^{-1} \bigl( E_{\ba}(B) \bigr),
\]
where on the right $E_{\ba}(B) \subseteq E_{\ba}(K_{\infty})$ via $B \hookrightarrow K_{\infty}$.

Now $\Lie(E_{\ba})(K_{\infty})$ has the structure of a $k_{\infty}$-vector space with
\begin{equation} \label{E:Liedim}
\dim_{k_{\infty}} \Lie(E_{\ba})(K_{\infty}) = [K:k] \cdot \dim(E_{\ba}) = rd.
\end{equation}
Also, $\Lie(E_{\ba})(K_{\infty})$ is an $\bA$-module via $\rd E_{\ba} : \bA \to \Mat_d(K)$. Letting $\bk_{\infty} = \laurent{\FF_q}{t^{-1}}$, this map extends to $\rd E_{\ba} : \bk_{\infty} \to \Mat_d(K_{\infty})$ via
\[
\rd E_{\ba} : \sum_{i=i_0}^{\infty} c_i t^{-i} \mapsto \sum_{i=i_0}^{\infty} c_i \bigl( \rd E_{\ba,t} \bigr)^{-i},
\]
which converges by~\cite{Fang15}*{Lem.~1.7}. This makes $\Lie(E_{\ba})(K_{\infty})$ into a $\bk_{\infty}$-vector space, also of dimension $rd$, in which $\Lie(E_{\ba})(B)$ is an $\bA$-lattice~\cite{Fang15}*{Thm.~1.10}. Thus, $\Lie(E_{\ba})(B)$ is simultaneously an $A$-lattice and an $\bA$-lattice. When $\rd E_{\ba,t} = \theta \rI$, these two structures are essentially the same, up to interchanging the roles of $t$ and $\theta$. 

By~\citelist{\cite{Taelman10}*{Thm.~1} \cite{Fang15}*{Thm.~1.10}}, one finds that $\rU(E_{\ba}/B)$ is an $\bA$-lattice in $\Lie(E_{\ba})(K_{\infty})$. The \emph{regulator} of $E_{\ba}$ is
\[
\Reg(E_{\ba}/B) \assign \bigl[ \Lie(E_{\ba})(B) : \rU(E_{\ba}/B)\bigr]_{\bA} \big|_{t=\theta} \in k_{\infty}^{\times}.
\]
\end{subsubsec}

\begin{subsubsec}{Class module of $E_{\ba}$}
As noted in~\cite{ANT22}*{\S 2.1}, the map $\Exp_{E_{\ba}}$ induces an exact sequence of $\bA$-modules
\[
0 \to \rU(E_{\ba}/B) \to \Lie(E_{\ba})(K_{\infty}) \to
\frac{E_{\ba}(K_{\infty})}{E_{\ba}(B)} \to
\frac{E_{\ba}(K_{\infty})}{E_{\ba}(B) + \Exp_{E_{\ba}}(\Lie(E_{\ba})(K_{\infty}))} \to 0.
\]
This last term
\[
\rH(E_{\ba}/B) \assign \frac{E_{\ba}(K_{\infty})}{E_{\ba}(B) + \Exp_{E_{\ba}}(\Lie(E_{\ba})(K_{\infty}))},
\]
is the \emph{class module} of $E_{\ba}/B$. Fang~\cite{Fang15}*{Thm.~1.10} and Taelman~\cite{Taelman10}*{Thm.~1} showed that it is a finite $\bA$-module. We let
\[
\rh(E_{\ba}/B) \assign [ \rH(E_{\ba}/B)]_A \quad (\in A_+).
\]
The fundamental result is the class formula, which we have specialized to our setting.
\end{subsubsec}

\begin{theorem}[{Fang~\cite{Fang15}*{Thm.~1.10}, Taelman~\cite{Taelman12}*{Thm.~1}}] \label{T:TaelmanFang}
Let $\ba \in \cA_f$ be effective with $\deg \ba > 0$, and let $E_{\ba} : \bA \to \Mat_d(B[\tau])$ be a model of $E_{\ba}$ defined over~$B$. Then
\[
\prod_{\fp} \frac{[\Lie(\oE_{\ba})(\FF_{\fp})]_A}{[\oE_{\ba}(\FF_{\fp})]_A} = \Reg(E_{\ba}/B) \cdot \rh(E_{\ba}/B),
\]
where the product is taken over all primes of $B$ and converges in $k_{\infty}^{\times}$.
\end{theorem}

\subsection{Special \texorpdfstring{$L$}{L}-value identities for \texorpdfstring{$E_{\ba}$}{Ea}} \label{SS:specialLvalues}
Let $\ba \in \cA_f$ be effective with $\deg \ba > 0$. We now analyze the terms in Theorem~\ref{T:TaelmanFang} in detail and use them to prove identities for $L(E_{\ba}^{\vee},0)$ and $L(\psi_{\ba},0)$. At first we allow $\ba$ to be arbitrary, but for our main results, Theorem~\ref{T:GossLMain} and Corollary~\ref{C:HeckeLMain}, we require that $\ba = [a/f]$ be basic.

\begin{subsubsec}{Model for $E_{\ba}$ over $B$} \label{SSS:EaBmodel}
We assume that $E_{\ba} : \bA \to \Mat_{d}(B[\tau])$ has been chosen in the following way. First we choose a model for $E_{\ba}$ using a normalized $\ok[\sigma]$-basis defined over $K$ as in Theorem~\ref{T:EaKdef}. Then similar to the beginning of \S\ref{SS:reduction}, we choose $\alpha \in A_+$ such that conjugating $E_{\ba}$ by $\alpha\rI$ yields a model defined over~$B$. We require that for $\fp\in \Sigma_{\ba}$ we have $\oE_{\ba} \cong \rd \oE_{\ba}$. Also, if $\fp \in \Sigma_{\ba}$ lies above above a prime $\wp$ of $A$, then all of the primes in $B$ above $\wp$ are also in $\Sigma_{\ba}$. These conditions can be achieved by building $\alpha$ up by powers of  primes in $A$ one prime at a time.
\end{subsubsec}

\begin{subsubsec}{Local factors}
Let $\fp$ be a prime of $B$, lying above $\wp \in A_+$. We investigate the terms in the product in Theorem~\ref{T:TaelmanFang}.
\end{subsubsec}

\begin{lemma} \label{L:EaFittings}
For a prime $\fp$ of $B$ above a prime $\wp \in A_+$, the following hold.
\begin{alphenumerate}
\item For all such $\fp$,
\[
\bigl[ \Lie(\oE_{\ba})(\FF_{\fp}) \bigr]_A = \wp^{d[\FF_{\fp}:\FF_{\wp}]}.
\]
\item For $\fp \notin \Sigma_{\ba}$,
\[
\bigl[ \oE_{\ba}(\FF_{\fp}) \bigr]_A = \gamma P_{\ba,\fp}(1) = \gamma \prod_{b \in (A/fA)^{\times}} \bigl( 1 - \chi_{\ba}^{\rho_{b}}(\fp) \bigr),
\]
where $\gamma \in \FF_q^{\times}$ is chosen to make the right-hand expressions monic.
\end{alphenumerate}
\end{lemma}

\begin{proof}
Part (a) is a straightforward modification of the proof of~\cite{Demeslay22}*{Prop.~4.10}. As the eigenvalues of $\rd\oE_{\ba,\wp(t)}$ are all zero in $\FF_{\fp}$, it follows that
\[
\rd\oE_{\ba,\wp(t)^d} = (\rd\oE_{\ba,\wp(t)})^d = 0.
\]
Thus $\Lie(\oE_{\ba})(\FF_{\fp})$ is annihilated by $\wp(t)^d$. We then have an isomorphism of $\bA$-modules $\Lie(\oE_{\ba})(\FF_{\fp}) \cong \bA/\wp(t)^{e_1}\bA \oplus \cdots \oplus \bA/\wp(t)^{e_m}\bA$ with $e_1 \leqslant \cdots \leqslant e_m \leqslant d$. On the one hand,
\[
\dim_{\FF_q} \Lie(\oE_{\ba})(\FF_{\fp}) = d \cdot [\FF_{\fp}:\FF_q] = d \cdot [\FF_{\fp}:\FF_{\wp}] \cdot \deg \wp,
\]
and on the other, $\dim_{\FF_q} \bA/\wp(t)^{e_1}\bA \oplus \cdots \oplus \bA/\wp(t)^{e_m}\bA = (e_1 + \cdots + e_m)\deg \wp$.

In Proposition~\ref{P:charpolys}, we observed that $P_{\ba,\fp}(X) = \Char(\tau^\ell,T_{\lambda}(\oE_{\ba}),X)|_{t=\theta}$ when $\fp$ is a good prime. Part (b) follows from Proposition~\ref{P:charpolys} combined with~\cite{HuangP22}*{Cor.~3.7.8}.
\end{proof}

\begin{proposition} \label{P:Lvalues}
For $\ba \in \cA_f$ effective with $\deg \ba > 0$, fix a model $E_{\ba}: \bA \to \Mat_d(B[\tau])$ as in \S\ref{SSS:EaBmodel}. Then
\[
\prod_{\fp} \frac{[\Lie(\oE_{\ba})(\FF_{\fp})]_A}{[\oE_{\ba}(\FF_{\fp})]_A} = L(E_{\ba}^{\vee}/K,\Sigma_{\ba},0) = L(\psi_{\ba},\Sigma_{\ba},0)^r.
\]
\end{proposition}

\begin{proof}
For $\fp \in \Sigma_{\ba}$ the term in the left-hand product is~$1$ since $\oE_{\ba} \cong \rd\oE_{\ba}$. For $\fp \notin \Sigma_{\ba}$, we have $\prod_{b \in (A/fA)^{\times}} \chi_{\ba}^{\rho_b}(\fp) = N^K_k ( \chi_{\ba}(\fp) )$, which generates the ideal $N_k^K(\chi_{\ba}(\fp)B) \cap A$ of~$A$. By Theorem~\ref{T:AndHecke}(a), we see
\[
N_k^K \bigl( \chi_{\ba}(\fp)B \bigr) \cap A = N^K_k \bigl( \fp^{\Theta_{\ba}} \bigr) \cap A = \bigl( N^K_k(\fp)\bigr)^{\Theta_{\ba}} \cap A = \wp^{[\FF_{\fp}:\FF_\wp] d}A.
\]
Thus $N^K_k(\chi_{\ba}(\fp))$ equals $\wp^{[\FF_{\fp}:\FF_\wp]d}$ up to a constant from~$\FF_q^{\times}$. Lemma~\ref{L:EaFittings} then implies
\[
\frac{[\Lie(\oE_{\ba})(\FF_{\fp})]_A}{[\oE_{\ba}(\FF_{\fp})]_A}
= \prod_{b \in (A/fA)^{\times}} \frac{-\chi_{\ba}^{\rho_b}(\fp)}{1-\chi_{\ba}^{\rho_b}(\fp)} = \prod_{b \in (A/fA)^{\times}} \frac{1}{1 - \psi_{\ba}(\fp)} = Q_{\ba,\fp}^{\vee}(1)^{-1}.
\]
As for what happened to $\gamma$ from Lemma~\ref{L:EaFittings}(b), it has been accounted for with the minus sign in the middle term: $\sgn(-\chi_{\ba}^{\rho_b}(\fp)/(1-\chi_{\ba}^{\rho_b}(\fp))) = 1$ since $\inorm{\chi_{\ba}^{\rho_b}(\fp)} > 1$ by Theorem~\ref{T:AndHecke}(b). The result follows from~\eqref{E:LEaKSigmadef} and Proposition~\ref{P:Lidentities}(d).
\end{proof}

\begin{subsubsec}{Regulator of $E_{\ba}$} \label{SSS:RegEa}
Henceforth we will assume that $\ba = [a/f]$ is basic. In this case the dimension of $E_{\ba}$ is $d = r/(q-1) = n = [K^{+}:k]$ by Lemma~\ref{L:Haprops}(b). Also, Lemma~\ref{L:dEab} implies that the $\bA$-module structure on $\Lie(E_{\ba})(K_{\infty})$ coincides with the scalar $A$-module structure on $K_{\infty}^d$. Recalling that we have constructed $E_{\ba}$ as in \S\ref{SSS:EaBmodel}, the period lattice for this $E_{\ba}$ model is then $\alpha^{-1} \Lambda_{\ba}$ from Theorem~\ref{T:periods}.

Recall from \S\ref{SSS:GammaPi} that $\Pi(x) \in k_{\infty}^{\times}$ for $x \in k \setminus {A_+}$. Considering Theorem~\ref{T:periods}(b), if
\[
\rD_i \assign \begin{pmatrix}
\ddots & &  \\
& \rho_c(\zeta)^{i-1} & \\
 & & \ddots 
\end{pmatrix}_{c \in \cP_{\ba}} \in \Mat_{d}(B),
\quad \bspi \assign \begin{pmatrix}
\vdots \\ \Pi \bigl( [ca/f] \bigr) \\ \vdots
\end{pmatrix}_{c \in \cP_{\ba}} \in k_{\infty}^d,
\]
then for $1 \leqslant i \leqslant r$, we have $\pi_i = \rD_i \bspi$. Thus for $b \in I_f$, we can embed
\[
\varepsilon_b : \alpha^{-1} \Lambda_{\ba} \hookrightarrow K_{\rho_b(\infty_1)}^d
\]
via $a_1 \alpha^{-1}\pi_1 + \cdots + a_r \alpha^{-1} \pi_r \mapsto a_1 \alpha^{-1} \rho_b (\rD_1) \bspi + \cdots + a_r \alpha^{-1} \rho_b(\rD_r) \bspi$. As in \S\ref{SSS:regulators}, taking $\varepsilon \assign \oplus \varepsilon_b$ we can then embed
\begin{equation} \label{E:epsilonembed}
\varepsilon : \bigl( \alpha^{-1}\Lambda_{\ba} \bigr)^d \hookrightarrow K_{\infty}^d \cong \biggl( \bigoplus_{\substack{b \in A_+ \\ \deg b < \deg f \\ (b,f) = 1}} K_{\rho_b(\infty_1)}^d \biggr).
\end{equation}
Letting $\Lambda \assign \im \varepsilon$, we note that $\Lambda \subseteq \ker \Exp_{E_{\ba}}$, and thus
\[
\Lambda \subseteq \rU(E_{\ba}/B).
\]
Furthermore by \eqref{E:Liedim}, $\rank_{\bA} \Lambda = rd = \dim_{k_\infty} \Lie(E_{\ba})(K_{\infty}) = \rank_{\bA} \rU(E_{\ba}/B)$. Therefore, $\Lambda$ has finite index in $\rU(E_{\ba}/B)$ and is an $\bA$-lattice in $\Lie(E_{\ba})(K_{\infty})$.
\end{subsubsec}

\begin{remark} \label{R:basic}
It is here that $\ba = [a/f]$ being basic is significant. In~\eqref{E:epsilonembed}, the tuples in $(\alpha^{-1} \Lambda_{\ba})^d$ are indexed by $b \in I_f$, which has $r/(q-1)$ elements. However, the $d$ in $K_{\infty}^d$ on the right is coming from the dimension of $E_{\ba}$. When $\ba$ is basic, $d = r/(q-1)$, and $\Lambda$ is an $\bA$-lattice. When $\ba$ is more general, one can still construct $\Lambda \subseteq \rU(E_{\ba}/B)$, but its rank over $\bA$ will be $r^2/(q-1)$, which generally will be strictly smaller than $rd$.
\end{remark}

\begin{subsubsec}{The completion $k_{\infty}(\zeta)$} \label{E:kinftyzeta}
Elements of $\FF_q^{\times} \subseteq \Gal(K/k)$ fix powers $\zeta^{m}$ with $(q-1)\mid m$, and so in these cases $\zeta^m \in k_{\infty}$. Furthermore, $\{ 1, \zeta, \dots, \zeta^{q-2}\}$ is a $k_{\infty}$-basis of $k_{\infty}(\zeta)$, as $[k_{\infty}(\zeta):k_{\infty}] = q-1$ by \S\ref{SSS:infinity}, as well as a $k$-basis of $K^+=k(\zeta^{q-1})$.
\end{subsubsec}

\begin{proposition} \label{P:RegCalc}
Let $\ba = [a/f] \in \cA_f$ be basic. Fix a model $E_{\ba} : \bA \to \Mat_d(B[\tau])$ chosen as in \S\ref{SSS:EaBmodel}.
\begin{alphenumerate}
\item There is $C_{\rR} \in \ok^{\times}$ with $\sgn(C_{\rR})=1$ and
\[
C_{\rR}^{r} \in \begin{cases}
k^{\times} &\text{if $r$ is even,} \\ 
K^{\times}\ \text{and}\ C_{\rR}^{2r} \in k^{\times} &\text{if $r$ is odd,}
\end{cases}
\]
such that
\[
\Reg(E_{\ba}/B) = C_{\rR}^r \cdot \prod_{c \in \cP_{\ba}} \Pi \bigl( [ca/f] \bigr)^r.
\]
\item Furthermore, if $f/(a,f) = f' \in A_+$, then
\[
\Reg(E_{\ba}/B) = C_{\rR}^r \cdot \prod_{\substack{b' \in A_+ \\ \deg b' < \deg f' \\ (b',f')=1}} \Pi ( b'/f')^{[K_f:K_{f'}]r}
\]
\end{alphenumerate}
\end{proposition}

\begin{proof}
We continue with the notation of \S\ref{SSS:RegEa}. From~\S\ref{SSS:covol},
\begin{equation} \label{E:RegEatoLambda}
\Reg(E_{\ba}/B) = \frac{[\Lie(E_{\ba})(B) : \Lambda]_A}{[\rU(E_{\ba}/B) : \Lambda]_A},
\end{equation}
with $[\rU(E_{\ba}/B):\Lambda]_A \in A_+$. We calculate $[\Lie(E_{\ba})(B) : \Lambda]_A$. An $\bA$-basis for $\Lambda$ consists of vectors of the form
\begin{equation} \label{E:Lambdabasis}
\bigl(0, \dots, 0, \alpha^{-1} \rho_b(\rD_i)\bspi,0 \dots, 0  \bigr)^{\tr} \in \bigoplus_{b \in I_f} K_{\rho_b(\infty_1)}^d,
\end{equation}
where the non-zero entry is in the $b$-th component. We consider first the component corresponding to $b=1$, which is simply $k_{\infty}(\zeta)^d$. By Theorem~\ref{T:periods}(b) a short calculation yields that the block matrix in $\Mat_r(k_{\infty})$ (with row blocks indexed by $c \in \cP_{\ba})$,
\[
M_1 = \alpha^{-1} \Bigl(
\Pi\bigl( [ca/f] \bigr) \rI_{q-1}, \rho_c(\zeta)^{q-1}\Pi\bigl( [ca/f] \bigr) \rI_{q-1} ,\ldots, \rho_c(\zeta)^{(q-1)(d-1)} \Pi\bigl( [ca/f] \bigr) \rI_{q-1}
\Bigr)_{c \in \cP_{\ba}},
\]
represents the elements $\alpha^{-1} \rD_1 \bspi, \dots, \alpha^{-1} \rD_r \bspi$ of our $\bA$-basis for $\Lambda$ in $k_{\infty}(\zeta)^d$ with respect to $\{ 1, \zeta, \dots, \zeta^{q-2} \}$. Properties of determinants yield
\[
\det M_1 = \alpha^{-r} \prod_{c \in \cP_{\ba}} \Pi\bigl([ca/f] \bigr)^{q-1}
\cdot \det \Bigl( \rI_{q-1}, \rho_c(\zeta)^{q-1} \rI_{q-1}, \dots, \rho_c(\zeta)^{(q-1)(d-1)} \rI_{q-1} \Bigr)_{c \in \cP_{\ba}}.
\]
The determinant on the right is the determinant of the Kronecker product of the $d \times d$ Vandermonde matrix $V_1$ for $\{ \rho_c(\zeta)^{q-1} \mid c \in \cP_{\ba} \}$ and $\rI_{q-1}$. Thus,
\[
\det M_1 = \alpha^{-r} \det(V_1)^{q-1} \prod_{c \in \cP_{\ba}} \Pi\bigl([ca/f] \bigr)^{q-1}.
\]
An observation from~\eqref{E:XiWaf} and~\eqref{E:XiWafpullback} is that 
\begin{equation} \label{E:rhoc=rhob}
\{ \rho_c|_{K^+} : c \in \cP_{\ba} \} = \{ \rho_b|_{K^+} : b \in I_f\} = \Gal(K^+/k)
\end{equation}
(cf.~\cite{BP02}*{Thm.~4.5.4}), and so $\det V_1$ is a square-root of the discriminant of the irreducible polynomial for $\zeta^{q-1}$ over~$k$. In particular $(\det V_1)^2 \in A$.

Now $\det M_1$ is the determinant of the $b=1$ component of the $\bA$-basis in~\eqref{E:Lambdabasis} in terms of the $k_{\infty}$-basis of $k_{\infty}(\zeta)^d = K_{\rho_1(\infty_1)}^d$ built from $\{ 1, \zeta, \dots, \zeta^{q-2} \}$. To consider the arbitrary $b$-th component in~\eqref{E:Lambdabasis}, the calculations are exactly the same, but with $\zeta$ replaced by $\rho_b(\zeta)$. This yields a matrix $M_b \in \Mat_r(k_{\infty})$ as above with Vandermonde matrix $V_b$ for $\{ \rho_c (\rho_b(\zeta))^{q-1} \mid c \in \cP_{\ba} \}$ such that
\[
\det M_b = \alpha^{-r} \det(V_b)^{q-1} \prod_{c \in \cP_{\ba}} \Pi\bigl([ca/f] \bigr)^{q-1}.
\]
By~\eqref{E:rhoc=rhob}, we see that $\det V_b = \det V_1$ up to $\pm 1$. Pulling everything together and using~\eqref{E:covoldet}, we see that
\begin{equation} \label{E:Lambdacovol}
\bigl[ \Lie(E_{\ba})(B) : \Lambda \bigr]_A = \prod_{b \in I_f} \det M_b
= \alpha^{-rd} \det(V_1)^{r} \prod_{c \in \cP_{\ba}} \Pi \bigl( [ca/f] \bigr)^{r}.
\end{equation}
Take $C_{\rR} = \alpha^{-d} \det(V_1)/[\rU(E_{\ba}/B):\Lambda]_A^{1/r}$ chosen with sign~$1$. We note $C_{\rR}^r\in k^{\times}$ when~$r$ is even, as $\det(V_1)^2 \in A$. Otherwise, $C_{\rR}^r \in K^{\times}$ and $C_{\rR}^{2r} \in k^{\times}$, and (a) is done by~\eqref{E:RegEatoLambda}.

For (b), if we write $a = \gamma a'$, then by reduction modulo $f'$,
\begin{multline*}
\cP_{\ba} = \{ c \in A \mid \deg c < \deg f,\ (c,f)=1,\ \langle ca/f \rangle = 1 \} \\
\twoheadrightarrow \{ c' \in A \mid \deg c' < \deg f',\ (c',f')=1,\ \langle c'a'/f' \rangle = 1\},
\end{multline*}
and this map is $[K_f:K_{f'}]$-to-one. Also, since $(a',f')=1$, if $b' \in A_+$, $\deg b' < \deg f'$, $(b',f')=1$, then there a unique $c'$ in the target set such that $c'a' \bmod f' = b'$. In this case, if $c \mapsto c'$, then $\Pi([ca/f]) = \Pi([c'a'/f']) = \Pi(b'/f')$, and the result follows.
\end{proof}

\begin{remark} \label{R:squareroot}
Taelman~\cite{Taelman10}*{Rem.~4} observed that in the analogy with unit groups of number fields, his regulator also accounts for the square root of the discriminant in the usual residue formula for the Dedekind zeta function. We see this here as the constant~$C_{\rR}$ contains terms coming from $\det V_1$, which is a $k^{\times}$-multiple of a square root of the discriminant of $K^+/k$ (cf.~Anderson's $\Omega_k(n,\ba)$ in \cite{And86Taniyama}*{Thm.~2.4}).
\end{remark}

\begin{theorem} \label{T:GossLMain}
Let $\ba = [a/f] \in \cA_f$ be basic, and fix $E_{\ba}: \bA \to \Mat_d(B[\tau])$ as in \S\ref{SSS:EaBmodel}. Let $C_{\rR}$ be chosen as in Proposition~\ref{P:RegCalc}, and pick $C_{\rH} \in \ok^{\times}$ with $\sgn(C_{\rH}) = 1$ so that $\rh(E_{\ba}/B) = C_{\rH}^r$. Then there exists $C_{\rE} \in K^{\times}$ with $C_{\rE}^{r} \in k^{\times}$ and $\sgn(C_{\rE}) = 1$ so that the following hold.
\begin{alphenumerate}
\item If $(a,f) = 1$, then
\[
L(E_{\ba}^{\vee}/K,0) = C_{\rR}^r \cdot C_{\rH}^r \cdot C_{\rE}^r  \cdot \prod_{\substack{b \in A_+ \\ \deg b < \deg f \\ (b,f)=1}} \Pi(b/f)^r.
\]
\item More generally, if $f/(a,f) = f' \in A_+$, then
\[
L(E_{\ba}^{\vee}/K,0) = C_{\rR}^r \cdot C_{\rH}^r \cdot C_{\rE}^r \cdot \prod_{\substack{b' \in A_+ \\ \deg b' < \deg f' \\ (b',f')=1}} \Pi(b'/f')^{[K_f:K_{f'}]r}.
\]
\end{alphenumerate}
\end{theorem}

\begin{proof}
Combining Theorem~\ref{T:TaelmanFang} with Propositions~\ref{P:Lvalues} and~\ref{P:RegCalc}, we obtain these same identities for $L(E_{\ba}^{\vee}/K,\Sigma_{\ba},0)$, but without the factor of $C_{\rE}^r$, which comes from the extra Euler factors for $L(E_{\ba}^{\vee}/K,0)$ in~\eqref{E:LEaKdef}. Then $C_{\rE} \in K^{\times}$ by~\eqref{E:AndHeckeLdef} and Remark~\ref{R:LnoSigmas}.
\end{proof}

\begin{remark}
When $r$ is even, we see that $L(E_{\ba}^{\vee}/K,0)$ is a $k^{\times}$-multiple of special $\Pi$-values. When $r$ is odd, this multiple is a square root of an element of $k^{\times}$.

Our main corollary follows immediately from Remark~\ref{R:LnoSigmas} and Theorem~\ref{T:GossLMain}.
\end{remark}

\begin{corollary} \label{C:HeckeLMain}
Under the same conditions as Theorem~\ref{T:GossLMain} the following hold.
\begin{alphenumerate}
\item If $(a,f) = 1$, then
\[
L(\psi_{\ba},0) = C_{\rR} \cdot C_{\rH} \cdot C_{\rE} \cdot \prod_{\substack{b \in A_+ \\ \deg b < \deg f \\ (b,f)=1}} \Pi(b/f).
\]
\item More generally, if $f/(a,f) = f' \in A_+$, then
\[
L(\psi_{\ba},0) = C_{\rR} \cdot C_{\rH} \cdot C_{\rE} \cdot \prod_{\substack{b' \in A_+ \\ \deg b' < \deg f' \\ (b',f')=1}} \Pi(b'/f')^{[K_f:K_{f'}]}.
\]
\end{alphenumerate}
\end{corollary}

\section{Examples} \label{S:examples}

We provide examples of Anderson Hecke characters, Sinha modules, and special $L$-values. One should compare with previous work in~\citelist{\cite{ABP04}*{\S 6.6} \cite{BP02}*{\S 6} \cite{Sinha97b}*{\S 3.3}}, and additional details can be found in~\cite{DavisPhD}*{\S 6}.

\subsection{The case \texorpdfstring{$\ba = [1/\theta]$}{a=[1/theta]}}
Taking $f = \theta$ and $\ba = [1/\theta]$, we have $\zeta = \zeta_{\theta} = (-\theta)^{1/(q-1)}$. We find that $\cD_{\theta}(t,z) = t + z^{q-1}$. The Coleman function is
\[
g_{[1/\theta]} = 1 - \frac{z}{\zeta}
\]
with $W_{[1/\theta]} = 0$ and $\Xi_{[1/\theta]} = \xi$. In this case $B = \FF_q[\zeta]$ is a principal ideal domain, and after some calculations involving the techniques from \S\ref{SS:AndHecke}, we find for $\fp = p(\zeta)B$ a prime ideal of $B$ with $p(\zeta) \in B_+$ prime to $\theta$ that
\begin{equation} \label{E:psitheta}
\psi_{[1/\theta]}(\fp) = \frac{p(0)}{p}.
\end{equation}
The $t$-comotive $H_{[1/\theta]}(\ok)$ is simply $\ok[t,z]$ with the usual $\sigma$-action, and it has rank~$q-1$, dimension~$1$, and $\ok[\sigma]$-basis $\{1 \}$ defined over $K$. We find (cf.~\cite{BP02}*{\S 6.1}) that
\begin{equation}
E_{[1/\theta],z} = \zeta - \zeta^q \tau, \quad
E_{[1/\theta],t} = \theta(1+ \theta \tau)^{q-1}.
\end{equation}
Thus $E_{[1/\theta]}$ is essentially a twist of the Carlitz module for $B$. One finds that
\begin{equation} \label{E:Ltheta}
L(E_{[1/\theta]}^{\vee}/K,0) = \Pi(1/\theta)^{q-1}, \quad
L(\psi_{[1/\theta]},0) = \Pi(1/\theta),
\end{equation}
and so the constants in Theorem~\ref{T:GossLMain} and Corollary~\ref{C:HeckeLMain} are all~$1$. Although we omit the details, one way to prove these identities is to use specializations of Pellarin's $L$-series for $\bB = \FF_q[z]$ at $z=0$, since that is essentially how $\psi_{[1/\theta]}$ is given in~\eqref{E:psitheta}. Pellarin's main identity~\cite{Pellarin12}*{Thm.~1} then leads to~\eqref{E:Ltheta}. We wonder if there is a more general connection between $L(\psi_{\ba},0)$ and deformation $L$-series as studied in~\cites{ANT17b, ANT20, ANT22, APT18, AT17, Demeslay22, HuangP22}.

\subsection{The case \texorpdfstring{$\ba = [1/(\theta^2-\theta)]$}{a=[1/(theta2-theta)]}, \texorpdfstring{$q=3$}{q=3}}
Consider $f=\theta(\theta-1)$ and $\ba = [1/f]$ when $q=3$. Then $K = k(\zeta_f)$ is biquadratic over $k$, containing intermediate quadratic fields $k(\zeta_{\theta}) = k(\sqrt{-\theta})$, $k(\zeta_{\theta-1}) = k(\sqrt{1-\theta})$, and $k(\sqrt{f}) = K^+$. Moreover,
\[
\cD_f(t,z) = z^4 + (t+1)z^2 + 1.
\]
Since $t$ is a rational expression of $z$ in $\bK=\FF_3(t,z)$, we see that $\bK = \FF_3(z)$ and that $X_f = \PP^1$. Then $X_f$ has two points at infinity, $\infty_1$, $\infty_2$, and $\divf(z) = \infty_1 - \infty_2$, so $z \in \bB^{\times}$. Moreover, since $t = -1 -z^2 - z^{-2}$, it follows that $\bB = \FF_3[z,z^{-1}]$. We find $W_{[1/f]} = \xi_1$,  $\Xi_{[1/f]} = \xi_1 + \xi_{\theta+1}$, and
\[
g_{[1/f]} = (1-\zeta^4)^{-1} (z - \zeta^3)(z^{-1} + \zeta).
\]
Using~\eqref{E:chiafdef}, if $\fp$ is a prime ideal of $B = \FF_3[\zeta,\zeta^{-1}]$ generated by $p(\zeta) \in \FF_3[\zeta]_+$, then
\begin{equation} \label{E:psiex2}
\psi_{[1/f]}(\fp) = \frac{p(1)p(-1)p(\sqrt{-1})p(-\sqrt{-1})}{p(\zeta)p(-\zeta^{-1})}.
\end{equation}
The $t$-comotive $H_{\ba}(\ok)$ is the ideal $H^0(\bU,\cO_{\bX}(-\xi_1^{(1)})) = \ok[z,z^{-1}](z-\zeta^3)$, and it has rank~$4$ and dimension~$2$. Using the techniques in \S\ref{SS:Fodef}, we obtain a $\ok[\sigma]$-basis
\begin{align}
h_1 &= (1-\zeta^4)^{-1} (z - \zeta^3)(z^{-1} + \zeta) = g_{[1/f]},
&&\divf(h_1) = \xi_1^{(1)} + \xi_{\theta+1} - \infty_1 - \infty_2, \\
h_2 &= (1+\zeta^2)^{-1}(z - \zeta^3)(z^{-1}-\zeta^{-1}),
&&\divf(h_2) = \xi_1^{(1)} + \xi_1 - \infty_1 - \infty_2. \notag
\end{align}
We note that $h_1(\xi_1) = 1$, $h_1(\xi_{\theta+1}) = 0$, $h_2(\xi_1)=0$, and $h_2(\xi_{\theta+1}) = -(\theta+1)$. The last calculation implies that $\{ h_1, h_2\}$ is almost but not exactly normalized in the sense of~\S\ref{SSS:normalized}. However, this change simplified some of our calculations. We find that
\begin{align}
E_{[1/f],z} &= \begin{pmatrix}
\zeta & 0 \\ 0 & -\zeta^{-1}
\end{pmatrix}
+ \begin{pmatrix}
\zeta-\zeta^3 & -\zeta^{-5} (1-\zeta^2)^4 \\
-\zeta & \zeta^{-5}(1-\zeta^2)^3
\end{pmatrix} \tau, 
\\
E_{[1/f],z^{-1}} &= \begin{pmatrix}
\zeta^{-1} & 0 \\ 0 & -\zeta
\end{pmatrix}
+ \begin{pmatrix}
\zeta^{-3}(1-\zeta^2) & \zeta^{-3} (1-\zeta^2)^4 \\
\zeta^{-1} & \zeta^{-1}(1-\zeta^2)^3
\end{pmatrix} \tau, \notag
\end{align}
and
\begin{multline} \label{E:Eex2t}
E_{[1/f],t} = \begin{pmatrix}
\theta & 0 \\ 0 & \theta
\end{pmatrix}
+ \begin{pmatrix}
-\theta(\theta^2-1) & -\theta^3(\theta-1) \\
\theta(\theta-1) & \theta^2(\theta^2-1)
\end{pmatrix} \tau \\
{}+ \begin{pmatrix}
\theta^2(\theta-1)^2(\theta^2+1) & -\theta^6(\theta^3-\theta)(\theta-1) \\
-\theta (\theta^3-\theta)(\theta -1) & \theta^6 (\theta-1)^2 (\theta^2 + 1)
\end{pmatrix} \tau^2.
\end{multline}
The determinant of the $\tau^2$ coefficient of $E_{[1/f],t}$ simplifies to $f^8$, and thus $E_{[1/f]}$ is defined over $B$ (in fact $E_{[1/f],t}$ is defined over $A$ itself) and is strictly pure in the sense of~\cite{NamoijamP24}*{Ex.~3.38}. Moreover, it follows that the reduction of $E_{[1/f]}$ is good for every prime not dividing~$f=\theta(\theta-1)$. We note that~\cite{BP02}*{\S 6.2} indicates that $E_{[1/f]}$ is isogenous to $E_{[1/\theta]}^{\oplus 2}$, which can be verified from~\eqref{E:Eex2t} after a change of coordinates. For $L$-values we find after some calculation that
\begin{equation}
L(E_{[1/f]}^{\vee},0) \mayeq \frac{\theta^2(\theta-1)^2}{(\theta+1)^4} \cdot \Pi \biggl( \frac{1}{\theta(\theta-1)} \biggr)^4 \Pi \biggl( \frac{\theta+1}{\theta(\theta-1)} \biggr)^4,
\end{equation}
and
\begin{equation}
L(\psi_{[1/f]},0) \mayeq \frac{\sqrt{\theta(\theta-1)}}{\theta+1} \cdot \Pi \biggl( \frac{1}{\theta(\theta-1)} \biggr) \Pi \biggl( \frac{\theta+1}{\theta(\theta-1)} \biggr),
\end{equation}
where ``$\mayeq$'' represents that we have checked these identities numerically (with precision approximately $50$ digits). We note that $\sqrt{\theta(\theta-1)}$ is the square root of the discriminant of $K^+/k$, as we saw in Remark~\ref{R:squareroot}.

\appendix
\section{Primes of good reduction} \label{App:GoodRed}

We prove here that the primes $\fp$ of $B$ not dividing $f$ are primes of good reduction for $E_{\ba}$, as mentioned in \S\ref{SS:reduction}. The key arguments rely on analyses of the relationships between Sinha $t$-comotives over $K$ and over $\FF_{\fp}$.

\subsection{Good reduction of \texorpdfstring{$H_{\ba}$}{Ha} and \texorpdfstring{$E_{\ba}$}{Ea}} \label{SS:appdualtmot}
We review some generalities for reductions of function spaces modulo $\fp$. For a divisor $D \in \Div_K(\bX)$ such that $\cL(D) \subseteq \ok[t,z]$, we let
\[
\cL(D,B_{\fp}) \assign \cL(D) \cap B_{\fp}[t,z].
\]
Because $D$ is defined over $K$, as in \S\ref{SS:AndHecke} it can be interpreted as a sum $\cD$ of horizontal divisors on the curve $\cX = B_{\fp} \times_{\FF_q} X$ over $\Spec B_{\fp}$. As such we can identify $\cL(D,B_{\fp}) = H^0(\cX/B_{\fp},\cO_{\cX/B_{\fp}}(\cD))$, and moreover,
\[
K \otimes_{B_{\fp}} \cL(D,B_{\fp}) = \cL(D) \cap K[t,z] = H^0(X/K,\cO_{X/K}(D)) \rassign \cL(D,K).
\]
The Riemann-Roch theorem~\cite{Rosen}*{Thm.~5.4} implies that $\cL(D)$ has a $\ok$-basis consisting of elements in $K[t,z]$, so in particular
\[
\cL(D) = \ok \otimes_{K} \cL(D,K) = \ok \otimes_{B_{\fp}} \cL(D,B_{\fp}).
\]
Now let
\[
\cL(\oD,\FF_{\fp}) \assign H^0(X,\cO_X(\oD)) \cap \FF_{\fp}[t,z] = H^0(X/\FF_{\fp},\cO_{X/\FF_{\fp}}(\oD)),
\]
and consider the reduction map
\begin{equation} \label{E:pidef}
\pi : \cL(D,B_{\fp}) \to \cL(\oD,\FF_{\fp}).
\end{equation}
Clearly, $\ker \pi = \cL(D,B_{\fp}) \cap \fp B_{\fp}[t,z] = \fp \cL(D,B_{\fp})$. If we let $P = \im \pi$, then $P$ is a finite dimensional $\FF_{\fp}$-vector space, and
\begin{equation} \label{E:LDBp}
P \cong \cL(D,B_{\fp})/\fp\cL(D,B_{\fp}).
\end{equation}
By Nakayama's lemma \cite{LangAlg}*{\S X.4}, if we pick $\alpha_1, \dots, \alpha_n \in \cL(D,B_{\fp})$ such that $\oalpha_1, \dots, \oalpha_n$ is an $\FF_{\fp}$-basis of~$P$, then $\alpha_1, \dots, \alpha_n$ are a $B_{\fp}$-basis of $\cL(D,B_{\fp})$. Hence they also form a $K$-basis of $\cL(D,K)$, and therefore,
\begin{equation} \label{E:reddim}
\dim_{K} \cL(D,K) \leqslant \dim_{\FF_{\fp}} \cL(\oD,\FF_{\fp}).
\end{equation}
Now since $\bX$ is simply the extension of scalars to~$\ok$ of the smooth projective irreducible curve $X/\FF_q$, both $\bX$ and $X$ have the same genus, say~$\gamma$. In the case that $\deg D = \deg \oD > 2\gamma -2$, the Riemann-Roch theorem assures us that~\eqref{E:reddim} is an equality.

We apply these considerations to the situation of \S\ref{SS:Fodef}.
As in~\eqref{E:Midef}, for $i \in \ZZ$ we let $M_i = \cL(-W_{\ba}^{(1)} + i I_{\ba})$ and $\oM_i = \cL( -\oW_{\ba}^{(1)} + i I_{\ba})$. We write
\[
M_i(B_{\fp}) = M_i \cap B_{\fp}[t,z] = \cL(-W_{\ba}^{(1)} + i I_{\ba},B_{\fp}), \quad M_i(K) = M_i \cap K[t,z],
\]
and
\[
\oM_i(\FF_{\fp}) = \oM_i \cap \FF_{\fp}[t,z] = \cL(-\oW_{\ba}^{(1)} + i I_{\ba},\FF_{\fp}).
\]
We have reduction maps for each $i \in \ZZ$,
\begin{equation} \label{E:piidef}
\pi_i : M_i(B_{\fp}) \to \oM_i(\FF_{\fp}).
\end{equation}
For $N \geqslant 0$ taken from Lemma~\ref{L:Miprops}(g) and its proof, we obtain that for $i \geqslant N-1$,
\[
\dim_K M_i(K) = \dim_{\FF_{\fp}} \oM_i(\FF_{\fp}) \quad \Rightarrow \quad
\pi_i(M_i(B_{\fp})) = \oM_i(\FF_{\fp}).
\]
In particular,
\[
\pi_{N-1}(M_{N-1}(B_{\fp})) =\oM_{N-1}(\FF_{\fp}), \quad
\pi_N(M_N(B_{\fp})) = \oM_{N}(\FF_{\fp}).
\]
The maps $\{ \pi_i \}$ form maps of directed systems and induce
\[
\pi : \bigcup_{i=0}^{\infty} M_i(B_{\fp}) \to \bigcup_{i=0}^{\infty} \oM_i(\FF_{\fp}), \quad \bigl(= H_{\ba,\fp}(\FF_{\fp}) \bigr).
\]
We now prove a straightforward lemma on extending bases on two-dimensional arrays of vector spaces. Throughout we let `$\sqcup$' denote a disjoint union.

\begin{lemma} \label{L:triangle-v2}
Let $N \geqslant0$, and let $\{ V_{i,j} \mid 0 \leqslant i \leqslant j \leqslant N \}$ be (finite dimensional) vector spaces over a field~$F$. Suppose we have injective $F$-linear maps
\begin{align*}
\rho_{i,j} : V_{i,j-1} &\hookrightarrow V_{i,j}, \quad &&(0 \leqslant i < j \leqslant N), \\
\delta_{i,j} : V_{i-1,j} &\hookrightarrow V_{i,j}, \quad &&(0 < i \leqslant j \leqslant N),
\end{align*}
such that for $0 < i < j \leqslant N$,
\[
\delta_{i,j} \circ \rho_{i-1,j} = \rho_{i,j} \circ \delta_{i,j-1} : V_{i-1,j-1} \to V_{i,j},
\]
and the induced map
\[
V_{i-1,j-1} \iso V_{i,j-1} \times_{V_{i,j}} V_{i-1,j}
\]
is an isomorphism. Then for $0 \leqslant i \leqslant j \leqslant N$, we can pick (finite) subsets $D_{i,j}$, $R_{i,j}$, $Z_{i,j} \subseteq V_{i,j}$ such that the following hold.
\begin{alphenumerate}
\item $Z_{i,j}$ is a basis of $V_{i,j}$ for $0 \leqslant i \leqslant j \leqslant N$.
\item $Z_{i,j} = \delta_{i,j}(Z_{i-1,j}) \sqcup D_{i,j}$ for $0 < i \leqslant j \leqslant N$.
\item $Z_{i,j} = \rho_{i,j}(Z_{i,j-1}) \sqcup R_{i,j}$ for $0 \leqslant i < j \leqslant N$.
\item For $0 < i < j \leqslant N$, 
\[
Z_{i,j} = (\delta_{i,j} \circ \rho_{i-1,j})(Z_{i-1,j-1}) \sqcup \rho_{i,j}(D_{i,j-1}) \sqcup \delta_{i,j}(R_{i-1,j}) \sqcup (D_{i,j} \cap R_{i,j}).
\]
\end{alphenumerate}
\end{lemma}

\begin{proof}
The diagram to keep in mind is the following,
\begin{equation} \label{E:triangle-v2}
\begin{tikzcd}[row sep=scriptsize, column sep=scriptsize]
V_{0,0} \arrow[r, hook, "\rho_{0,1}"]
&V_{0,1} \arrow[r, hook, "\rho_{0,2}"] \arrow[d, hook, "\delta_{1,1}"]
&V_{0,2} \arrow[r, hook] \arrow[d, hook, "\delta_{1,2}"]
&\cdots \arrow[r,hook]
&V_{0,N} \arrow[d, hook, "\delta_{1,N}"]
\\
& V_{1,1} \arrow[r, hook, "\rho_{1,2}"]
& V_{1,2} \arrow[r, hook] \arrow[d, hook, "\delta_{2,2}"]
& \cdots \arrow[r, hook]
& V_{1,N} \arrow[d, hook, "\delta_{2,N}"]
\\
& & V_{2,2} \arrow[r, hook]
& \cdots \arrow[r, hook]
& V_{2,N} \arrow[d, hook]
\\
& & & \ddots & \vdots \arrow[d, hook]
\\
& & & & V_{N,N}.
\end{tikzcd}
\end{equation}
For $0 \leqslant i \leqslant j \leqslant N$, we define $D_{i,j}$, $R_{i,j}$, and $Z_{i,j}$ recursively. We let $Z_{0,0} = D_{0,0} = R_{0,0}$ be any basis of $V_{0,0}$. For $0 < j \leqslant N$, if $Z_{0,j-1}$ has been defined, we choose $R_{0,j}$ so that $Z_{0,j} \assign \rho_{0,j}(Z_{0,j-1}) \sqcup R_{0,j}$ is a basis for $V_{0,j}$. For convenience we set $D_{0,j} \assign R_{0,j}$. The sets $R_{0,j}$, $Z_{0,j}$ for $0 \leqslant j \leqslant N$ then satisfy parts (a) and~(c) as desired.
For $0 < i \leqslant N$, suppose that $Z_{i-1,i}$ has been defined. We pick $D_{i,i} \subseteq V_{i,i}$ so that $Z_{i,i} \assign \delta_{i,i}(Z_{i-1,i}) \sqcup D_{i,i}$ is a basis of $V_{i,i}$. For convenience we set $R_{i,i} \assign D_{i,i}$.

For $0 < i < j \leqslant N$ suppose that $Z_{i-1,j-1}$, $Z_{i,j-1}$, $Z_{i-1,j}$, $D_{i,j-1}$, and $R_{i-1,j}$ have been defined. As such, by hypothesis
\[
Z_{i,j-1} = \delta_{i,j-1}(Z_{i-1,j-1}) \sqcup D_{i,j-1}, \quad Z_{i-1,j} = \rho_{i-1,j}(Z_{i-1,j-1}) \sqcup R_{i-1,j}.
\]
Consider the augmented square from~\eqref{E:triangle-v2},
\begin{center}
\begin{tikzcd}
V_{i-1,j-1} \arrow[r, hook, "\rho_{i-1,j}"] \arrow[d, hook, "\delta{i,j-1}"]
& V_{i-1,j} \arrow[d, hook]  \arrow[dr, bend left=20, "\delta_{i,j}"]& \\
V_{i,j-1} \arrow[r,hook] \arrow[rr, bend right=20, "\rho_{i,j}"']
& \rho_{i,j}(V_{i,j-1}) + \delta_{i,j}(V_{i-1,j}) \arrow[r, hook] & V_{i,j}.
\end{tikzcd}
\end{center}
Because $V_{i-1,j-1} \cong V_{i,j-1} \times_{V_{i,j}} V_{i-1,j}$, we conclude that
\[
\rho_{i,j}(V_{i,j-1}) \cap \delta_{i,j}(V_{i-1,j}) = (\delta_{i,j} \circ \rho_{i-1,j})(V_{i-1,j-1}).
\]
Thus $(\delta_{i,j} \circ \rho_{i-1,j})(Z_{i-1,j-1}) \sqcup \rho_{i,j}(D_{i,j-1}) \sqcup \delta_{i,j}(R_{i-1,j})$ is a basis of $\rho_{i,j}(V_{i,j-1}) + \delta_{i,j}(V_{i-1,j})$, and so we can pick $Y \subseteq V_{i,j}$ so that
\[
Z_{i,j} \assign (\delta_{i,j} \circ \rho_{i-1,j})(Z_{i-1,j-1}) \sqcup \rho_{i,j}(D_{i,j-1}) \sqcup \delta_{i,j}(R_{i-1,j}) \sqcup Y
\]
is a basis for $V_{i,j}$. We then set
\[
D_{i,j} \assign \rho_{i,j}(D_{i,j-1}) \sqcup Y, \quad
R_{i,j} \assign \delta_{i,j}(R_{i-1,j}) \sqcup Y.
\]
In particular, $D_{i,j} = \rho_{i,j}(D_{i,j-1}) \sqcup (D_{i,j} \cap R_{i,j})$ and $R_{i,j} = \delta_{i,j}(R_{i-1,j}) \sqcup (D_{i,j} \cap R_{i,j})$.
Properties (b)--(d) then follow.
\end{proof}

\begin{proposition} \label{P:Hagoodred}
Let $\ba \in \cA_f$ be effective with $\deg \ba > 0$, and let $\fp$ be a maximal ideal of~$B$ relatively prime to $f$. Then $H_{\ba}(\ok)$ has a $\ok[\sigma]$-basis $\{ h_1, \dots, h_d \} \subseteq B_{\fp}[t,z]$ such that $\{ \oh_1, \dots, \oh_d \} \subseteq \FF_{\fp}[t,z]$ is an $\FF_{\fp}[\sigma]$-basis of $H_{\ba,\fp}(\FF_{\fp})$.
\end{proposition}

\begin{proof}
We construct our $\ok[\sigma]$-basis of $H_{\ba}(\ok)$ through Proposition~\ref{P:HaKdef}, namely using the recursion in \S\ref{SSS:UVTSdefs} and Remark~\ref{R:algorithm}.
We first define a sequence $\{ \oP_i \}_{i \in \ZZ}$ of $\FF_{\fp}$-subspaces of $H_{\ba,\fp}(\FF_{\fp})$. For $i < 0$, we set $\oP_i = \{ 0 \}$, and as in \eqref{E:piidef}, for $i \geqslant 0$ we take
\begin{equation}
\oP_i \assign \pi(M_i(B_{\fp})) = \pi_i(M_i(B_{\fp})).
\end{equation}
We consider the following diagram of finite dimensional $\FF_{\fp}$-vector spaces, where for $0 \leqslant i \leqslant j \leqslant N$, the $i$-th row is filtration of $\oM_{i}(\FF_{\fp})$ and the $j$-th column is a filtration of $\oP_j$,
\begin{equation} \label{E:triangleMP}
\begin{tikzcd}[row sep=scriptsize, column sep=small]
\oP_0 \arrow[r, hook]
&\oM_0(\FF_{\fp}) \cap \oP_1 \arrow[r, hook] \arrow[d, hook]
&\oM_0(\FF_{\fp}) \cap \oP_2 \arrow[r, hook] \arrow[d, hook]
&\cdots \arrow[r,hook]
&\oM_0(\FF_{\fp}) \cap \oP_N = \oM_0(\FF_{\fp}) \arrow[d, hook]
\\
& \oP_1 \arrow[r, hook]
& \oM_1(\FF_{\fp}) \cap \oP_2 \arrow[r, hook] \arrow[d, hook]
& \cdots \arrow[r, hook]
& \oM_1(\FF_{\fp}) \cap \oP_N = \oM_1(\FF_{\fp}) \arrow[d, hook]
\\
& & \oP_2 \arrow[r, hook]
& \cdots \arrow[r, hook]
& \oM_2(\FF_{\fp}) \cap \oP_N = \oM_2(\FF_{\fp}) \arrow[d, hook]
\\
& & & \ddots & \vdots \arrow[d, hook]
\\
& & & & \oM_N(\FF_{\fp}) \cap \oP_N = \oP_N.
\end{tikzcd}
\end{equation}
We are in the situation of Lemma~\ref{L:triangle-v2} with $V_{i,j} = \oM_i(\FF_\fp) \cap \oP_j$. The maps $\delta_{i,j}$, $\rho_{i,j}$, in Lemma~\ref{L:triangle-v2} are inclusions and can be suppressed. Thus, for $0 \leqslant i \leqslant j \leqslant N$, we select
\[
\oD_{i,j},\ \oR_{i,j},\ \oZ_{i,j} \subseteq \oM_i(\FF_{\fp}) \cap \oP_j,
\]
such that $\oZ_{i,j}$ is an $\FF_{\fp}$-basis of $\oM_i(\FF_{\fp}) \cap \oP_j$ and $\oD_{i,j}$, $\oR_{i,j}$ satisfy the conclusions in Lemma \ref{L:triangle-v2}. Accordingly, $\oZ_{0,0} = \oR_{0,0}$, and after repeated applications of Lemma~\ref{L:triangle-v2}(b)--(c),
\[
\oZ_{i,j} = \biggl( \bigsqcup_{0 < \ell \leqslant j} \oR_{0,\ell} \biggr) \sqcup \biggl( \bigsqcup_{0 < k \leqslant i} \oD_{k,j} \biggr), \quad (0 \leqslant i \leqslant j \leqslant N).
\]
In particular for $0 \leqslant i \leqslant N$,
\begin{gather*}
\oP_i = \biggl( \bigoplus_{0 \leqslant \ell \leqslant i} \Span_{\FF_{\fp}} (\oR_{0,\ell}) \biggr) \oplus \biggl( \bigoplus_{0 < k \leqslant i} \Span_{\FF_{\fp}} (\oD_{k,i}) \biggr), \\
\oM_i(\FF_{\fp})  = \biggl( \bigoplus_{0 \leqslant \ell \leqslant N} \Span_{\FF_{\fp}} (\oR_{0,\ell}) \biggr) \oplus \biggl( \bigoplus_{0 < k \leqslant i} \Span_{\FF_{\fp}} (\oD_{k,N}) \biggr) = \oP_i \oplus \biggl( \bigoplus_{i < \ell \leqslant N} \Span_{\FF_{\fp}} (\oR_{i,\ell}) \biggr). \notag
\end{gather*}
Now as $\{ \oM_i(\FF_{\fp}) \}$ satisfies the criteria of \S\ref{Cr:Mi}, we can construct an $\FF_{\fp}[\sigma]$-basis of $H_{\ba,\fp}(\FF_{\fp})$ using Proposition~\ref{P:HaKdef}, namely the construction in~\S\ref{SSS:UVTSdefs} together with Lemma \ref{L:Silinind}. Following along with~\S\ref{SSS:UVTSdefs}, we construct $\FF_{\fp}$-bases $\{ \oS_i \}$ and $\{ \oT_i \}$. We can take
\[
\oT_0 \assign \oR_{0,0} \sqcup \dots \sqcup \oR_{0,N},
\]
and we set $\oT_{0,j} \assign \oR_{0,j}$ for $0 \leqslant j \leqslant N$.
By the construction in~\S\ref{SSS:UVTSdefs}, for $i > 0$, $\oT_i$ is selected by choosing an $\FF_{\fp}$-basis of $\oM_i(\FF_{\fp})/(\oM_{i-1}(\FF_{\fp}) + \sigma \oM_{i-1}(\FF_{\fp}))$.
We have inclusions
\begin{multline*}
\frac{\oP_i + (\oM_{i-1}(\FF_{\fp}) + \sigma \oM_{i-1}(\FF_{\fp}))}{\oM_{i-1}(\FF_{\fp}) + \sigma \oM_{i-1}(\FF_{\fp})} \subseteq \frac{(\oM_i(\FF_{\fp}) \cap \oP_{i+1}) + (\oM_{i-1}(\FF_{\fp}) + \sigma \oM_{i-1}(\FF_{\fp}))}{\oM_{i-1}(\FF_{\fp}) + \sigma \oM_{i-1}(\FF_{\fp})} \\
{}\subseteq \dots \subseteq
\frac{(\oM_i(\FF_{\fp}) \cap \oP_{N}) + (\oM_{i-1}(\FF_{\fp}) + \sigma \oM_{i-1}(\FF_{\fp}))}{\oM_{i-1}(\FF_{\fp}) + \sigma \oM_{i-1}(\FF_{\fp})} = \frac{\oM_i(\FF_{\fp})}{\oM_{i-1}(\FF_{\fp}) + \sigma \oM_{i-1}(\FF_{\fp})}.
\end{multline*}
From this we see that we can select $\oT_i$ to be a union of subsets of bases of
\[
\oP_i \subseteq (\oM_i(\FF_{\fp}) \cap \oP_{i+1}) \subseteq \cdots \subseteq (\oM_i(\FF_{\fp}) \cap \oP_{N-1}) \subseteq \oM_i(\FF_{\fp}).
\]
However, since $\oM_{i-1}(\FF_{\fp}) + \sigma \oM_{i-1}(\FF_{\fp}) \supseteq \oM_{i-1}(\FF_{\fp})$, these subsets do not come from the rows $< i$ in~\eqref{E:triangleMP}. Therefore, for $i \leqslant j \leqslant N$, using Lemma~\ref{L:triangle-v2}(d), we can go back and choose $\oD_{i,j}$, $\oR_{i,j}$, with
\begin{equation} \label{E:oTijdef}
\oT_{i,j} \subseteq \oD_{i,j} \cap \oR_{i,j} \quad \textup{so that} \quad
\oT_i \assign \oT_{i,i} \sqcup \cdots \sqcup \oT_{i,N}
\end{equation}
induces an $\FF_{\fp}$-basis of $\oM_i(\FF_{\fp})/(\oM_{i-1}(\FF_{\fp}) + \sigma \oM_{i-1}(\FF_{\fp}))$ as desired.  By \S\ref{SSS:UVTSdefs} and Remark~\ref{R:algorithm}, $\oS_{0} = \oT_0 = \oT_{0,0} \sqcup \cdots \sqcup \oT_{0,N}$, and for $1 \leqslant i \leqslant N$,
\[
\oS_i = \bigsqcup_{0 \leqslant k \leqslant i} \oT_k =  \bigsqcup_{\substack{0 \leqslant k \leqslant i \\ k \leqslant \ell \leqslant N}} \oT_{k,\ell} \quad \Rightarrow \quad \oS_N = \bigsqcup_{0 \leqslant k \leqslant \ell \leqslant N} \oT_{k,\ell}.
\]
In particular $\oS_N$ is an $\FF_{\fp}[\sigma]$-basis of $H_{\ba,\fp}(\FF_{\fp})$.

Now we consider $M_i(B_{\fp})$ for $0 \leqslant i \leqslant N$, and the $B_{\fp}$-submodules, for $0 \leqslant i \leqslant j \leqslant N$,
\begin{equation}
\cM_{i,j} \assign \pi_j^{-1} \bigl( \oM_{i}(\FF_{\fp}) \bigr) = \pi_j^{-1} \bigl( \oM_i(\FF_{\fp}) \cap \oP_j \bigr) \subseteq M_j(B_{\fp}).
\end{equation}
We can fashion a similar diagram to~\eqref{E:triangleMP} by placing $\cM_{i,j}$ in each entry, and we note from the definition of $\oP_i$ that
\[
\cM_{i,i} = \pi_i^{-1} \bigl( \oM_{i}(\FF_{\fp}) \cap \oP_i \bigr) = M_i(B_{\fp}), \quad 0 \leqslant i \leqslant N.
\]
Furthermore, for $0 \leqslant i \leqslant j \leqslant N$,
\[
\pi_j (\cM_{i,j}) = \oM_i(\FF_{\fp}) \cap \oP_j,
\]
as $\pi_j$ surjects onto $\oP_j$ by construction. For $0 \leqslant j \leqslant N$, by the discussion in~\eqref{E:LDBp}, Nakayama's lemma implies that $M_j(B_{\fp})$ is a free $B_{\fp}$-module of finite rank, and
\[
\rank_{B_{\fp}} M_j(B_{\fp}) = \dim_{\FF_{\fp}} \oP_j, \quad (0 \leqslant j \leqslant N).
\]
More generally, for $0 \leqslant i \leqslant j \leqslant N$, we have the exact sequence of $B_{\fp}$-modules,
\begin{equation} \label{E:Mijseq}
0 \to \cM_{i,j} \cap \fp M_j(\FF_{\fp}) \to \cM_{i,j} \xrightarrow{\pi_j} \oM_{i}(\FF_{\fp}) \cap \oP_j \to 0,
\end{equation}
and
\begin{equation} \label{E:2ndiso}
\frac{\cM_{i,j}}{\cM_{i,j} \cap \fp M_j(B_{\fp})} \cong
\frac{\cM_{i,j} + \fp M_j(B_{\fp})}{\fp M_j(B_{\fp})} \subseteq
\frac{M_j(B_{\fp})}{\fp M_j(B_{\fp})}.
\end{equation}
Thus,
\[
\dim_{\FF_{\fp}} \biggl( \frac{\cM_{i,j} + \fp M_j(B_{\fp})}{\fp M_j(B_{\fp})} \biggr) = \dim_{\FF_{\fp}} \bigl( \oM_i(\FF_{\fp}) \cap \oP_j \bigr)
\leqslant \rank_{B_{\fp}} M_j(B_{\fp}).
\]
For $0 \leqslant i \leqslant j \leqslant N$, we set
\[
\cV_{i,j} \assign \frac{\cM_{i,j} + \fp M_j(B_{\fp})}{\fp M_j(B_{\fp})} = \frac{\pi_j^{-1}(\oM_i(\FF_{\fp})) + \fp M_j(B_{\fp})}{\fp M_j(B_{\fp})},
\]
and we note
\[
\cV_{i,i} = \frac{M_i(B_{\fp}) + \fp M_i(B_{\fp})}{\fp M_i(B_{\fp})} = \frac{M_i(B_{\fp})}{\fp M_i(B_{\fp})}, \quad (0 \leqslant i \leqslant N).
\]
We define $\FF_{\fp}$-linear maps
\[
\rho_{i,j} : \cV_{i,j-1} \to \cV_{i,j},\ (0 \leqslant i < j \leqslant N), \quad
\delta_{i,j} : \cV_{i-1,j} \to \cV_{i,j}, \ (0 < i \leqslant j \leqslant N),
\]
by the natural maps
\begin{align*}
\rho_{i,j} &: \frac{\pi_{j-1}^{-1}(\oM_i(\FF_{\fp})) + \fp M_{j-1}(B_{\fp})}{\fp M_{j-1}(B_{\fp})} \to \frac{\pi_j^{-1}(\oM_i(\FF_{\fp})) + \fp M_j(B_{\fp})}{\fp M_j(B_{\fp})}, \\
\delta_{i,j} &: \frac{\pi_{j}^{-1}(\oM_{i-1}(\FF_{\fp})) + \fp M_{j}(B_{\fp})}{\fp M_{j}(B_{\fp})} \to \frac{\pi_j^{-1}(\oM_i(\FF_{\fp})) + \fp M_j(B_{\fp})}{\fp M_j(B_{\fp})}.
\end{align*}
It is clear that $\delta_{i,j}$ is injective for $0 < i \leqslant j \leqslant N$. Furthermore, for $0 \leqslant i < j \leqslant N$,
\[
\bigl( \pi_{j-1}^{-1} \bigl( \oM_i(\FF_{\fp}) \bigr) + \fp M_{j-1}(B_{\fp}) \bigr) \cap \fp M_j(B_{\fp}) \subseteq M_{j-1}(B_{\fp}) \cap \fp M_j(B_{\fp}) = \fp M_{j-1} (B_{\fp}),
\]
so $\rho_{i,j}$ is injective. Now for $0 < i < j \leqslant N$, we have $\delta_{i,j} \circ \rho_{i-1,j} = \rho_{i,j} \circ \delta_{i,j-1} : \cV_{i-1,j-1} \to \cV_{i,j}$, and we consider the induced map, $\phi : \cV_{i-1,j-1} \to \cV_{i,j-1} \times_{\cV_{i,j}} \cV_{i-1,j}$. Moreover, we have
\begin{multline*}
\phi : \frac{\pi_{j-1}^{-1}(\oM_{i-1}(\FF_{\fp})) + \fp M_{j-1}(B_{\fp})}{\fp M_{j-1}(B_{\fp})} \\
\to
\frac{\pi_{j-1}^{-1}(\oM_{i}(\FF_{\fp})) + \fp M_{j-1}(B_{\fp})}{\fp M_{j-1}(B_{\fp})} \times_{\cV_{i,j}}
\frac{\pi_{j}^{-1}(\oM_{i-1}(\FF_{\fp})) + \fp M_{j}(B_{\fp})}{\fp M_{j}(B_{\fp})}, 
\end{multline*}
where the right-hand term consists of pairs $(\alpha,\beta)$ for which $\rho_{i,j}(\alpha) = \delta_{i,j}(\beta)$, and for $x \in \pi_{j-1}^{-1}(\oM_{i-1}(\FF_{\fp}))$,
\[
\phi \bigl( x + \fp M_{j-1}(B_{\fp}) \bigr) =
\bigl( x + \fp M_{j-1}(B_{\fp}), x + \fp M_j(B_{\fp}) \bigr).
\]
We see easily that $\phi$ is injective. To see that $\phi$ is surjective, suppose that $y \in \pi_{j-1}^{-1}(\oM_{i}(\FF_{\fp}))$ and $z \in \pi_{j}^{-1}(\oM_{i-1}(\FF_{\fp}))$ satisfy $\rho_{i,j} ( y + \fp M_{j-1}(B_{\fp})) = \delta_{i,j} ( z + \fp M_{j}(B_{\fp}))$.
That is,
\[
y + \fp M_j(B_{\fp}) = z + \fp M_j(B_{\fp}).
\]
But then $\pi_j(z) = \pi_j(y) = \pi_{j-1}(y)$. Since $\pi_j(z) \in \oM_{i-1}(\FF_{\fp})$ and $\pi_{j-1}(y) \in \oM_i(\FF_{\fp})$, we see that $\pi_{j-1}(y) \in \oM_{i-1}(\FF_{\fp})$, and so $y \in \pi_{j-1}^{-1}(\oM_{i-1}(\FF_{\fp}))$. Furthermore,
\[
\phi \bigl( y + \fp M_{j-1}(B_{\fp}) \bigr) = \bigl( y + \fp M_{j-1}(B_{\fp}), y + \fp M_j(B_{\fp}) \bigr) = \bigl( y + \fp M_{j-1}(B_{\fp}), z + \fp M_j(B_{\fp}) \bigr),
\]
and thus $\phi$ is surjective.

Therefore, the $\FF_{\fp}$-vector spaces $\{ \cV_{i,j} \mid 0 \leqslant i \leqslant j \leqslant N \}$ together with corresponding maps $\delta_{i,j}$, $\rho_{i,j}$, satisfy the hypotheses of Lemma~\ref{L:triangle-v2}. Furthermore, the isomorphisms
\[
\cV_{i,j} \iso \oM_i(\FF_{\fp}) \cap \oP_{j},
\]
induced by \eqref{E:Mijseq}--\eqref{E:2ndiso}, are compatible so that
\begin{equation} \label{E:alignment}
\bigl\{ \cV_{i,j} \mid 0 \leqslant i \leqslant j \leqslant N \bigr\} \quad \stackrel{\sim}{\longleftrightarrow} \quad \bigl\{ \oM_{i} \cap \oP_j \mid 0 \leqslant i \leqslant j \leqslant N \bigr \}
\end{equation}
are isomorphic systems satisfying the hypotheses of Lemma~\ref{L:triangle-v2}. In this way, the $\FF_{\fp}$-bases $\oZ_{i,j}$ obtained from Lemma~\ref{L:triangle-v2} for $\{ \oM_i(\FF_{\fp}) \cap \oP_j \}$ can be transferred to $\{ \cV_{i,j} \}$ with the same properties (and vice versa). Furthermore, for $0 \leqslant i \leqslant N$, the inclusions $\cV_{0,i} \subseteq \cV_{1,i} \subseteq \cdots \subseteq \cV_{i,i}$ are
\[
\frac{\pi_i^{-1}(\oM_0(\FF_{\fp})) + \fp M_i(B_{\fp})}{\fp M_i(B_{\fp})} \subseteq \frac{\pi_i^{-1}(\oM_1(\FF_{\fp})) + \fp M_i(B_{\fp})}{\fp M_i(B_{\fp})} \subseteq \cdots \subseteq  \frac{M_i(B_{\fp})}{\fp M_i(B_{\fp})}.
\]
Applying Nakayama's lemma we find subsets $D_{k,i} \subseteq \pi_i^{-1}(\oM_{k}(\FF_{\fp}))$ for $0 \leqslant k \leqslant i$ so that
\begin{equation} \label{E:ZiMidef}
Z_{i} \assign D_{0,i} \sqcup \cdots \sqcup D_{i,i}
\end{equation}
is a $B_{\fp}$-basis of $M_i(B_{\fp})$ with $\pi(Z_i) = \oZ_{i,i}$. Since $\pi_i(\pi_i^{-1}(\oM_k(\FF_{\fp}))) = \oM_k(\FF_{\fp}) \cap \oP_i$, these can be arranged so that
\[
\pi_i(D_{k,i}) = \oD_{k,i}, \quad (0 \leqslant k \leqslant i \leqslant N).
\]

Before the final analysis, we need to investigate compatibility with the bases $\oT_i$ coming from $\oM_i(\FF_{\fp}) / (\oM_{i-1}(\FF_{\fp}) + \sigma \oM_{i-1}(\FF_{\fp}))$. To that end, we recall from Lemma~\ref{L:Miprops}(c) and Remark~\ref{R:Kdef} that, for $i > 0$, $M_{i-1} + \sigma M_{i-1}$ is defined over $K$, and moreover,
\[
M_{i-1} + \sigma M_{i-1} = \cL(-W_{\ba}^{(1)} + (i-1)I_{\ba}) + \cL(-W_{\ba}^{(1)} - \Xi_{\ba} + i I_{\ba}).
\]
We set
\begin{align*}
(M_{i-1} + \sigma M_{i-1})(B_{\fp}) &\assign (M_{i-1} + \sigma M_{i-1}) \cap B_{\fp}[t,z], \\
(M_{i-1} + \sigma M_{i-1})(K) &\assign (M_{i-1} + \sigma M_{i-1}) \cap K[t,z],
\end{align*}
and note that since $M_{i-1} + \sigma M_{i-1}$ is defined over $K$,
\[
(M_{i-1} + \sigma M_{i-1})(K) = K \otimes_{B_{\fp}} (M_{i-1} + \sigma M_{i-1})(B_{\fp}).
\]
We claim that the $B_{\fp}$-modules
\[
\frac{M_i(B_{\fp})}{M_{i-1}(B_{\fp})}, \quad
\frac{M_i(B_{\fp})}{(M_{i-1} + \sigma M_{i-1})(B_{\fp})},
\]
are torsion free. Let us consider the second module. Suppose that $h \in M_i(B_{\fp})$ and $u \in B_{\fp}$ with $uh \in (M_{i-1} + \sigma M_{i-1})(B_{\fp})$. Then $h \in (M_{i-1} + \sigma M_{i-1})(K)$, and so
\begin{multline*}
h \in (M_{i-1} + \sigma M_{i-1})(K) \cap M_i(B_{\fp}) \\
= (M_{i-1} + \sigma M_{i-1})(K) \cap M_i(K) \cap B_{\fp}[t,z] = (M_{i-1} + \sigma M_{i-1})(B_{\fp}).
\end{multline*}
The argument for the first module is similar. In particular, as $M_{i-1}(B_{\fp}) \subseteq (M_{i-1} + \sigma M_{i-1})(B_{\fp}) \subseteq M_i(B_{\fp})$ are all free $B_{\fp}$-modules of finite rank, we find that a $B_{\fp}$-basis for $M_{i-1}(B_{\fp})$ can be extended to bases of both $(M_{i-1} + \sigma M_{i-1})(B_{\fp})$ and $M_i(B_{\fp})$.

We can construct the sets $\{ S_i \}$, $\{ T_i \}$, in \S\ref{SSS:UVTSdefs} by using $B_{\fp}$-bases of $M_i(B_{\fp})/(M_{i-1} + \sigma M_{i-1})(B_{\fp})$. Moreover, we have a filtration,
\[
\frac{\pi_i^{-1}(\oM_k(\FF_{\fp})) + (M_{i-1} + \sigma M_{i-1})(B_{\fp})}{(M_{i-1} + \sigma M_{i-1})(B_{\fp})} \subseteq \frac{M_i(B_{\fp})}{(M_{i-1} + \sigma M_{i-1})(B_{\fp})}, \quad (0 \leqslant k \leqslant i),
\]
of free $B_{\fp}$-modules of finite rank. For $0 \leqslant k \leqslant i$, since $\pi_i(\pi_i^{-1}(\oM_k(\FF_{\fp}))) = \oM_k(\FF_{\fp}) \cap \oP_i$, it follows from Nakayama's lemma that
\begin{multline*}
\rank_{B_{\fp}} \biggl( \frac{\pi_i^{-1}(\oM_k(\FF_{\fp})) + (M_{i-1} + \sigma M_{i-1})(B_{\fp})}{(M_{i-1} + \sigma M_{i-1})(B_{\fp})} \biggr) \\
\geqslant \dim_{\FF_{\fp}} \biggl( \frac{(\oM_k(\FF_{\fp}) \cap \oP_{i}) + (\oM_{i-1}(\FF_{\fp}) + \sigma \oM_{i-1}(\FF_{\fp}))}{\oM_{i-1}(\FF_{\fp}) + \sigma \oM_{i-1}(\FF_{\fp})} \biggr).
\end{multline*}
Therefore, after perhaps rearranging our $B_{\fp}$-basis $Z_{i}$ of $M_i(B_{\fp})$ from~\eqref{E:ZiMidef}, we can pick a subset $T_{k,i} \subseteq Z_{i}$ which induces a $B_{\fp}$-basis of $(\pi_i^{-1}(\oM_k(\FF_{\fp})) + (M_{i-1} + \sigma M_{i-1})(B_{\fp})) / (M_{i-1}$ $+ \sigma M_{i-1})(B_{\fp}))$, such that
\begin{equation} \label{E:Tki}
\pi_i(T_{k,i}) \supseteq \oT_{k,i}.
\end{equation}
Moreover, we can take $T_i = T_{0,i} \sqcup \dots \sqcup T_{i,i}$ to be our choice of subset of $M_i$ that induces a $\ok$-basis of $M_{i}/(M_{i-1} + \sigma M_{i-1})$ as in~\S\ref{SSS:UVTSdefs}.

As noted, $\oS_N$ is an $\FF_{\fp}[\sigma]$-basis of $H_{\ba,\fp}(\FF_{\fp})$. Proposition~\ref{P:HaKdef} and Remark~\ref{R:algorithm} imply
\[
S_N \assign \bigsqcup_{0 \leqslant i \leqslant N} T_i = \bigsqcup_{\substack{0 \leqslant i \leqslant N \\ 0 \leqslant k \leqslant i}} T_{k,i} = \bigsqcup_{0 \leqslant k \leqslant i  \leqslant N} T_{k,i}
\]
is a $\ok[\sigma]$-basis of $H_{\ba}(\ok)$. But then by~\eqref{E:Haprankdim},
\[
d = |S_N| = \sum_{i=0}^{N} \sum_{k=0}^i |T_{k,i}| \geqslant \sum_{i=0}^{N} \sum_{k=0}^i |\oT_{k,i}| = |\oS_N| = d.
\]
Therefore, this inequality is an equality, and the containments in~\eqref{E:Tki} are also equalities. Furthermore, $S_N$ satisfies $\pi(S_N) = \oS_N$, and we are done.
\end{proof}

The following corollary is then a consequence of Propositions~\ref{P:Eagoodred} and~\ref{P:Hagoodred}.

\begin{corollary} \label{C:Eagoodprimes}
Let $\ba \in \cA_f$ be effective with $\deg \ba > 0$, and let $\fp$ be a maximal ideal of $B$ relatively prime to $f$. Then $E_{\ba}$ has good reduction at~$\fp$.
\end{corollary}


\end{document}